\documentclass[11pt]{amsart}
\pdfoutput=1
\usepackage{amssymb, amsfonts, latexsym, amsthm, amsmath, verbatim, enumerate}
\usepackage{soul} 
\usepackage{paralist}
\usepackage{graphicx}
\usepackage[all]{xy}
\usepackage{amscd}
\usepackage{mathabx}
\usepackage[english]{babel}
\usepackage[margin=1in]{geometry}
\usepackage{tikz-cd}
\usepackage{hyperref}
\usepackage{graphicx}
\usepackage{epstopdf}
\usetikzlibrary{arrows}
\usepackage{float}
\floatstyle{boxed}
\restylefloat{figure}
\theoremstyle{plain}
\newtheorem{thm}{Theorem}[section]
\newtheorem*{thm*}{Theorem}
\newtheorem{prop}[thm]{Proposition}
\newtheorem{cor}[thm]{Corollary}
\newtheorem{lem}[thm]{Lemma}
\newtheorem{rmk}[thm]{Remark}
\newtheorem{examp}[thm]{Example}
\newtheorem{fact}[thm]{Fact}
\theoremstyle{definition}
\newtheorem{defn}[thm]{Definition}

\newcommand{\G}{\mathit{\mathcal{G}}}

\newcommand{\f}{\mathit{\mathbb{F}}}

\newcommand{\p}{\mathit{\mathrm{pt}}}
\newcommand{\tH}{\mathit{\tilde{H}}}
\newcommand{\ce}{\mathfrak{CE}}
\newcommand{\E}{\mathfrak{E}}

\newcommand{\CEL}{\mathfrak{CLE}}

\newcommand{\LE}{\mathfrak{LE}}

\newcommand{\bt}{\mathcal{T}^+}

\newcommand{\s}{\mathit{\mathfrak{s}}}

\newcommand{\hgc}{\mathbb{H}(\Gamma,\chi)}
\newcommand{\sucp}{\mathit{\mathrm{suc}_+}}
\newcommand{\sucm}{\mathit{\mathrm{suc}_-}}
\newcommand{\suc}{\mathit{\mathrm{suc}}}

\newcommand{\sym}{\mathit{\mathrm{sym}}}
\newcommand{\gcp}{\mathrm{\Gamma}_{\mathnormal{C_p}}}
\newcommand{\dcp}{\Delta_{C_p}}
\newcommand{\dyp}{\Delta_{Y_p}}
\newcommand{\tdyp}{\tilde{\Delta}_{Y_p}}
\newcommand{\ny}{\mathit{N_{Y_p}}}
\newcommand{\tx}{\mathit{\tilde{X}_{Y_p}}}
\newcommand{\dz}{\Delta_{Z_{p}}}
\newcommand{\dw}{\Delta_{W_{p}}}
\newcommand{\lowerromannumeral}[1]{\romannumeral#1\relax}

\newcommand{\syp}{\mathit{\tilde{S}_{Y_p}}}
\title{Manolescu Invariants of Connected Sums}
\author{Matthew Stoffregen}
\begin{document}
\begin{abstract}
We give inequalities for the Manolescu invariants $\alpha,\beta,\gamma$ under the connected sum operation.  We compute the Manolescu invariants of connected sums of some Seifert fiber spaces.  Using these same invariants, we provide a proof of Furuta's Theorem, the existence of a $\mathbb{Z}^\infty$ subgroup of the homology cobordism group.  To our knowledge, this is the first proof of Furuta's Theorem using monopoles.  We also provide information about Manolescu invariants of the connected sum of $n$ copies of a three-manifold $Y$, for large $n$.
\end{abstract}
\maketitle
\section{Introduction} \label{sec:intro}   
\subsection{Manolescu Invariants}\label{subsec:manoinvars}  Let $G=\mathrm{Pin}(2)$ be the group consisting of two copies of the complex unit circle, along with a map $j$ interchanging them, so that $ij=-ji$ and $j^2=-1$.  In \cite{ManolescuPin}, Manolescu associated to a spin three-manifold $(Y,\s)$ the $G$-equivariant Seiberg-Witten Floer homology $\mathit{SWFH}^G(Y,\mathfrak{s})$, the Borel homology of a stable homotopy type $\mathit{SWF}(Y,\s)$.  From its module structure, he defined homology cobordism invariants $\alpha, \beta, \gamma$ as analogues of the Fr{\o}yshov invariant of the usual, $S^1$-equivariant, Seiberg-Witten Floer homology \cite{Froyshov4mflds}.  The $\mathrm{Pin}(2)$-invariants all reduce to the Rokhlin invariant mod $2$, satisfy $\alpha(Y)\geq\beta(Y)\geq \gamma(Y)$, and furthermore $\beta(-Y)=-\beta(Y)$.  

The existence of such an invariant $\beta$ implies that there is no integer homology sphere $Y$ with odd Rokhlin invariant so that $Y \# Y \sim S^3$, where $\sim$ denotes homology cobordism.  Using work of Matumoto \cite{Matumoto} and Galewski-Stern \cite{GalewskiStern}, the nonexistence of such $Y$ disproves the triangulation conjecture.  That is, Manolescu \cite{ManolescuPin} shows there exist manifolds of all dimensions $n\geq 5$ that cannot be triangulated.

Since the introduction of $\mathit{SWFH}^G(Y,\s)$, other versions of Floer homologies with symmetries beyond the $S^1$-symmetry have become available.  Lin \cite{flin} constructed a $G$-equivariant refinement of monopole Floer homology in the setting of Kronheimer-Mrowka \cite{KM}.  In the setting of Heegaard Floer homology introduced by Ozsv{\'a}th-Szab{\'o} in \cite{OzSz1}, \cite{OzSz2}, Hendricks and Manolescu \cite{HendricksManolescu} defined an analogue, $\mathit{HFI}(Y,\s)$, of $\mathbb{Z}/4$-equivariant Seiberg Witten Floer homology (where we regard $\langle j \rangle =\mathbb{Z}/4\subset G$).  

In addition to the disproof of the triangulation conjecture, these theories have other applications inaccessible to $S^1$-Floer theory.  For instance, $\mathit{HFI}$ detects the non-sliceness of the figure-eight knot, while most concordance invariants coming from Floer theory and Khovanov homology fail to detect this.  Both Lin's theory \cite{flin2} and $\mathit{SWFH}^G(Y,\s)$ \cite{betaseifert} show the existence of three-manifolds not homology cobordant to any Seifert space.

Under orientation reversal, the numerical homology cobordism invariants coming from these theories are well-behaved.  For instance, the invariants $\alpha,\beta,$ and $\gamma$, and another invariant $\delta$ defined from the $S^1$ Borel homology of $\mathit{SWF}(Y,\s)$ satisfy:
\begin{equation}\label{eq:orierev}\alpha(-Y,\mathfrak{s})=-\gamma(Y,\mathfrak{s}),\; \beta(-Y,\mathfrak{s})=-\beta(Y,\mathfrak{s}),\; \gamma(-Y,\mathfrak{s})=-\alpha(Y,\mathfrak{s}),\; \delta(-Y,\s)=-\delta(Y,\s).
\end{equation}   

However, the numerical homology cobordism invariants $\alpha,\beta, \gamma$ are not homomorphisms $\theta^H_3 \rightarrow \mathbb{Z}$, unlike, for instance, the Fr{\o}yshov invariant $\delta$.  This makes them harder to compute, but does not mean they are blind to the group structure of $\theta^H_3$.  For instance, see Theorem \ref{thm:furuta} below.  

In the present paper we investigate the behavior of the Manolescu invariants, $\alpha$, $\beta$, and $\gamma$, under the connected sum operation.  In particular, we have the following theorems:

\begin{thm}\label{thm:standineq}
Let $(Y_1,\s_1),(Y_2,\s_2)$ be rational homology three-spheres with spin structure.  Then:
\begin{align}\label{eq:alph}
\alpha(Y_1,\s_1)+\gamma(Y_2,\s_2) &\leq \alpha(Y_1 \# Y_2,\s_1 \# \s_2) \leq \alpha(Y_1,\s_1)+\alpha(Y_2,\s_2),
\\  \label{eq:alphadual}
\gamma(Y_1,\s_1)+\gamma(Y_2,\s_2) &\leq \gamma(Y_1 \# Y_2,\s_1 \# \s_2) \leq \alpha(Y_1,\s_1) + \gamma(Y_2,\s_2), \\ \label{eq:beta2}
\gamma(Y_1,\s_1)+\beta(Y_2,\s_2) &\leq  \beta(Y_1 \# Y_2,\s_1 \# \s_2) \leq \alpha(Y_1,\s_1)+\beta(Y_2,\s_2),
\end{align}\begin{equation}\label{eq:beta}
\gamma(Y_1 \# Y_2,\s_1\# \s_2) \leq \beta(Y_1,\s_1) + \beta(Y_2,\s_2) \leq \alpha(Y_1 \# Y_2,\s_1 \# \s_2).
\end{equation}
\end{thm}

\begin{thm}\label{thm:ineq2}
Let $(Y,\s)$ be a rational homology three-sphere with spin structure.  Then:
\begin{equation}\label{eq:ineq2}
\gamma(Y,\s) \leq \delta(Y,\s) \leq \alpha(Y,\s).
\end{equation}
\end{thm}
We note, for comparison with Heegaard Floer theory, that the invariant $\delta(Y,\s)$ should correspond to the Heegaard Floer correction term $d(Y,\s)/2$.

If we regard Theorem \ref{thm:ineq2} as a statement constraining the behavior of $\delta(Y,\s)$ in terms of the Manolescu invariants $\alpha,\beta$, and $\gamma$, then we may think of the following as a kind of converse statement, showing that $\delta(Y,\s)$ heavily constrains the behavior of the Manolescu invariants:
\begin{thm}\label{thm:asymp}
Let $(Y,\s)$ be a rational homology three-sphere with spin structure.  Then:
\begin{equation}\label{eq:preasymp}
\alpha(\#_n(Y,\s))-n\delta(Y,\s), \; \beta(\#_n(Y,\s))-n\delta(Y,\s), \; \mathrm{and}\; \gamma(\#_n(Y,\s))-n\delta(Y,\s)
\end{equation}
are bounded functions of $n$, where $\#_n (Y,\s)$ denotes the connected sum of $n$ copies of $(Y,\s)$.  In particular:
\begin{equation}\label{eq:asymp}
\lim_{n\rightarrow \infty} \frac{\alpha(\#_n(Y,\s))}{n}=\lim_{n\rightarrow \infty} \frac{\beta(\#_n(Y,\s))}{n}=\lim_{n\rightarrow \infty} \frac{\gamma(\#_n(Y,\s))}{n} = \delta(Y,\s).
\end{equation}
\end{thm} 
That is, one might think of the Manolescu invariants as perturbations of the $S^1$-Fr{\o}yshov invariant.  

In order to obtain Theorem \ref{thm:asymp}, we will make an explicit calculation of the Manolescu invariants of connected sums of negative Seifert spaces \emph{of projective type}.  We call a Seifert rational homology three-sphere $Y$ \emph{negative} if the orbifold line bundle of $Y$ is of negative degree (see Section \ref{sec:gauge}).  For example, the Brieskorn sphere $\Sigma(a_1,...,a_n)$, for coprime $a_i$, is negative.  We call a negative Seifert rational homology three-sphere with spin structure $(Y,\s)$ \emph{of projective type} if, for some constants $d,n,m,a_i,m_i$ and some index set $I$, its Heegaard Floer homology is of the form
\begin{equation}\label{eq:1proj}
\mathit{HF^{+}}(Y,\mathfrak{s})=\mathcal{T}^+_{d}\oplus \mathcal{T}^{+}_{-2n+1}(m) \oplus \bigoplus_{i \in I} \mathcal{T}^{+}_{a_i}(m_i)^{\oplus 2}.
\end{equation}
We will be interested in negative Seifert spaces because their $G$-equivariant Seiberg-Witten Floer stable homotopy types, $\mathit{SWF}(Y,\s)$, admit an especially simple description: $\mathit{SWF}(Y,\s)$ is \emph{$j$-split} for $Y$ negative Seifert.  We call a $G$-space $X$ a \emph{$j$-split space} if
\[X/X^{S^1}=X_+ \vee X_-,\]
where $X^{S^1}$ is the $S^1$-fixed point set of $X$, and $j$ interchanges the $X_+$ and $X_-$ factors.  We may think of $j$-split spaces as the simplest kind of (nontrivial) $G$-spaces which may occur as the Seiberg-Witten Floer spectrum $\mathit{SWF}(Y,\s)$ of some $(Y,\s)$.  

  The \emph{projective type} condition further restricts what $X_+$ may be.  Examples of Seifert spaces of projective type include Brieskorn spheres of the form $\Sigma(p,q,pqn\pm 1)$, for relatively prime $p$ and $q$, from work of N{\'e}methi and Borodzik \cite{Nem07},\cite{BN} and Tweedy \cite{Tweedy}.  
\begin{thm}\label{thm:main}
Let $Y_1,\dots,Y_n$ be negative Seifert integral homology three-spheres of projective type.  Define $\tilde{\delta}(Z)=d(Z)/2+\bar{\mu}(Z)$, for $Z$ any Seifert fiber space, where $d$ is the Heegaard Floer correction term from \cite{OzSzgrad}, and where $\bar{\mu}$ is the Neumann-Siebenmann invariant defined in \cite{NeumannPlumbings}, \cite{Siebenmann}.  Set $\tilde{\delta}_i:=\tilde{\delta}(Y_i)$, and assume without loss of generality $\tilde{\delta}_1 \leq \dots \leq \tilde{\delta}_n$.  Then:
\begin{align} \label{eq:aadd}
\alpha(Y_1 \# \dots \# Y_{n}) &= 2\left\lfloor\frac{(\sum_{i=1}^n \tilde{\delta}_i)+1}{2}\right\rfloor-\sum_{i=1}^n\bar{\mu}(Y_i),
\\ \label{eq:badd}
\beta(Y_1 \# \dots \# Y_{n}) &= 2\left\lfloor\frac{(\sum_{i=1}^{n-1}\tilde{\delta}_i)+1}{2}\right\rfloor-\sum_{i=1}^n\bar{\mu}(Y_i),
\\ \label{eq:cadd}
\gamma(Y_1 \# \dots \# Y_{n}) &= 2\left\lfloor\frac{(\sum_{i=1}^{n-2}\tilde{\delta}_i)+1}{2}\right\rfloor-\sum_{i=1}^n\bar{\mu}(Y_i),
\end{align} and \begin{equation}\label{eq:dadd}
\delta(Y_1 \# \dots \# Y_{n}) = (d(Y_1)+\dots+d(Y_{n}))/2 = \sum_{i=1}^n \tilde{\delta}_i-\sum_{i=1}^n \bar{\mu}(Y_i).
\end{equation} 
\end{thm}
To prove Theorem \ref{thm:main} we will investigate the $\mathrm{Pin}(2)$-equivariant topology of joins of $j$-split spaces.  To do so, we will make use of the Gysin sequence for $\mathrm{Pin}(2)$-spaces, which provides a relationship between the $\mathrm{Pin}(2)$-equivariant and $S^1$-equivariant homology of a $\mathrm{Pin}(2)$-space.  Lin has already used the Gysin sequence in \cite{flin2} to study $\widecheck{HS}(Y,\s)$ for $Y$ surgery on an alternating knot.  

The proof of Theorem \ref{thm:main} also relies on the equivalence of several versions of Floer homologies: we employ the equivalence of $\mathit{HF^+}$ and $\widecheck{HM}$ from Colin-Ghiggini-Honda \cite{CGH1} and Taubes \cite{Taubes}, and Kutluhan-Lee-Taubes \cite{KLT1}, and the equivalence of $\widecheck{HM}$ and $\mathit{SWFH^{S^1}}$ due to Lidman-Manolescu \cite{LM}. 

To obtain Theorem \ref{thm:asymp} from Theorem \ref{thm:main} we will use the machinery of \emph{chain local equivalence} from \cite{betaseifert}, a refinement of the Manolescu invariants.  For a rational homology three-sphere $(Y,\s)$, the chain local equivalence class $[\mathit{SWF}(Y,\s)]_{cl}$ is a homology cobordism invariant taking values in the set $\mathfrak{CLE}$ of equivariant chain complexes up to the equivalence relation of chain local equivalence.  Chain local equivalence keeps track of how $G$-equivariant CW cells are attached to the reducible critical point in $\mathit{SWF}(Y,\s)$ (see Sections \ref{sec:2} and \ref{sec:gauge}).  One advantage of $[\mathit{SWF}(Y,\s)]_{cl}$ is that it behaves well under connected sum, while the behavior of the Manolescu invariants is more complicated.  

More specifically, to obtain Theorem \ref{thm:asymp}, we will show that any CW chain complex associated to a $\mathrm{Pin}(2)$-space admits some ``large" $j$-split subcomplex (partly controlled by the $\delta$ invariant).  Here, we call a $\mathrm{Pin}(2)$-chain complex \emph{$j$-split} if it is the CW chain complex of a $j$-split space.  Using the ``large" $j$-split subcomplex inside a given $\mathrm{Pin}(2)$-complex, the calculation of Theorem \ref{thm:main} may be carried over, in part, to arbitrary rational homology three-spheres, yielding Theorem \ref{thm:asymp}.  \\

\subsection{Applications}\label{subsec:introapps}
We apply Theorem \ref{thm:main} to study homology cobordisms among Seifert spaces.  A homology cobordism between two closed, oriented three-manifolds $Y_1,Y_2$ is a smooth compact oriented four-manifold $X$ such that $\partial X=Y_1 \amalg -Y_2$, and such that the inclusions $Y_i \hookrightarrow X$ induce isomorphisms $H_*(Y_i,\mathbb{Z}) \rightarrow H_*(X;\mathbb{Z})$.  Homology cobordism is an equivalence relation on the set of closed three-manifolds.  We write $Y_1 \sim Y_2$ if $Y_1$ and $Y_2$ are homology cobordant three-manifolds.  We form the integral homology cobordism group $\theta^H_3$, consisting of homology cobordism classes of integral homology three-spheres, with addition given by connected sum.  

A corollary of Theorem \ref{thm:main} is:

\begin{thm}\label{thm:nonSeifert}
Let $Y_{1},\dots,Y_{n}$ be negative Seifert integral homology spheres of projective type, with at least two of the $Y_i$  having $\frac{d(Y_i)}{2} \geq -\bar{\mu}(Y_i)+2$.  Then $Y:=Y_1 \# \dots \# Y_n$ is not homology cobordant to any Seifert fiber space.  
\end{thm}

We say that an integral homology three-sphere $Y$ is \emph{$H$-split} if $\alpha(Y)=\beta(Y)=\gamma(Y)$.  Theorem \ref{thm:standineq} implies that the set $\theta_{H\text{--}\mathrm{split}} \subset \theta^H_3$ of $H$-split integral homology three-spheres is, in fact, a subgroup.  We obtain from Theorem \ref{thm:main}:

\begin{thm}\label{thm:furuta}
Let $\theta_{SFP}$ be the subgroup of $\theta^H_3$ generated by negative Seifert spaces of projective type, and let $\theta_{H\text{--}\mathrm{split},SFP}$ be the subgroup consisting of $Y \in \theta_{SFP}$ such that $\alpha(Y)=\beta(Y)=\gamma(Y)$.  Then:
\begin{equation}\label{eq:Seifertsplitting} \theta_{SFP}=\theta_{H\text{--}\,\mathrm{split},SFP}\oplus \mathbb{Z}^\infty. \end{equation}
The $\mathbb{Z}^\infty$ summand is generated by $\{Y_p=\Sigma(p,2p-1,2p+1)\mid 3\leq p, p \;\mathrm{odd}\}$.  In particular, the elements $\{Y_p\mid 3\leq p, p \; \mathrm{odd}\}$ are linearly independent in $\theta^H_3$.  
\end{thm}

This implies the existence of a $\mathbb{Z}^\infty$ subgroup of $\theta^H_3$, a result originally due to Furuta \cite{Furuta} and Fintushel-Stern \cite{FSZinfty}, both building on the $R$-invariant introduced by Fintushel and Stern \cite{FintushelStern85} using instantons.  Fintushel and Stern \cite{FSZinfty} show that the collection $\{\Sigma(p,q,pqn-1)\mid n\geq 1\}$ is linearly independent in $\theta^H_3$ for any relatively prime $p,q$, and Furuta's construction of $\mathbb{Z}^\infty \subseteq \theta^H_3$ is the special case $p=2,q=3$ of Fintushel and Stern's construction.   However, we will see from Theorem \ref{thm:minisum} that the image of $\{ \Sigma(p,q,pqn-1)\mid n \geq 1\}$ in $\theta^H_3$ is contained in $\theta_{H\text{--}\mathrm{split}}\oplus \mathbb{Z}$, for any fixed $p,q$.  In particular, the $\mathbb{Z}^\infty$ subgroups that Furuta and Fintushel-Stern originally identified are not detected by $\mathrm{Pin}(2)$-techniques.  We then obtain the following corollary:  

\begin{cor}\label{cor:hsplitcor} 
The subgroup $\theta_{H\text{--}\mathrm{split}}\subset \theta^H_3$ is infinitely-generated.  
\end{cor}  

To our knowledge, Theorem \ref{thm:furuta} is the first proof of the existence of a $\mathbb{Z}^\infty$ subgroup of $\theta^H_3$ using either monopoles or the technology of Heegaard Floer homology.  The Fintushel-Stern $R$ invariant also shows that $Y_p$, for $p$ odd, are linearly independent in the homology cobordism group \cite{Endo}, but it does not show the splitting as in (\ref{eq:Seifertsplitting}).  

Theorem \ref{thm:furuta} follows from Theorem \ref{thm:main}, once one finds a collection of Seifert integral homology spheres $Y$ of projective type with $d(Y)/2+\bar{\mu}(Y)$ arbitrarily large:
\begin{thm}\label{thm:submain}
Let $Y_p=\Sigma(p,2p-1,2p+1)$.  For odd $p\geq 3$, $Y_p$ is of projective type and $d(Y_p)/2+\bar{\mu}(Y_p)=\frac{p-1}{2}$.  
\end{thm} 
Theorem \ref{thm:submain} is proved using the technology of graded roots, introduced by N{\'e}methi \cite{Nemethigr}, and refinements of the method of graded roots for Seifert spaces in \cite{CanKarakurt},\cite{KarakurtLidman}.  The proof is essentially borrowed from the partial calculation of $\mathit{HF}^+(Y_p)$ for even $p$ by Hom, Karakurt, and Lidman \cite{HLK}.  

Other convenient choices of the generating set for $\mathbb{Z}^\infty$ in Theorem \ref{thm:furuta} are possible, such as, for example, $\{\Sigma(2,q,2q+1)\mid q \equiv 3\bmod{4}\}$.  See Theorem \ref{thm:minisum} for a more precise statement.

Using Theorem \ref{thm:furuta}, we may also obtain statements about knots.  Endo showed in \cite{Endo} that the smooth concordance group of topologically slice knots, denoted $\mathcal{C}_{TS}$, contains a $\mathbb{Z}^\infty$ subgroup, using the Fintushel-Stern $R$-invariant.  Using Theorem \ref{thm:furuta}, we are able to reproduce Endo's result: 

\begin{cor}\label{cor:knot}
The pretzel knots $K(-p,2p-1,2p+1)$, for odd $p\geq 3$, are linearly independent in $\mathcal{C}_{TS}$.
\end{cor}
\begin{proof}
We chose the Seifert spaces $Y_p$ in Theorem \ref{thm:furuta} instead of other possible generating sets because $Y_p$ are branched double covers of pretzel knots:
\[ Y_p=\Sigma(K(-p,2p-1,2p+1))\]
where $K(-p,2p-1,2p+1)$ is the pretzel knot of type $(-p,2p-1,2p+1)$.  We note that the Alexander polynomial
\[ \Delta_K(K(-p,2p-1,2p+1))=1\]
for all odd $p$.  Thus, by \cite{Freedman}, $K(-p,2p-1,2p+1)$ are topologically slice.  By Theorem \ref{thm:furuta}, the present Corollary follows.
\end{proof}
The subgroup that Endo identifies in $\mathcal{C}_{TS}$ is identical to that of Corollary \ref{cor:knot}.   Hom \cite{Hom} much extended Endo's result, showing that $\mathcal{C}_{TS}$ has a $\mathbb{Z}^\infty$ summand, using the knot concordance invariant $\epsilon$ defined in \cite{Hom12}.  Additionally, Ozsv{\'a}th, Stipsicz, and Szab{\'o} \cite{OSS} gave another proof that $\mathcal{C}_{TS}$ has a $\mathbb{Z}^\infty$ summand using the knot concordance invariant $\Upsilon$.  

Another natural question is whether the Manolescu invariants of a pair of three-manifolds determine the Manolescu invariants of the connected sum.  This is not the case, as may be seen using Theorem \ref{thm:main}.  We take $Y=\Sigma(2,3,7)$, noting
\begin{equation}
\alpha(Y)=1,\; \beta(Y)=-1,\; \gamma(Y)=-1,\; \delta(Y)=0,\; \mathrm{and}\; \bar{\mu}(Y)=1.
\end{equation}
Then we have $\tilde{\delta}(Y)=1$, and, by Theorem \ref{thm:main}, the Manolescu invariants of $2(n+1)Y$ and $(2n+1)Y$ are independent of $n\geq 0$.  Specifically, 
\[\alpha(2(n+1)Y)=0,\;\beta(2(n+1)Y)=0,\; \gamma(2(n+1)Y)=-2, \]
\[ \alpha((2n+1)Y)=1,\;\beta((2n+1)Y)=-1,\; \gamma((2n+1)Y)=-1.\]
Then the Manolescu invariants of $2nY$ and $2mY$ agree for $n >m \geq 1$.  However, 
\[\alpha(2nY \# -2nY)=\beta(2nY\#-2nY)=\gamma(2nY \# -2nY)=0,\]
while
\[\alpha(2nY \# -2mY)=\beta(2nY\#-2mY)=0, \; \gamma(2nY\#-2mY)=-2.\]
Thus, the Manolescu invariants of $Y_1$ and $Y_2$ do not determine those of the connected sum $Y_1 \# Y_2$.

\subsection{Organization} The structure of the paper is as follows.  In Section \ref{sec:2} we review the definitions of the Manolescu invariants and local equivalence.  In Section \ref{sec:ineqs}, we provide the algebraic part of the proof of Theorems \ref{thm:standineq} and \ref{thm:ineq2}.  In Section \ref{sec:smash} we provide the algebraic-topological part of the proof of Theorem \ref{thm:main}.  In Section \ref{sec:gauge} we state the output of the constructions of \cite{ManolescuPin}, and use the results of Sections \ref{sec:ineqs} and \ref{sec:smash} to obtain Theorems \ref{thm:standineq}, \ref{thm:ineq2}, and \ref{thm:main}.  In Section \ref{sec:app} we list applications, providing proofs of Theorems \ref{thm:asymp}, \ref{thm:nonSeifert} and \ref{thm:furuta}.  In Section \ref{sec:gr} we give a brief introduction to graded roots, and in Section \ref{sec:semicre} we prove Theorem \ref{thm:submain}.  Throughout the paper all homology will be taken with $\mathbb{F}:=\mathbb{Z}/2$ coefficients, unless otherwise specified.

\section*{Acknowledgments}
The author would like to thank Ciprian Manolescu for his guidance and advice.  The author is also grateful to Cagri Karakurt, Tye Lidman, and Francesco Lin for helpful conversations.  

\section{Spaces of type SWF}\label{sec:2}
 
\subsection{$G$-CW Complexes}
	In this section we recall the definition of spaces of type SWF from \cite{ManolescuPin}, and introduce local equivalence.  Spaces of type SWF are the output of the construction of the Seiberg-Witten Floer stable homotopy type of \cite{ManolescuPin} and \cite{ManolescuK}; see Section \ref{sec:gauge}.  The present section follows Section 2 of \cite{betaseifert} prior to \S \ref{subsec:order}, which is new.
	
We briefly review here the equivariant topology we will need, and refer to Section 2 of \cite{ManolescuPin} for further details.  

 A (finite) $K$-CW decomposition of a space $(X,A)$ with (left) $K$-action is a filtration $(X_n\mid n \in \mathbb{Z}_{\geq 0})$ of $X$ such that 
	\begin{itemize}
	\item $A \subset X_0$ and $X=X_n$ for $n$ sufficiently large.
	\item The space $X_n$ is obtained from $X_{n-1}$ by attaching $K$-equivariant $n$-cells, copies of $(K/H) \times D^n$, for $H$ a closed subgroup of $K$.  
	\end{itemize}
When $A$ is a point, we call $(X,A)$ a pointed $K$-CW complex.  

Let $X$ be a pointed, finite $K$-CW complex for a compact Lie group $K$.  Let $EK$ denote the classifying space of $K$, and write $EK_+$ for $EK$ with a disjoint base point.  Further, for pointed $K$-CW complexes $X_1,X_2$, we define the equivariant smash product by
\[X_1 \wedge_K X_2 =(X_1\wedge X_2)/K.\] 
The reduced Borel homology and cohomology of $X$ are defined by:
\begin{align}\label{eq:boreldef}
\tilde{H}^K_*(X)&= \tilde{H}_*(EK_+ \wedge_K X),\\ \nonumber
\tilde{H}^*_K(X)&= \tilde{H}^*(EK_+\wedge_K X).
\end{align}
Borel homology and cohomology are invariants of $K$-equivariant homotopy equivalence.

Furthermore, the map $\pi: EK_+ \times_K X \rightarrow EK_+/K=BK_+$ given by projection onto the first factor induces a map in cohomology $\pi^*: \tilde{H}^*(BK_+)=H^*(BK)\rightarrow \tilde{H}^*_K(X)$.  The map $\pi^*$ gives $\tilde{H}^*_K(X)$ and $\tilde{H}^K_*(X)$ the structure of $H^*(BK)$-modules.

Let $G=\mathrm{Pin}(2)$ and $BG$ its classifying space.  In addition to the previous definition of $G$, one may think of $G$ as the set $S^1 \cup jS^1 \subset \mathbb{H}$, where $S^1$ is the unit circle in the $\langle 1, i \rangle$ plane, with group action on $G$ induced from the group action of the unit quaternions.  Manolescu shows in \cite{ManolescuPin} that $H^*(BG)=\f[q,v]/(q^3)$, where $\mathrm{deg} \; q =1$ and $\mathrm{deg} \; v =4$, so $H^G_*(X)$ is naturally an $\f[q,v]/(q^3)$-module for $X$ a $G$-space.  Moreover $S^\infty=S(\mathbb{H}^\infty)$ with its quaternion action is a free $G$-space.  Since $S^\infty$ is contractible, we identify $EG=S^\infty$.  

We will also need to relate $G$-Borel homology and $S^1$-Borel homology.  Consider 
\[f: \mathbb{C}P^\infty = BS^1 \rightarrow BG, \]
the map given by quotienting by the action of $j \in G$ on $BS^1$.  Recall $H^*(BS^1)=\f[U]$ where $\mathrm{deg}\; U=2$.  Then we have the following fact (for a proof, see \cite[Example 2.11]{ManolescuPin}):  

\begin{fact}\label{fct:bs1bg} 
The natural map
\[\mathrm{res}^G_{S^1}:=f^*: \f[q,v]/(q^3)=H^*(BG) \rightarrow H^*(BS^1)=\f[U] \]
is an isomorphism in degrees divisible by $4$, and zero otherwise.  In particular, $v \rightarrow U^2$.  Similarly, 
\[f_*: H_*(BS^1)\rightarrow H_*(BG)\]
has $f_*(u^{-2n})=v^{-n}$ and $f_*(u^{-2n+1})=0$, where $u^{-n}$ is the unique nonzero element of $H_*(BS^1)$ in degree $2n$, and $v^{-n}$ is the unique nonzero element of $H_*(BG)$ in degree $4n$.  
\end{fact}

For $X$ a pointed $G$-CW complex and $V$ a finite-dimensional $G$-representation, we let $V^+$ denote the one-point compactification of $V$, and we call $\Sigma^VX=V^+\wedge X$ the suspension of $X$ by $V$.  We will need to use that Borel homology behaves well with respect to suspension:
\begin{prop}[\cite{ManolescuPin} Proposition 2.2]\label{prop:susp}
Let $V$ be a finite-dimensional representation of $G$.  Then, as $\f[q,v]/(q^3)$-modules:
\begin{align}\label{eq:suspensioninvar}
\tilde{H}^*_G(\Sigma^VX)\cong \tilde{H}^{*-\mathrm{dim}\, V}_G(X)\\ \nonumber
\tilde{H}^G_*(\Sigma^VX)\cong \tilde{H}^G_{*-\mathrm{dim}\, V}(X).
\end{align} 
\end{prop}

We mention three irreducible representations of $G$ which we will need:
\begin{itemize}
\item The trivial one-dimensional representation $\mathbb{R}$.
\item The one-dimensional real vector space on which $j\in G$ acts by $-1$, and on which $S^1$ acts trivially, denoted $\tilde{\mathbb{R}}$.
\item The quaternionic representation $\mathbb{H}$, where $G$ acts by left multiplication.
\end{itemize}
\begin{defn}\label{thm:swfdefn}
Let $s \in \mathbb{Z}$.  A \emph{space of type SWF} at level $s$ is a pointed, finite $G$-CW complex $X$ with
\begin{itemize}
\item The $S^1$-fixed-point set $X^{S^1}$ is $G$-homotopy equivalent to $(\tilde{\mathbb{R}}^s)^+$, the one-point compactification of $\tilde{\mathbb{R}}^s$.
\item The action of $G$ on $X-X^{S^1}$ is free.  
\end{itemize}
\end{defn}	

We note that for a space $X$ of type SWF, \[\tilde{H}^G_*(X^{S^1})=\tilde{H}^G_*(\tilde{\mathbb{R}}^s)=\tilde{H}^G_{*-s}(S^0)=H_{*-s}(BG),\] and
\[\tilde{H}^*_G(X^{S^1})=H^{*-s}(BG),\]
using Proposition \ref{prop:susp}.

Associated to a space $X$ of type SWF at level $s$, we take the Borel cohomology $\tilde{H}_G^*(X)$, from which Manolescu \cite{ManolescuPin} defines $a(X),b(X),$ and $c(X)$:
\begin{align}\label{eq:adef} 
a(X) &= \mathrm{min}\{ r \equiv s \;\mathrm{mod}\; 4 \mid \exists \, x \in \tilde{H}^r_G(X), v^lx \neq 0 \; \mathrm{ for\; all} \; l \geq 0 \}, \\ \nonumber
b(X) &= \mathrm{min}\{ r \equiv s+1 \;\mathrm{mod}\; 4 \mid \exists \, x \in \tilde{H}^r_G(X), v^lx \neq 0 \; \mathrm{for\; all}\; l \geq 0 \}-1,  
\\ \nonumber
c(X) &= \mathrm{min}\{ r \equiv s+2 \;\mathrm{mod}\; 4 \mid \exists \, x \in \tilde{H}^r_G(X), v^lx \neq 0 \;\mathrm{for\; all}\; l \geq 0 \}-2. 
\end{align}
The well-definedness of $a,b$, and $c$ follows from the Equivariant Localization Theorem.  We list a version of this theorem for spaces of type SWF:

\begin{thm}[\cite{tomDieck} \S III (3.8)]\label{thm:equivlocalzn}
Let $X$ be a space of type SWF.  Then the inclusion $X^{S^1}\rightarrow X$, after inverting $v$, induces an isomorphism of $\f[q,v,v^{-1}]/(q^3)$-modules:
\[v^{-1}\tilde{H}^*_G(X^{S^1})\cong v^{-1}\tilde{H}^*_G(X).\]
\end{thm}
For $X$ a space of type SWF, $X$ is a finite $G$-complex and so we have that $\tilde{H}^*_G(X)$ is finitely generated as an $\f[v]$-module.  In particular, the $\f[v]$-torsion part of $\tilde{H}^*_G(X)$ is bounded above in grading.  Theorem \ref{thm:equivlocalzn} then implies:
\begin{fact}\label{fct:highgrads}
Let $X$ be a space of type SWF.  Then the inclusion $\iota:X^{S^1} \rightarrow X$ induces an isomorphism \[\iota^*:\tilde{H}^*_G(X)\rightarrow \tilde{H}^*_G(X^{S^1})\]in cohomology in sufficiently high degrees.  Dualizing, $\iota_*$ induces an isomorphism in homology in sufficiently high degrees
\end{fact}
We note that Fact \ref{fct:highgrads} implies 
\begin{equation}\label{eq:iotaim}\mathrm{Im}\, \iota_*=\{x\in \tH^G_*(X) \mid x \in \mathrm{Im} \, v^l \; \mathrm{for\; all} \; l \geq 0\}. \end{equation}  We also list an equivalent definition of $a,b,$ and $c$ from \cite{ManolescuPin}, using homology: 
\begin{align}\label{eq:aaltdef}
a(X) &= \mathrm{min}\; \{ r \equiv s \; \mathrm{mod}\; 4 \mid \exists \, x \in \tH^G_r(X), x \in \mathrm{Im}\, v^l \;\mathrm{for\; all}\; l \geq 0 \},
\\ \nonumber
b(X) &= \mathrm{min}\; \{ r \equiv s+1 \; \mathrm{mod}\; 4 \mid \exists \, x \in \tH^G_r(X), x \in \mathrm{Im}\, v^l \;\mathrm{for\; all}\; l \geq 0 \}-1,
\\ \nonumber
c(X) &= \mathrm{min}\; \{ r \equiv s+2 \; \mathrm{mod}\; 4 \mid \exists \, x \in \tH^G_r(X), x \in \mathrm{Im}\, v^l \;\mathrm{for\; all}\; l \geq 0 \}-2.
\end{align}
\begin{rmk}  The Manolescu invariants of \cite{ManolescuPin} are defined in terms of $a,b$, and $c$, as we will review in Section \ref{sec:gauge}. \end{rmk} 

\begin{defn}[see \cite{ManolescuK}]
Let $X$ and $X'$ be spaces of type SWF, $m,m' \in \mathbb{Z}$, and $n,n' \in \mathbb{Q}$.  We say that the triples $(X,m,n)$ and $(X',m',n')$ are \emph{stably equivalent} if $n-n' \in \mathbb{Z}$ and there exists a $G$-equivariant homotopy equivalence, for some $r \gg 0$ and some nonnegative $M \in \mathbb{Z}$ and $N \in \mathbb{Q}$:
\begin{equation}\label{eq:gstabdef} \Sigma^{r\mathbb{R}} \Sigma^{(M-m)\tilde{\mathbb{R}}}\Sigma^{(N-n)\mathbb{H}}X \rightarrow \Sigma^{r\mathbb{R}} \Sigma^{(M'-m')\tilde{\mathbb{R}}}\Sigma^{(N-n')\mathbb{H}}X'.\end{equation}
\end{defn} 
Let $\mathfrak{E}$ be the set of equivalence classes of triples $(X,m,n)$ for $X$ a space of type $\mathrm{SWF}$, $m\in \mathbb{Z}$, $n \in \mathbb{Q}$, under the equivalence relation of stable $G$-equivalence\footnote{This convention is slightly different from that of \cite{ManolescuK}.  The object $(X,m,n)$ in the set of stable equivalence classes $\mathfrak{E}$, as defined above, corresponds to $(X,\frac{m}{2},n)$ in the conventions of \cite{ManolescuK}.}.  The set $\mathfrak{E}$ may be considered as a subcategory of the $G$-equivariant Spanier-Whitehead category \cite{ManolescuPin}, by viewing $(X,m,n)$ as the formal desuspension of $X$ by $m$ copies of $\tilde{\mathbb{R}} $ and $n$ copies of $\mathbb{H}$.  For $(X,m,n),(X',m',n') \in \mathfrak{E}$, a map $(X,m,n) \rightarrow (X',m',n')$ is simply a map as in (\ref{eq:gstabdef}) that need not be a homotopy equivalence.  We define Borel homology for $(X,m,n) \in \E$ by 
\begin{equation}\label{eq:borelgeneral}
\tilde{H}^{G}_*((X,m,n))=\tilde{H}^G_*(X)[m+4n],
\end{equation}
where, for a graded module $M$, $M[s]$ is defined by $(M[s])_i=M_{i+s}$.  The well-definedness of (\ref{eq:borelgeneral}) follows from Proposition \ref{prop:susp}.

\begin{defn}\label{def:loceq}
We call $X_1,X_2 \in \E$ \emph{locally equivalent} if there exist $G$-equivariant (stable) maps
\[\phi: X_1 \rightarrow X_2, \]
\[\psi: X_2 \rightarrow X_1, \]
 which are $G$-homotopy equivalences on the $S^1$-fixed-point set.  For such $X_1,X_2$, we write $X_1 \equiv_{l} X_2$, and let $\mathfrak{LE}$ denote the set of local equivalence classes.
 
\end{defn} 
 Local equivalence is easily seen to be an equivalence relation.  The set $\mathfrak{LE}$ comes with an abelian group structure, with multiplication given by smash product, and inverses given by Spanier-Whitehead duality.  
 
We also define the invariants $\alpha,\beta,$ and $\gamma$ associated to an element of $\mathfrak{E}$.  
\begin{defn}\label{def:manolescudefn1}
For $[(X,m,n)] \in \E$, we set 
\begin{equation}\label{eq:manoldefns}
\alpha((X,m,n))=\frac{a(X)}{2}-\frac{m}{2}-2n,
\;\beta((X,m,n))=\frac{b(X)}{2}-\frac{m}{2}-2n,
\end{equation}
\begin{equation}\label{eq:manoldefns2}
\gamma((X,m,n))=\frac{c(X)}{2}-\frac{m}{2}-2n.
\end{equation}
The invariants $\alpha,\beta$ and $\gamma$ do not depend on the choice of representative of the class $[(X,m,n)]\in \mathfrak{LE}$.  
\end{defn}   
 
\subsection{Algebra and $G$-Spaces}\label{subsec:modfromg}
Throughout this section $X$ will denote a space of type SWF.  Here we will review the construction of equivariant chain complexes associated to spaces of type SWF and recall the definition of chain local equivalence.  

From the group structure of $G$, $C^{CW}_{*}(G)=:\G$ acquires an algebra structure.  The algebra structure is determined by the multiplication map $G \times G\rightarrow G$ from which we obtain a map $C^{CW}_*(G)\otimes_\f C^{CW}_*(G) \rightarrow C^{CW}_*(G)$.  Explicitly, 
\begin{equation}\label{eq:grelations}C^{CW}_{*}(G) \cong \mathbb{F}[s,j]/(sj=j^3s,s^2=0,j^4=1),\end{equation}
where $j$ is the $0$-cell that is the group element $j \in G$, and $s$ is the $1$-cell \[ s=\{e^{i\theta}\mid\theta\in (0,\pi)\}.\]  In particular, $j \in C^{CW}_0(G)$ and $s \in C^{CW}_1(G)$.
The differential on $\G$ is given by $\partial s =1+j^2$, and $\partial j=0$.  A $G$-CW structure on $X$ induces a CW structure on $X$, using the CW structure of $G$ above.  Indeed, each equivariant cell $G\times D^k$ of $X$ may be decomposed into a union of ordinary CW cells using the CW decomposition of $G$.  The relative CW chain complex $C^{CW}_*(X,\p)$ inherits the structure of a $\G$-chain complex via the map \[ G \times X \rightarrow X,\] which induces a map \[C^{CW}_*(G) \otimes_\f C^{CW}_*(X,\p) \rightarrow C^{CW}_*(X,\p).\]  That is, $C^{CW}_*(X,\p)$ is a module over $\G$, such that, for $z \in C^{CW}_*(X,\p)$, and $a\in\G$, \[\partial(az)=a\partial(z)+\partial(a)z.\]

\begin{examp}\label{ex:rsusp2}
We write down explicitly the $\G$-chain complex structure on the CW chain complex $C^{CW}_*((\tilde{\mathbb{R}}^t)^+,\p)$ associated to $(\tilde{\mathbb{R}}^t)^+$, a space of type SWF at level $t$.  Note that the representation $(\tilde{\mathbb{R}}^t)^+$ admits a $G$-CW decomposition with $0$-skeleton a copy of $S^0$ on which $G$ acts trivially, and an $i$-cell $c_i$ of the form $D^i \times \mathbb{Z}/2$ for $i$ with $ 1\leq i \leq t$.  One of the two points of the $0$-skeleton of $(\tilde{\mathbb{R}}^t)^+$ is fixed as the basepoint.  

Identifying $c_i$ with its image in $C^{CW}_*((\tilde{\mathbb{R}}^t)^+,\p)$, we have $\partial (c_0)=0, \partial (c_1)=c_0,$ and $\partial(c_i)=(1+j)c_{i-1}$ for $i\geq 2$.  The action of $\G$ is given by the relations $jc_0=c_0$, $j^2c_i=c_i$ for $i\geq 1$, and $sc_i=0$ for all $i$ (in particular, the CW cells of $((\tilde{\mathbb{R}}^t)^+,\p)$ are precisely $c_0,c_1,\dots,c_t$ and $jc_1,\dots,jc_t$, and all of these are distinct).  
\end{examp}
  
\begin{examp}\label{ex:hsusp2}We find a CW decomposition for $\mathbb{H}^+$ as a $G$-space and the corresponding $\G$-chain complex structure.  We write elements of $\mathbb{H}$ as pairs of complex numbers $(z,w)=(r_1e^{i\theta_1},r_2e^{i\theta_2})$ in polar coordinates.  The action of $j$ is then given by $j(z,w)=(-\bar{w},\bar{z})$.  Fix the point at infinity as the base point.  We let $(0,0)$ be the ($G$-invariant) $0$-cell labelled $r_0$.  We let $y_1$ be the $G$-$1$-cell given by the orbit of $\{(r_1,0) \mid r_1 > 0\}$:
\[\{(r_1e^{i\theta},r_2e^{i\theta})\mid \, r_1r_2=0\}.\]
We take $y_2$ the $G$-$2$-cell given by the orbit of $\{(r_1,r_2) \mid r_1r_2 \neq 0\}$:
\[\{(r_1e^{i\theta_1},r_2e^{i\theta_2}) \mid \, \theta_1=\theta_2 \;\mathrm{or}\; \theta_1=\theta_2+\pi \;\mathrm{mod}\; 2\pi\}.\] 
Finally, $y_3$ consists of the orbit of $\{(r_1e^{i\theta_1},r_2) \mid \theta_1\in (0,\pi),\;r_1r_2 \neq 0\}$:
\[\{(r_1e^{i\theta_1},r_2e^{i\theta_2}) \mid \, \theta_1 \neq \theta_2 \; \mathrm{mod} \; \pi\}.\]
We also find the $\G$-chain complex structure of $C^{CW}_*(\mathbb{H}^+,\p)$.  One may check that the differentials are given by 
\begin{equation}\label{eq:hdiffs}\partial(r_0)=0,\; \partial y_1 = r_0, \;\partial y_2 = (1+j)y_1, \; \mathrm{and} \;\partial y_3 = sy_1+(1+j)y_2.\end{equation}  The $\G$-action on the fixed-point set, $r_0$, is necessarily trivial.  However, elsewhere the $G$-action on $(\mathbb{H}^+,\p)$ is free, and so the submodule (not a subcomplex, however) of $C^{CW}_*(\mathbb{H}^+,\p)$ generated by $y_1,y_2,y_3$ is $\G$-free, specifying the $\G$-module structure of $C^{CW}_*(\mathbb{H}^+,\p)$.
\end{examp}
The $\G$-chain complex structure on a smash product is given by:
\begin{equation}\label{eq:smashform} C^{CW}_*(X_1 \wedge X_2, \p)=C^{CW}_*(X_1,\p) \otimes_\f C^{CW}_*(X_2,\p),\end{equation}
with $\G$-action given by:
\begin{equation}\label{eq:gtensoract}
s(a \otimes b) = sa \otimes b + j^2a \otimes sb,
\end{equation} 
\[j(a\otimes b) =ja \otimes jb.\]
  Furthermore, the CW chain complex for the $G$-smash product $X_1 \wedge_G X_2 $ is given by:
  \begin{equation}\label{eq:smashdec2} C^{CW}_*(X_1 \wedge_G X_2,\p) \simeq C^{CW}_*(X_1\wedge X_2,\p) /\G=:C^{CW}_*(X_1,\p)\otimes_\G C^{CW}_*(X_2,\p). \end{equation}
We will write elements of the latter as $x_1 \otimes_\G x_2$.  Note that Borel homology $\tilde{H}^G_*(X)$ is calculated using a $G$-smash product, and so may be computed from the following chain complex:
\begin{equation}\label{eq:defofborel}
\tilde{H}^G_*(X)=H(C^{CW}_*(EG) \otimes_\G C^{CW}_*(X,\p),\partial).
\end{equation}
In (\ref{eq:defofborel}), we choose some fixed $G$-CW decomposition of $EG$ to define $C^{CW}_*(EG)$.  Equation (\ref{eq:defofborel}) shows that we may define $G$-Borel homology for $\G$-chain complexes $Z$ by:
\begin{equation}\label{eq:chainborel} H^G_*(Z)=H(C^{CW}_*(EG) \otimes_\G Z,\partial),\end{equation}
and similarly for $S^1$-Borel homology:
\begin{equation}\label{eq:chains1borel}
H^{S^1}_*(Z)=H(C^{CW}_*(EG) \otimes_{C^{CW}_*(S^1)} Z,\partial),
\end{equation}  
where $C^{CW}_*(S^1)$ is viewed as a subcomplex of $\G$.  The definitions are compatible in that:

\begin{fact}\label{fct:chainhomologydefagrees}
If $Z=C^{CW}_*(X,\p)$ is the relative CW chain complex of a $G$-space $X$, then $H^G_*(Z)=\tilde{H}^G_*(X).$
\end{fact}

We note also that $C^*_{CW}(EG)$, the dual of $C^{CW}_*(EG)$, acts on $C^{CW}_*(EG)$ by a version of the cap product (see, for instance, \cite{Brown}[\S V.3]).  In particular, $H(C^*_{CW}(EG))=\f[q,v]/(q^3)$ acts on $H(C^{CW}_*(EG)\otimes Z)$ for $Z$ a $\G$-chain complex.

\begin{defn}\label{def:typswf} 
We call a $\G$-chain complex $Z$ \emph{a suspensionlike chain complex of type SWF} at level $s$ if $Z$ is generated by
\begin{equation}\label{eq:gmoddef}  \{ c_0,c_1,c_2,\dots,c_s \} \cup \bigcup_{i \in I} \{x_i\},\end{equation}  
subject to the following conditions\footnote{In \cite{betaseifert}, we also introduced \emph{chain complexes of type SWF}, but we will not have need of this more general definition here.}.  The element $c_i$ is of degree $i$, and $\{c_0,c_1,c_2, \dots, c_s\}$ generates a subcomplex of $Z$.  The only relations are $j^2c_i=c_i,sc_i=0,jc_0=c_0$.  The differentials are given by $\partial{c_1}=c_0$, and $\partial{c_i}=(1+j)c_{i-1}$ for $2\leq i \leq s$.  Here $I$ is some index set, and the submodule generated by $\{ x_i \}_{i \in I}$ is free under the action of $\G$.  Moreover, for all $i$, $\mathrm{deg}\, x_i >s$.  We call the submodule generated by $\{c_i\}$ the \emph{fixed-point set} of $Z$.  We call $(1+j)c_s$, for $s>0$, or $c_0$ for $s=0$, the \emph{fundamental class} of $Z^{S^1}$, in analogy with the usual fundamental class of $(\mathbb{R}^{s})^+$.  Indeed, $(1+j)c_s$ (respectively, $c_0$) generates $H_s(Z)\cong H_s((\tilde{\mathbb{R}}^s)^+)$ for $s>0$ (respectively, $s=0$).
\end{defn}

Suspensionlike chain complexes of type SWF are to be thought of as reduced $G$-CW chain complexes of spaces to type SWF.  

For $Z$ a suspensionlike chain complex of type SWF we define $a(Z),b(Z),c(Z)$ as in (\ref{eq:aaltdef}).  That is,
\begin{align}\label{eq:chainaaltdef} 
a(Z)&=\mathrm{min}\; \{ r \equiv s \; \mathrm{mod}\; 4 \mid \exists \, x \in H^G_r(Z), x \in \mathrm{Im}\, v^l \;\mathrm{for\; all}\; l \geq 0 \},\\ \nonumber
b(Z)&=\mathrm{min}\; \{ r \equiv s+1 \; \mathrm{mod}\; 4 \mid \exists \, x \in H^G_r(Z), x \in \mathrm{Im}\, v^l \;\mathrm{for\; all}\; l \geq 0 \}-1,
\\ \nonumber
c(Z)&=\mathrm{min}\; \{ r \equiv s+2 \; \mathrm{mod}\; 4 \mid \exists \, x \in H^G_r(Z), x \in \mathrm{Im}\, v^l \;\mathrm{for\; all}\; l \geq 0 \}-2.
\end{align}
\begin{fact}\label{fct:abcdefsag}
If $Z=C^{CW}_*(X,\p)$ is the CW-chain complex of a $G$-space $X$, then $a(Z)=a(X),b(Z)=b(X),c(Z)=c(X)$.  
\end{fact}

To introduce chain local equivalence, we will consider the CW chain complexes coming from suspensions.

For $Z$ a suspensionlike chain complex of type SWF, and $V$ a $G$-representation a direct sum of copies of $\mathbb{R}$, $\tilde{\mathbb{R}}$ and $\mathbb{H}$, we define suspension by $V$ by: \[ \Sigma^VZ=C^{CW}_*(V^+,\p)\otimes_\f Z,\] with the $\G$-action as in (\ref{eq:gtensoract}).  Since suspension by $V$ applied to a $G$-space $X$ is the smash product $V \wedge X$, and since the chain complex of a smash product is as in (\ref{eq:smashform}), we see:
\begin{lem}\label{lem:cwsus}
Let $V=\mathbb{H}$, $\tilde{\mathbb{R}}$, or $\mathbb{R}$.  If $Z=C^{CW}_*(X,\p)$ for $X$ a space of type SWF, then $\Sigma^VZ = C^{CW}_*(\Sigma^VX,\p)$. 
\end{lem} 
As an example, we have $C^{CW}_*(\mathbb{R}^+,\p)=\langle r_1 \rangle$, on which $\G$ acts by $jr_1=r_1, sr_1=0$.  Then:\begin{equation}\label{eq:rsusp}\Sigma^{\mathbb{R}}Z=Z[-1].\end{equation}
We further note that $C^{CW}_*((V \otimes_\f W)^+,\p)\simeq C^{CW}_*(V^+,\p) \otimes C^{CW}_*(W^+,\p)$, where $\simeq$ denotes chain homotopy equivalence.
 
\begin{defn}\label{def:compst}
Let $Z_i$ be suspensionlike chain complexes of type SWF, $m_i \in \mathbb{Z}, n_i \in \mathbb{Q}$, for $i=1,2$.  We call $(Z_1,m_1,n_1)$ and $(Z_2,m_2,n_2)$ \emph{chain stably equivalent} if there exist $M \in \mathbb{Z}, N \in \mathbb{Q}$ and maps
\begin{equation}\label{eq:compst}
\Sigma^{(N-n_1)\mathbb{H}}\Sigma^{(M-m_1)\tilde{\mathbb{R}}} Z_1 \rightarrow 
\Sigma^{(N-n_2)\mathbb{H}}\Sigma^{(M-m_2)\tilde{\mathbb{R}}} Z_2
\end{equation}
\begin{equation}\label{eq:compst2}
\Sigma^{(N-n_1)\mathbb{H}}\Sigma^{(M-m_1)\tilde{\mathbb{R}}} Z_1 \leftarrow 
\Sigma^{(N-n_2)\mathbb{H}}\Sigma^{(M-m_2)\tilde{\mathbb{R}}} Z_2,
\end{equation}
which are chain homotopy equivalences.  In particular, the existence of a map as in (\ref{eq:compst}) or (\ref{eq:compst2}) implies  $n_2 - n_1 \in \mathbb{Z}$.  
\end{defn}
\begin{rmk} We do not consider suspensions by $\mathbb{R}$, unlike in the case of stable equivalence for spaces, since by (\ref{eq:rsusp}), no new maps are obtained by suspending by $\mathbb{R}$.  
\end{rmk}
Chain stable equivalence is an equivalence relation, and we denote the set of chain stable equivalence classes by $\ce$.  Associated to a triple $(X,m,n) \in \E$, we obtain an element $(C^{CW}_*(X),m,n) \in \ce$, with $C^{CW}_*(X)$ a suspensionlike chain complex of type SWF.  

\begin{rmk}\label{rmk:suspensionissue}
For $X$ a space of type SWF, $C^{CW}_*(X,\p)$ need not be a suspensionlike chain complex of type SWF.  However, any class in $\E$ admits a representative $(X,m,n)$ with $C^{CW}_*(X,\p)$ a suspensionlike chain complex of type SWF.  
\end{rmk}  

\begin{defn}\label{def:compstl}
Let $Z_i$ be suspensionlike chain complexes of type SWF, $m_i \in \mathbb{Z}, n_i \in \mathbb{Q}$, for $i=1,2$.  We call $(Z_1,m_1,n_1)$ and $(Z_2,m_2,n_2)$ \emph{chain locally equivalent} if there exist $M \in \mathbb{Z}, N \in \mathbb{Q}$ and maps
\begin{equation}\label{eq:compstl}
\Sigma^{(N-n_1)\mathbb{H}}\Sigma^{(M-m_1)\tilde{\mathbb{R}}} Z_1 \rightarrow 
\Sigma^{(N-n_2)\mathbb{H}}\Sigma^{(M-m_2)\tilde{\mathbb{R}}} Z_2
\end{equation}
\begin{equation}\label{eq:compstl2}
\Sigma^{(N-n_1)\mathbb{H}}\Sigma^{(M-m_1)\tilde{\mathbb{R}}} Z_1 \leftarrow 
\Sigma^{(N-n_2)\mathbb{H}}\Sigma^{(M-m_2)\tilde{\mathbb{R}}} Z_2,
\end{equation}
which are chain homotopy equivalences on the fixed-point sets.  
\end{defn}  
We call a map as in (\ref{eq:compstl}) or (\ref{eq:compstl2}) a \emph{chain local equivalence}.  Elements $Z_1,Z_2 \in \ce$ are chain locally equivalent if and only if there are chain local equivalences $Z_1 \rightarrow Z_2$ and $Z_2 \rightarrow Z_1$.  There are pairs of chain complexes with a chain local equivalence in one direction but not the other; these are not chain locally equivalent complexes.    

Chain local equivalence is an equivalence relation, and we write $[(Z,m,n)]_{cl}$ for the chain local equivalence class of $(Z,m,n)\in \mathfrak{CE}$.  The set $\mathfrak{CLE}$ of chain local equivalence classes is naturally an abelian group, with multiplication given by the tensor product (over $\f$, with $\G$-action as above).  (This abelian group structure on $\mathfrak{CLE}$ corresponds to connected sum in the homology cobordism group; see Fact \ref{fct:homf}).  The inverse of an element $(Z,0,0)$ of $\mathfrak{CLE}$ is $(Z^*,0,0)$ where $Z^*$ denotes the chain complex dual to $Z$.  The identity element $0$ of $\mathfrak{CLE}$ is $(\f,0,0)$, where $C^{CW}_*(S^0,\p)=\f=\langle f_0\rangle$ is the $\G$-module concentrated in degree $0$ for which $jf_0=f_0$ and $sf_0=0$.

In analogy with (\ref{eq:borelgeneral}), we define Borel homology for elements of $\mathfrak{CE}$.  
\begin{defn}\label{def:cehomwelldef}
Let $(Z,m,n)\in \mathfrak{CE}$.  We define $H^G_*((Z,m,n))=H^G_*(Z)[m+4n]$.  
\end{defn}
\begin{fact}\label{fct:chainborelhomologywelldef}
For $Z\in \mathfrak{CE}$, $H^G_*(Z)$ is well-defined.  
\end{fact}\begin{proof}
It suffices to show, for $Z$ a suspensionlike chain complex of type SWF, that \begin{equation}\label{eq:homtarg}H^G_*(\Sigma^VZ)=H^G_{*-\mathrm{dim}\; V}(Z).\end{equation}
By (\ref{eq:chainborel}), we need to compute
\[ H_*(C^{CW}_*(EG) \otimes_\G (C^{CW}_*(V^+,\p) \otimes_\f Z)).\]
However, we have, by (\ref{eq:smashform}),
\[ C^{CW}_*(EG) \otimes_\f (C^{CW}_*(V^+,\p)\otimes_\f Z)=(C^{CW}_*(EG) \otimes_\f C^{CW}_*(V^+,\p))\otimes_\f Z,\]
as $\G$-modules.  Recalling the definition of $\otimes_\G$ in (\ref{eq:smashdec2}) we have
\[ C^{CW}_*(EG)\otimes_\G (C^{CW}_*(V^+,\p)\otimes_\f Z)=(C^{CW}_*(EG) \otimes_\f C^{CW}_*(V^+,\p))\otimes_\G Z.\]
Then to show (\ref{eq:homtarg}) we need only show
\[ H_*( (C^{CW}_*(EG) \otimes_\f C^{CW}_*(V^+,\p)) \otimes_\G Z)=H_{*-\mathrm{dim} V}(C^{CW}_*(EG) \otimes_\G Z).\]
Indeed, $C^{CW}_*(EG) \otimes_\f C^{CW}_*(V^+,\p)$ is the relative CW chain complex of $\Sigma^VEG_+$, a free $G$-space with nonzero homology only in degree $\mathrm{dim}\; V$.  As any two $\G$-free resolutions are homotopy equivalent, we obtain $C^{CW}_*(EG) \otimes_\f C^{CW}_*(V^+,\p) \simeq C^{CW}_*(EG)[-\mathrm{dim}\; V]$.  Then we have
\[ H_*( (C^{CW}_*(EG) \otimes_\f C^{CW}_*(V^+,\p)) \otimes_\G Z)=H_*((C^{CW}_*(EG) \otimes_\G Z)[-\mathrm{dim}\; V])=H^G_{*-\mathrm{dim}\; V}(Z),\]
as needed.  
\end{proof}

\begin{defn}\label{def:chainmanoinvars}
For $[(Z,m,n)] \in \CEL$, we call 
\begin{equation}
\alpha((Z,m,n))=\frac{a(Z)}{2}-\frac{m}{2}-2n,
\;\beta((Z,m,n))=\frac{b(Z)}{2}-\frac{m}{2}-2n,
\end{equation}
\[
\gamma((Z,m,n))=\frac{c(Z)}{2}-\frac{m}{2}-2n,
\]
the \emph{Manolescu invariants} of $(Z,m,n)$.  The invariants $\alpha,\beta$ and $\gamma$ do not depend on the choice of representative of the class $[(Z,m,n)]$.  
\end{defn}  

\begin{thm}[\cite{betaseifert}]\label{thm:gcwstruc}
The map
\begin{equation}\label{eq:etoce}
\mathfrak{E} \rightarrow \mathfrak{CE}
\end{equation}
given by sending $(X,m,n) \rightarrow (C^{CW}_*(X,\p),m,n)$ descends to a homomorphism:
\begin{equation}\label{eq:letocle}
\mathfrak{LE} \rightarrow \CEL .
\end{equation}
\end{thm}

We will denote the element $(Z,0,0) \in \CEL$, for $Z$ a suspensionlike chain complex of type SWF, simply by $Z$.  
\subsection{Ordering $\mathfrak{CLE}$}\label{subsec:order}
In the following section we depart from \cite{betaseifert} and define a partial order on $\mathfrak{CLE}$. 
\begin{defn}\label{def:partord1}
The groups $\mathfrak{LE}$ and $\mathfrak{CLE}$ also come with a natural partial ordering.  That is, we say $X_1 \preceq X_2$ if there exists a local equivalence $X_1 \rightarrow X_2$ or a local equivalence $\Sigma^{\frac{1}{2}\mathbb{H}}X_1 \rightarrow X_2$, for $X_1,X_2 \in \LE$.  For $(Z,m,n) \in \CEL$, we write $\Sigma^{\frac{1}{2}\mathbb{H}}(Z,m,n)=(Z,m,n-\frac{1}{2})$.  For $Z_1,Z_2 \in \CEL$, we say $Z_1 \preceq Z_2$ if there exists a chain local equivalence $Z_1 \rightarrow Z_2$ or if there exists a chain local equivalence $\Sigma^{\frac{1}{2}\mathbb{H}}Z_1 \rightarrow Z_2.$ 
\end{defn} 

We have:
\begin{lem}\label{lem:abcor1}
If $Z_1 \preceq Z_2 \in \CEL$, then $\alpha(Z_1)\leq \alpha(Z_2),\beta(Z_1)\leq \beta(Z_2), \gamma(Z_1) \leq \gamma(Z_2)$.
\end{lem}
\begin{proof}
We assume without loss of generality $Z_1=(Z_1,0,0), \; Z_2=(Z_2,0,0)$, for suspensionlike chain complexes of type SWF $Z_1$ and $Z_2$.  A chain local equivalence $\phi: Z_1 \rightarrow Z_2$ induces a map $\phi_G: C^{CW}_*(EG) \otimes_\G Z_1 \rightarrow C^{CW}_*(EG) \otimes_\G Z_2$.  We then have a commuting triangle, where $\iota_1$ and $\iota_2$ come from the inclusions $Z_1^{S^1} \rightarrow Z_1$ and $Z_2^{S^1}\rightarrow Z_2$.
\begin{equation}\label{eq:comtriag}
\begin{tikzcd}[column sep = small]
C^{CW}_*(EG) \otimes_\G Z_1 \arrow{rr}{\phi} & & C^{CW}_*(EG) \otimes_\G Z_2 \\
\\
& C^{CW}_*(EG) \otimes_\G C^{CW}_*((\tilde{\mathbb{R}}^{S^1})^+,\p)\arrow{luu}{\iota_1} \arrow{ruu}{\iota_2}
\end{tikzcd}
\end{equation} 
Diagram (\ref{eq:comtriag}) also induces a commuting triangle in homology:
\begin{equation}\label{eq:homcomtriag}
\begin{tikzcd}
H^G_*(Z_1) \arrow{rr}{\phi_*} & & H^G_*(Z_2) \\
\\
& H^G_*(Z_1^{S^1})\arrow{luu}{\iota_{1,*}} \arrow{ruu}{\iota_{2,*}}
\end{tikzcd}
\end{equation}
As in Remark 2.24 of \cite{betaseifert}, a suspensionlike chain complex of type SWF is chain stably equivalent to some $C^{CW}_*(X,\p)$ for $X$ a space of type SWF.  Then we may apply Fact \ref{fct:highgrads} to see that $\iota_{1,*}$ and $\iota_{2,*}$ are isomorphisms in sufficiently high degree. Thus $\phi_*$ must be an isomorphism in sufficiently high degree.  Furthermore, 
\[ \mathrm{Im}\;\iota_i =\{ x \in H^G_*(Z_i) \mid x\in \mathrm{Im} \, v^l \, \mathrm{for \;all}\; l\geq 0 \},\]
from (\ref{eq:iotaim}).  Thus, if $x \in H^G_*(Z_2)$ is in $\mathrm{Im}\, v^l$ for all $l\geq 0$, there exists some $y$ so that $x=\iota_{2,*}y$.  By the commutativity of (\ref{eq:homcomtriag}), $\iota_{1,*}(y)\neq 0$.  Applying the definitions (\ref{eq:aaltdef}), we see $m(Z_2)\geq m(Z_1)$ where $m$ is any of $a,b,c$.  Applying Definition \ref{def:chainmanoinvars}, the Lemma follows.

A similar argument applies for a chain local equivalence $\phi: \Sigma^{\frac{1}{2}\mathbb{H}}Z_1 \rightarrow Z_2$, in which case one has:
\[\alpha(Z_1)\leq \alpha(Z_2)-1,\,\beta(Z_1)\leq \beta(Z_2)-1,\, \gamma(Z_1) \leq \gamma(Z_2)-1.\]
\end{proof}

\begin{lem}\label{lem:ordadd}
Let $Z_1,Z_2,Z_3$ complexes of type SWF with $Z_1 \preceq Z_2$.  Then $Z_1 \otimes Z_3 \preceq Z_2 \otimes Z_3$.  
\end{lem}
\begin{proof}
If there exists a (stable) map:
\[\phi: Z_1 \rightarrow Z_2,\]
then $\phi \otimes \mathrm{Id}: Z_1 \otimes Z_3 \rightarrow Z_2 \otimes Z_3$ satisfies the conditions of Definition \ref{def:partord1}, establishing the Lemma (and similarly for suspensions by $\frac{1}{2}\mathbb{H}$).  
\end{proof}

\section{Inequalities for the Manolescu Invariants}\label{sec:ineqs} 
In this section we will obtain bounds on the Manolescu invariants of tensor products of suspensionlike chain complexes.  In Section \ref{sec:gauge} we will apply these results to obtain bounds on the Manolescu invariants of three-manifolds.
\subsection{Calculating Manolescu Invariants from a Chain Complex}\label{subsec:chaincalcs}
We start by fixing a convenient $G$-CW decomposition of $EG=S(\mathbb{H}^\infty)$.  Recalling Example \ref{ex:hsusp2}, we have a $G$-CW decomposition for $\mathbb{H}^+\cong S^4=\langle r_0, y_1,y_2,y_3\rangle$ with differentials as in (\ref{eq:hdiffs}).  We then attach free $G$ cells $y_5, y_6,y_7$, with $\mathrm{deg} \; y_i = i$, where the attaching map of $y_i$ is the suspension of the attaching map of $y_{i-4}$.  The result is a $G$-CW decomposition by cells $\{ r_0,  y_i \}$, for $i\leq 7, i\neq 4$, of $S^8\cong (\mathbb{H}^2)^+$.  We can repeat this procedure to obtain a $G$-CW decomposition of $((\mathbb{H}^n)^+,\p)$ for any $n$, by cells $\{ r_0,y_i\}_{i \not\equiv 0 \bmod{4}}$.  

The unit sphere $S(\mathbb{H}^n)$ admits a $G$-CW decomposition with $G$ $(i-1)$-cells $e_{i-1}=y_i\cap S(\mathbb{H}^n)$ for $i\leq {4n-1}$.

In the limit, the $e_i$ provide a $G$-CW decomposition of $S(\mathbb{H}^\infty)=EG$.  That is, there is a $G$-CW decomposition of $EG$ with cells $e_{4i},\, e_{4i+1},\, e_{4i+2}$ for $i\geq 0$.  The chain complex $C^{CW}_*(EG)$ is then the free $\G$-module on $e_i$ with
\begin{align}\label{eq:diffrules1} \partial(e_0)&=0, \\ \partial(e_{4i})&=s(1+j+j^2+j^3)e_{4i-2} \; \mathrm{for} \; i \geq 1, \nonumber \\ \partial(e_{4i+1})&=(1+j)e_{4i}, \nonumber \\ \partial(e_{4i+2})&=(1+j)e_{4i+1}+se_{4i}. \nonumber
\end{align}
The reader may check that $H(C^{CW}_*(EG))$, for $C^{CW}_*(EG)$ as above, is a copy of $\f$ concentrated in degree $0$.  As all contractible free $\G$-chain complexes are chain homotopy equivalent, all $G$-CW complexes for $EG$ have CW chain complex chain homotopic to that given above.  

Fix a space $X$ of type SWF so that $Z=C^{CW}_*(X,\p)$ is a suspensionlike chain complex of type SWF.  (By Remark \ref{rmk:suspensionissue}, for any class in $\E$ there will be such a representative $X$).  One may compute the reduced Borel homology of $X$ in terms of $Z$, using (\ref{eq:chainborel}) and (\ref{eq:chains1borel}).  

In particular, we show how to determine $a(Z),b(Z),c(Z)$ from $Z$.    
\begin{lem}\label{lem:manolowerbounds}
Let $Z$ be a suspensionlike chain complex of type SWF at level $t$, with fundamental class $f_t \in Z^{S^1}$, of degree $t$, and $A,B,C \in \mathbb{Z}_{\geq 0}$.  Then $a(Z) \geq 4A+t$ if and only if there exist elements $x_i \in Z$, $\mathrm{deg}\, x_i=i$, for all $i$ with $t+1 \leq i \leq t+4A-3$ and $i \not\equiv t+2 \, \mathrm{mod} \, 4$, so that 
\begin{equation}\label{eq:rawlowerbounds}
\partial(x_{i})=\begin{cases}f_t & \mbox{if } i=t+1\\ s(1+j+j^2+j^3)x_{i-2} & \mbox{if } i \equiv t+3 \, \mathrm{mod} \, 4, \, i \leq t+4A-3 \\ (1+j)x_{i-1} & \mbox{if } i \equiv t \, \mathrm{mod} \, 4, i \leq t+4A-3 \\ (1+j)x_{i-1}+sx_{i-2} & \mbox{if } i \equiv t+1 \, \mathrm{mod} \, 4, \, t+1<i \leq t +4A -3.  \end{cases}
\end{equation}
Also, $b(Z) \geq 4B+t$ if and only if there exist elements $x_i \in Z$, $\mathrm{deg}\, x_i=i$, for all $i$ with $t+1 \leq i \leq t+4B-2$ and $i \not\equiv t+3 \, \mathrm{mod} \, 4$ so that 
\begin{equation}\label{eq:brawlowerbounds}
\partial(x_{i})=\begin{cases}f_t & \mbox{if } i=t+1\\ (1+j)x_{t+1} & \mbox{if } i=t+2\\ s(1+j+j^2+j^3)x_{i-2} & \mbox{if } i \equiv t \, \mathrm{mod} \, 4, \, i \leq t+4B-2 \\ (1+j)x_{i-1} & \mbox{if } i \equiv t+1 \, \mathrm{mod} \, 4, t+1<i \leq t+4B-2 \\ (1+j)x_{i-1}+sx_{i-2} & \mbox{if } i \equiv t+2 \, \mathrm{mod} \, 4, \, t+2<i \leq t +4B -2.  \end{cases}\end{equation}
Also, $c(Z) \geq 4C+t$ if and only if there exist elements $x_i \in Z$, $\mathrm{deg}\, x_i=i$, for all $i$ with $t+1 \leq i \leq t+4C-1$ and $i \not\equiv t \, \mathrm{mod} \, 4$ so that 
\begin{equation}\label{eq:crawlowerbounds}
\partial(x_{i})=\begin{cases}f_t & \mbox{if } i=t+1\\ (1+j)x_{i-1} & \mbox{if } i \equiv t+2 \, \mathrm{mod} \, 4, \, i \leq t+4C-1 \\ (1+j)x_{i-1}+sx_{i-2} & \mbox{if } i \equiv t+3 \, \mathrm{mod} \, 4, i \leq t+4C-1 \\ s(1+j+j^2+j^3)x_{i-2} & \mbox{if } i \equiv t+1 \, \mathrm{mod} \, 4, \, t+1<i \leq t +4C -1.  \end{cases}
\end{equation}
\end{lem}
\begin{proof}
By (\ref{eq:iotaim}), we have, where $\iota_*: H^G_*(Z^{S^1})\rightarrow H^G_*(Z)$ is the map induced by inclusion, 
\begin{equation}\label{eq:hominftydef}
\mathrm{Im}\; \iota_*=\{ x \in H^G_*(Z) \mid x \in \mathrm{Im}\; v^l \; \mathrm{for\; all } \; l \geq 0\}.
\end{equation}
Further, $H^G_*(Z^{S^1})$ is given by:
\[H^G_*(Z^{S^1})= C^{CW}_*(EG) \otimes_\f f_t,\]
which is an $\f$-vector space with generators $e_i \otimes f_t$ in degree $i+t$ for $i$ such that $i \geq 0 $ and $i \not\equiv 3 \bmod{4}$.  Then $a(Z) \geq 4A+t$ is equivalent to $e_{4A-4} \otimes f_t$ being a boundary in \[C^G_*(Z)=C^{CW}_*(EG)\otimes_\G Z.\]  That is, $a(Z) \geq 4A +t$ is equivalent to the existence of some 
\[x=\sum_{i=t}^{i=t+4A-3} e_{t+4A-3-i}\otimes x_{i} \in C^{CW}_*(EG) \otimes_\G Z,\]
so that $\partial(x)=e_{4A-4}\otimes f_t$, where $x_i \in Z$ is of degree $i$.  Writing out the differential of $x$, one obtains the conditions (\ref{eq:rawlowerbounds}) of the Lemma.  Similarly, $b(Z) \geq 4B+t$ if and only if $e_{4B-3} \otimes f_t$ is a boundary, and $c(Z) \geq 4C+t$ if and only if $e_{4C-2} \otimes f_t$ is a boundary, from which one obtains (\ref{eq:brawlowerbounds}) and (\ref{eq:crawlowerbounds}). 
\end{proof} 
\begin{lem}\label{lem:cord}
Let $Z$ be a suspensionlike chain complex of type SWF at level $t$, so that $c(Z) \geq 4C+t$.  Then \[C^{CW}_*(\Sigma^{C\mathbb{H}}(\tilde{\mathbb{R}}^t)^+,\p) \preceq Z.\]   
\end{lem}
\begin{proof}
The chain complex $C^{CW}_*(\Sigma^{C\mathbb{H}}(\tilde{\mathbb{R}}^t)^+,\p)$ consists of cells $c_0, \dots, c_t$ constituting the $S^1$-fixed point set, and has free cells $x_i$, of degree $i$, for $i=t+1,\dots, t+4C-1$, for $i\not \equiv t \bmod{4}$.  The fundamental class of the subcomplex $C^{CW}_*((\tilde{\mathbb{R}}^t)^+,\p)$ is $f_t=(1+j)c_t$ (if $t>0$, or $f_t=c_0$ if $t=0$).  The differentials of the $x_i$ in $C^{CW}_*(\Sigma^{C\mathbb{H}}(\tilde{\mathbb{R}}^t)^+,\p)$ are given exactly by the relations in (\ref{eq:crawlowerbounds}).  Then, since $Z$ has elements satisfying (\ref{eq:crawlowerbounds}), there exists a chain local equivalence
\[C^{CW}_*(\Sigma^{C\mathbb{H}}(\tilde{\mathbb{R}}^t)^+,\p)\rightarrow Z,\]
as needed.
\end{proof}
 The problem of computing the Manolescu invariants of tensor products (and, thus, connected sums, using Fact \ref{fct:homf}) then amounts to asking how to find towers of elements of the form (\ref{eq:rawlowerbounds})-(\ref{eq:crawlowerbounds}) in  $Z_1 \otimes_\f Z_2$ from towers in $Z_1$ and $Z_2$.  Note that if $Z_1,Z_2$ are suspensionlike $\G$-chain complexes of type SWF at levels $t_1,t_2$ respectively, then their product $Z_1 \otimes_\f Z_2$ is of type SWF at level $t_1+t_2$.
 
\begin{rmk}\label{rmk:vanish}
Say $\alpha(Z)=\gamma(Z)=0$ for $Z$ a chain complex of type SWF.  Then Lemma \ref{lem:cord} implies $Z\succeq 0 \in \mathfrak{CLE}$.  By duality, $-\alpha(Z)=\gamma(Z^*)=0$, where $Z^*$ is the dual of $Z$, so $Z^* \succeq 0$.  Combined, we see $Z=0\in \mathfrak{CLE}$.  That is, if $Z \in \mathfrak{CLE}$ has $\alpha(Z)=\gamma(Z)=0$, then $[Z]_{cl}=[C^{CW}_*(S^0,\p)]_{cl}$.  
\end{rmk}
\begin{thm}\label{thm:easyineq}
For $Z_1,Z_2$ suspensionlike $\G$-chain complexes of type SWF, we have: 
\begin{equation}\label{eq:easyineqsthm}\alpha(Z_1)+\alpha(Z_2)\geq \alpha(Z_1 \otimes_\f Z_2) \geq \alpha(Z_1)+ \gamma(Z_2),\end{equation}\[
\alpha(Z_1)+\beta(Z_2)\geq \beta(Z_1 \otimes_\f Z_2) \geq \beta(Z_1)+ \gamma(Z_2),\]\[ 
\alpha(Z_1)+\gamma(Z_2)\geq\gamma(Z_1 \otimes_\f Z_2) \geq \gamma(Z_1)+\gamma(Z_2). \]
\end{thm}

\begin{proof} 
Let $Z_i$ be at level $t_i$ for $i=1,2$.  Then, by Lemma \ref{lem:cord}, $C^{CW}_*(\Sigma^{\frac{(c(Z_2)-t_2)}{4}\mathbb{H}}(\tilde{\mathbb{R}}^{t_2})^+,\p)\preceq Z_2$.  By Lemma \ref{lem:ordadd},
\[Z_1 \otimes_\f C^{CW}_*(\Sigma^{\frac{(c(Z_2)-t_2)}{4}\mathbb{H}}(\tilde{\mathbb{R}}^{t_2})^+,\p) \preceq Z_1 \otimes_\f Z_2.\]
However, $Z_1 \otimes_\f C^{CW}_*(\Sigma^{\frac{(c(Z_2)-t_2)}{4}\mathbb{H}}(\tilde{\mathbb{R}}^{t_2})^+,\p)$ is, by definition, $(Z_1,-t_2,\frac{-c(Z_2)+t_2}{4})$.  

Then \[(Z_1,-t_2,\frac{-c(Z_2)+t_2}{4})\preceq Z_1 \otimes_\f Z_2.\]  By Lemma \ref{lem:abcor1}, $M((Z_1,-t_2,\frac{-c(Z_2)+t_2}{4}))\preceq M(Z_1 \otimes_\f Z_2)$ where $M$ is any of $\alpha,\beta,$ or $\gamma$.  By Definition \ref{def:chainmanoinvars}, we have $\gamma(Z_2)=c(Z_2)/2$.  Then, again using Definition \ref{def:chainmanoinvars}, we see
\begin{align}
\alpha(Z_1,-t_2,\frac{-c(Z_2)+t_2}{4})&=\alpha(Z_1)+\gamma(Z_2)\leq \alpha(Z_1 \otimes_\f Z_2), \nonumber\\
\beta(Z_1,-t_2,\frac{-c(Z_2)+t_2}{4})&=\beta(Z_1)+\gamma(Z_2)\leq \beta(Z_1 \otimes_\f Z_2),\nonumber\\
\gamma(Z_1,-t_2,\frac{-c(Z_2)+t_2}{4})&=\gamma(Z_1)+\gamma(Z_2)\leq \gamma(Z_1 \otimes_\f Z_2).\nonumber
\end{align}  Thus, we have obtained the right-hand inequalities of (\ref{eq:easyineqsthm}).  

To obtain the left-hand inequalities, we recall from \cite{ManolescuPin}[Proposition 2.13] that $\alpha(X)=-\gamma(X^*)$ and $\beta(X)=-\beta(X^*)$ where $X$ is a space of type SWF and $X^*$ is Spanier-Whitehead dual to $X$.  The same argument as in \cite{ManolescuPin}[Proposition 2.13] implies that, for $Z$ a chain complex of type SWF, $\alpha(Z)=-\gamma(Z^*)$ and $\beta(Z)=-\beta(Z^*)$ where $Z^*$ is the dual chain complex.  The left-hand inequalities of (\ref{eq:easyineqsthm}) then follow by applying the right-hand inequalities to $Z_1^*$ and $Z_2^*$, and using the above rules for duality.  
\end{proof}
\begin{thm}\label{thm:hardineq}
For $Z_1,Z_2$ suspensionlike $\G$-chain complexes of type SWF, we have:
\begin{equation}\label{eq:hardineq}
\gamma(Z_1 \otimes_\f Z_2) \leq \beta(Z_1)+ \beta(Z_2) \leq \alpha(Z_1 \otimes_\f Z_2).
\end{equation}
\end{thm}
\begin{proof}  We construct a tower of elements in $Z_1\otimes_\f Z_2$ satisfying (\ref{eq:rawlowerbounds}) from towers in $Z_1$ and $Z_2$ satisfying (\ref{eq:brawlowerbounds}).  Say that $Z_1$ is at level $t_1$ and $Z_2$ is at level $t_2$, and denote the fundamental class of $Z_1^{S^1}$ by $f_{t_1}$ and that of $Z_2^{S^1}$ by $f_{t_2}$.  

Let $\{x_i\}_{i=t_1,\dots,b(Z_1)-2}$ and $\{y_i\}_{i=t_2,\dots,b(Z_2)-2}$ be sequences satisfying (\ref{eq:brawlowerbounds}) for $Z_1,Z_2$, respectively.  Then consider the sequence of elements:
\begin{equation}\label{eq:newbtower1}
f_{t_1} \otimes f_{t_2}, \; x_{t_1+1} \otimes f_{t_2},\; s(1+j^2)x_{t_1+2} \otimes f_{t_2}, 
\end{equation}
\[x_{t_1+4} \otimes f_{t_2},\;x_{t_1+5} \otimes f_{t_2},\; s(1+j^2)x_{t_1+6} \otimes f_{t_2},\]
\[ x_{t_1+8} \otimes f_{t_2},\; x_{t_1+9} \otimes f_{t_2}, s(1+j^2)x_{t_1+10} \otimes f_{t_2},\]\[ \dots,\] 
\[x_{b(Z_1)-4} \otimes f_{t_2},\; x_{b(Z_1)-3}\otimes f_{t_2},\;s(1+j^2)x_{b(Z_1)-2}\otimes f_{t_2}.\]
One may verify that the sequence in (\ref{eq:newbtower1}) satisfies (\ref{eq:rawlowerbounds}).  In fact, the sequence in (\ref{eq:newbtower1}) generates a subcomplex that is just a subcomplex of $Z_1$ satisfying (\ref{eq:rawlowerbounds}) smashed against $Z_2^{S^1}$.  To lengthen the sequence, we then incorporate chains coming from $Z_2$:
\begin{equation}\label{eq:newbtower2} s(1+j)^3x_{b(Z_1)-2}\otimes y_{t_2+1}, s(1+j)^3x_{b(Z_1)-2} \otimes y_{t_2+2}, \end{equation}\[s(1+j)^3x_{b(Z_1)-2}\otimes y_{t_2+4}  ,\;s(1+j)^3x_{b(Z_1)-2}\otimes y_{t_2+5},\;s(1+j)^3x_{b(Z_1)-2}\otimes y_{t_2+6},\; \]\[s(1+j)^3x_{b(Z_1)-2}\otimes y_{t_2+8},\;s(1+j)^3x_{b(Z_1)-2}\otimes y_{t_2+9},\;s(1+j)^3x_{b(Z_1)-2}\otimes y_{t_2+10} \]\[\dots,\]\[s(1+j)^3x_{b(Z_1)-2}\otimes y_{b(Z_2)-4},\; s(1+j)^3x_{b(Z_1)-2}\otimes y_{b(Z_2)-3},\;s(1+j)^3x_{b(Z_1)-2}\otimes y_{b(Z_2)-2}.
\]
One confirms that the sequence specified by (\ref{eq:newbtower1})-(\ref{eq:newbtower2}) satisfies (\ref{eq:rawlowerbounds}), and this establishes 
\[a(Z_1 \otimes_\f Z_2)\geq b(Z_1) +b(Z_2).\]
Using Definition \ref{def:chainmanoinvars}, we obtain the right-hand inequality of (\ref{eq:hardineq}).  The left-hand side follows from duality, as in the proof of Theorem \ref{thm:easyineq}.
\end{proof}

\subsection{Relationship with $S^1$-invariants} 
We also recall the definition of the invariant $d$ from \cite{ManolescuPin}, analogous to the Fr{\o}yshov invariant of $S^1$-equivariant Seiberg-Witten Floer theory.  
\begin{defn}
Let $Z$ be a suspensionlike chain complex of type SWF at level $t$.  
\begin{equation}\label{eq:ddef} 
d(Z)=\mathrm{min}\;\{ r \equiv t \;\mathrm{mod} \; 2 \mid \exists\, x \in H^{S^1}_{r}(Z),\,  x \in \mathrm{Im}\; u^l\; \mathrm{\; for\; all \;} l \geq 0 \}. 
\end{equation}
\end{defn}
\begin{rmk}In \cite{ManolescuPin}, $d_p$ is defined for coefficients in any field, rather than only $\mathbb{F}=\mathbb{Z}/2$.  The invariant $d$ in our notation is $d_2$ of \cite{ManolescuPin}.  
\end{rmk}

Analogous to the the calculation for $a,b,$ and $c$ in Lemma \ref{lem:manolowerbounds}, we find a formula for $d(Z)$.  We obtain:
\begin{lem}\label{lem:drules} Let $Z$ be a suspensionlike chain complex of type SWF at level $t$, and let $f_t$ denote the fundamental class of $Z^{S^1}$.  Then $d(Z) \geq 2D+t$ if and only if there exist elements $x_{i}$ in $Z$, for $i=t+1,\dots,t+2D-1$ with $i\equiv t+1 \bmod{2}$, where $\mathrm{deg}\; x_i = i$, such that 
\begin{equation}\label{eq:dtow} \partial(x_{i})=\begin{cases} f_t & \mbox{if } i=t+1,\\ s(1+j^2)x_{i-2}  & \mbox{if} \; t+3\leq i\leq t+2D-1. \end{cases}
\end{equation}\end{lem}
\begin{proof}
The proof is analogous to that of Lemma \ref{lem:manolowerbounds}.
\end{proof}

\begin{defn}\label{def:dtow}
We let $T_D(t)$ denote the chain complex given by \[C^{CW}_*((\tilde{\mathbb{R}}^t)^+,\p) \oplus \langle \{x_{t+2i-1}\}\rangle_{\{ i=1,\dots,D\} },\] where $\langle \{x_i\} \rangle$ is the free $\G$-module with generators $x_i$, with the following requirements.  We require that $C^{CW}_*((\tilde{\mathbb{R}}^t)^+,\p)\subseteq T_D(t)$ is a subcomplex, where $C^{CW}_*((\tilde{\mathbb{R}}^t)^+,\p)$ is as in Example \ref{ex:rsusp2}.  Also, we set $\mathrm{deg}\, x_{i}=i$.  The differentials of $T_D(t)$ are as in (\ref{eq:dtow}); namely, $\partial(x_{t+1})$ is the fundamental class of $(\tilde{\mathbb{R}}^t)^+$:
\[\partial(x_{t+1})=\begin{cases}
(1+j)c_t & \mathrm{if}\; t>0\\
c_0 & \mathrm{if} \; t=0.
\end{cases}\]
The differential of $x_i$ for $i>t+1$ is given by $\partial(x_i)=s(1+j)^2x_{i-2}$.  \end{defn}
\begin{fact}\label{fct:interpdtow}
If $t=0$, the chain complex $T_D(t)$ is the reduced CW complex of the unreduced suspension $\tilde{\Sigma}(S^{2D-1} \amalg S^{2D-1})$, where $S^1$ acts on $S^{2D-1}$ by complex multiplication, and $j$ interchanges the two copies of $S^{2D-1}$ (see Definition \ref{def:unredsusp}).  
\end{fact}
\begin{fact}\label{fct:dtowmanos}
We have $\beta(T_D(t))=t/2$ and $\gamma(T_D(t))=t/2$.
\end{fact}
\begin{proof}
Let $Q$ be the quotient complex $T_D(t)/ T_D(t)^{S^1}$.  By inspection $\partial Q \subseteq (1+j^2)Q$.  Then there is no pair of elements $x_1,x_2 \in T_D(t)$ so that $\partial x_1 = f_t$ and $\partial x_2=(1+j)x_1$.  By (\ref{eq:brawlowerbounds}) and (\ref{eq:crawlowerbounds}), we obtain $b(T_D(t))=c(T_D(t))=t$.  By Definition \ref{def:chainmanoinvars}, we obtain the Fact.
\end{proof}
The motivation for considering the complex $T_D(t)$ is that it is the ``minimal" $\G$-chain complex for a fixed $d$-invariant, as made precise in the following fact.
\begin{fact}\label{fct:dtows} Let $Z$ be a suspensionlike chain complex at level $t$.  Then $d(Z) \geq 2D+t$ if and only if $Z \succeq T_D(t)$.
\end{fact}
\begin{proof}
The Fact follows immediately from Lemma \ref{lem:drules}.
\end{proof}

We also recall the definition of the invariant $\delta$ from \cite{ManolescuPin}, analogous to Definition \ref{def:manolescudefn1}.
\begin{defn}\label{def:manolescudefns1}
For $[(X,m,n)] \in \E$, we set 
\begin{equation}
\delta((X,m,n))=d(C^{CW}_*(X,\p))/2-m/2-2n
\end{equation}
The invariant $\delta$ does not depend on the choice of representative of the class $[(X,m,n)]$.    
\end{defn}   

\begin{prop}\label{prop:dcalc} 
For $X_1,X_2 \in \E$, $\delta(X_1 \otimes X_2) = \delta(X_1) +\delta(X_2)$.
\end{prop}
\begin{proof}
Entirely analogous to the proof of Theorem \ref{thm:easyineq}, we obtain \[\delta(X_1 \otimes X_2) \geq \delta(X_1) +\delta(X_2).\]  Additionally, $\delta(X)=-\delta(X^*)$, as in (\ref{eq:orierev}), where $X^*$ denotes the dual of $X$.  We then obtain:
\[\delta(X_1 \otimes X_2) \leq \delta(X_1) + \delta(X_2),\]
completing the proof.
\end{proof}

We next relate the $\mathrm{Pin}(2)$-invariants to $d$. 

\begin{prop}\label{prop:adc} 
Let $Z$ be a suspensionlike $\G$-chain complex of type SWF.  Then $\alpha(Z) \geq \delta(Z)$.
\end{prop} 
\begin{proof}
We will use the description of $\alpha$ from Lemma \ref{lem:manolowerbounds}.  Recall that $EG$ is the total space of the universal $S^1$-bundle, by forgetting the action of $j \in G$.  Viewed thus, the chains 
\begin{equation}\label{eq:s1gens}
e_0,j((1+j)e_2+se_1),e_4,j((1+j)e_6+se_5),e_8,j((1+j)e_{10}+se_9),e_{12}, \dots
\end{equation}
descend to generators of homology in $BS^1=EG \times_{S^1} \{\p\}$. 

Say $Z$ is at level $t$ and let $f_t$ be the fundamental class of $Z^{S^1}$.  Using (\ref{eq:s1gens}) and repeating the proof of Lemma \ref{lem:manolowerbounds}, $d(Z)$ is the degree of the minimal element of the form
\[ e_{4i} \otimes f_t \; \mathrm{or}\; j((1+j)e_{4i+2}+se_{4i+1})\otimes f_t\]
that is not a boundary in $C^{S^1}_*(Z)=C^{CW}_*(EG)\otimes_{C^{CW}_*(S^1)} Z$.  

That is, $d(Z) \geq 4D+2+t$ if and only if $e_{4D}\otimes f_t$ is a boundary. Further, $d(Z)\geq 4D+t$ if and only if $j((1+j)e_{4D-2}+se_{4D-3})\otimes f_t$ is a boundary.  In particular, if, for some $A\geq 0$, $d(Z)\geq 4A+t-2$, we have $e_{4A-4} \otimes_{C^{CW}_*(S^1)} f_t$ is a boundary.  

However, if \[e_{4A-4} \otimes f_t \in C^{CW}_*(EG) \otimes_{C^{CW}_*(S^1)} Z\] is a boundary, then $e_{4A-4} \otimes f_t \in C^{CW}_*(EG) \otimes_{\G} Z$ is also a boundary.  Thus $a(Z) \geq 4A+t$, and so $a(Z)\geq d(Z)$.  Thus, using Definition \ref{def:chainmanoinvars}, the Proposition follows.  
\end{proof}

\section{Manolescu Invariants of unreduced suspensions}\label{sec:smash}
\subsection{Unreduced Suspensions}\label{subsec:unred}
We draw from \cite{ManolescuPin} the following calculation, which we will use in our application to Seifert fiber spaces.  

\begin{defn}\label{def:unredsusp}
Let $G$ act freely on a finite $G$-CW complex $X$ (not a space of type SWF).  We call
\begin{equation}
\tilde{\Sigma}X =([0,1] \times X) / ((0,x) \sim (0,x') \mathrm{ \;and\; } (1,x) \sim (1,x') \;\mathrm{ for\; all } \; x,x' \in X)
\end{equation}
the unreduced suspension of $X$.  The space $\tilde{\Sigma}X$ obtains a $G$-action by letting $G$ act trivially on the $[0,1]$ factor.  We make $\tilde{\Sigma}X$ into a pointed space by setting $(0,x)$ as the basepoint.  Then $\tilde{\Sigma}X$ is a space of type SWF, since $(\tilde{\Sigma}X)^{S^1}=S^0$, and away from $(\tilde{\Sigma}X)^{S^1}$, $G$ acts freely.  
\end{defn}

For $X$ a free $G$-space, the cone of the inclusion map $(\tilde{\Sigma}X)^{S^1} \rightarrow \tilde{\Sigma}X$ is $\Sigma^{\mathbb{R}}X_+$, where $X_+$ is $X$ with a disjoint basepoint added.  This gives the exact sequence, by taking Borel homology,
\begin{equation}\label{eq:unredseq}\begin{tikzcd}
\dots\arrow{r} &\tilde{H}^G_{*+1}(\Sigma^{\mathbb{R}}X_+) \arrow{r} &\tilde{H}^G_*(S^0) \arrow{r} &\tilde{H}^{G}_*(\tilde{\Sigma}X) \arrow{r} & \dots\end{tikzcd}
\end{equation} 
The term $\tilde{H}_{*+1}^G(\Sigma^\mathbb{R}X_+)$ is isomorphic to $\tilde{H}^G_*(X_+)$ because of suspension-invariance of Borel homology with $\f$-coefficients, from (\ref{eq:suspensioninvar}).  Furthermore, $\tilde{H}^G_*(X_+) \simeq H_*(X/G)$ since $G$ acts freely on $X$.  The exact sequence (\ref{eq:unredseq}) becomes (as an exact sequence of $\f[q,v]/(q^3)$-modules):
\begin{equation}\label{eq:unredseq2}\begin{tikzcd}
\dots \arrow{r} &H_*(X/G) \arrow{r}{\kappa_*} &H_*(BG) \arrow{r} & \tilde{H}^G_*(\tilde{\Sigma}X) \arrow{r} & \dots
\end{tikzcd}\end{equation}
Here $\kappa_*$ is induced from $\kappa: X/G \rightarrow BG$, the classifying space map.  Let $\kappa^d_*$ denote the restriction of $\kappa_*$ to degree $d$.  From the exactness of (\ref{eq:unredseq2}), we have:
\begin{align}\label{eq:asusp}
a(\tilde{\Sigma}X)=&\mathrm{min } \{d \equiv 0 \;\mathrm{mod }\; 4\mid\;\kappa^*_d =0\},
\\ \label{eq:bsusp}
b(\tilde{\Sigma}X)=&\mathrm{min } \{d \equiv 1 \;\mathrm{mod }\; 4\mid\;\kappa^*_d =0\}-1,
\\ \label{eq:csusp}
c(\tilde{\Sigma}X)=&\mathrm{min } \{d \equiv 2\; \mathrm{mod }\; 4\mid\;\kappa^*_d =0\}-2.
\end{align}

\subsection{Smash Products}\label{subsec:projsmash}
In this section we compute the Manolescu invariants for smash products of the form
\begin{equation}\label{eq:targsmash}
\bigwedge^n_{i=1}\tilde{\Sigma}(S^{2\tilde{\delta}_i-1}\amalg S^{2\tilde{\delta}_i-1}).
\end{equation}
This calculation will enable us to find the Manolescu invariants for connected sums of certain Seifert spaces in Section \ref{sec:gauge}.

We will find it convenient to write: 
\[E(x)=2\left\lfloor \frac{x+1}{2} \right\rfloor.\]
\begin{thm}\label{thm:premain}
Fix $\tilde{\delta}_i \in \mathbb{Z}_{\geq 1}$, and $\tilde{\delta}_1 \leq \dots \leq \tilde{\delta}_n$.  Let $X_i=S^{2\tilde{\delta}_i-1}\amalg S^{2\tilde{\delta}_i-1}$ for $i=1,\dots,n$, where $X_i$ has a $G$-action given by $S^1$ acting by complex multiplication on each factor, and $j$ acting by interchanging the sphere factors.  Then:
\begin{align}\label{eq:dcalc}
\delta(\bigwedge^n_{i=1} \tilde{\Sigma}X_i)&=\sum^n_{i=1} \tilde{\delta}_i, \\ \label{eq:acalc}
\alpha(\bigwedge^n_{i=1} \tilde{\Sigma}X_i)&=E(\sum^n_{i=1} \tilde{\delta}_i),\\  \label{eq:bcalc}
\beta(\bigwedge^n_{i=1} \tilde{\Sigma}X_i)&=E(\sum^{n-1}_{i=1} \tilde{\delta}_i),\\ \label{eq:ccalc}
\gamma(\bigwedge^n_{i=1} \tilde{\Sigma}X_i)&=E(\sum^{n-2}_{i=1}  \tilde{\delta}_i), 
\end{align}
\end{thm}

We will use Gysin sequences in the proof of Theorem \ref{thm:premain}; for convenience we record the necessary fact here.  As in \cite{tomDieck}[\S III.2] there exists a Gysin sequence in homology for a $G$-space $X$: 
\begin{equation}\label{eq:gysin}\begin{tikzcd}
H^G_*(X) \arrow{r}{(1+j)\cdot -}& H^{S^1}_*(X) \arrow{r}{\pi_*} &H^G_*(X) \arrow{r}{q\cap -}& H^G_{*-1}(X) \arrow{r} & \dots
\end{tikzcd}
\end{equation}
Here, the map $(1+j)\cdot -$ is the map sending a cycle $[x] \in H^G_*(X)$, with chain representative (not necessarily a cycle) $x \in H^{S^1}_*(X)$, to $[(1+j)x] \in H^{S^1}_*(X)$.  The map $\pi_*$ comes from the quotient $\pi:EG \times_{S^1} X \rightarrow EG \times_G X$.  From (\ref{eq:gysin}), we obtain immediately: 

\begin{fact}\label{fct:gysin1plusj}
Let $[x]\in H^G_*(X)$ so that $(1+j)\cdot[x]=0$.  Then $[x]\in \mathrm{Im}\, q$.
\end{fact}

\emph{Proof of Theorem \ref{thm:premain}.}
We will use the description in Section \ref{subsec:unred} to perform the required calculation.  Let $X=\star^n_{i=1} X_i$, where $\star^n_{i=1}$ denotes the join.  We note  
\begin{equation}\label{eq:join}
\bigwedge^n_{i=1} \tilde{\Sigma} X_i=\tilde{\Sigma} (\star^n_{i=1} X_i).
\end{equation}
Further, for each $i$, label one of the disjoint spheres of $X_i$ by $S_{i,0}$ and the other by $S_{i,1}$.  See Figures \ref{fig:theidean1}, \ref{fig:theidean2}, and \ref{fig:theidean3} for visualization of $X$.  As in the figures, we consider $X$ as if it were a polyhedron, with ``points" the $X_i$ and ``faces" (edges, etc.) the joins of subsets of $\{X_i\}$.
We write
\[F_{(k_1,\dots,k_n)}=\star_{i=1}^n S_{i,k_i},\] 
where $k_i \in \{ 0,1\}$ for all $i \in \{ 1,\dots,n\}$, for the ``face" spanned by $S_{i,k_i}$ (see Figure \ref{fig:theidean3}).  

\begin{figure}
  \scalebox{.1}{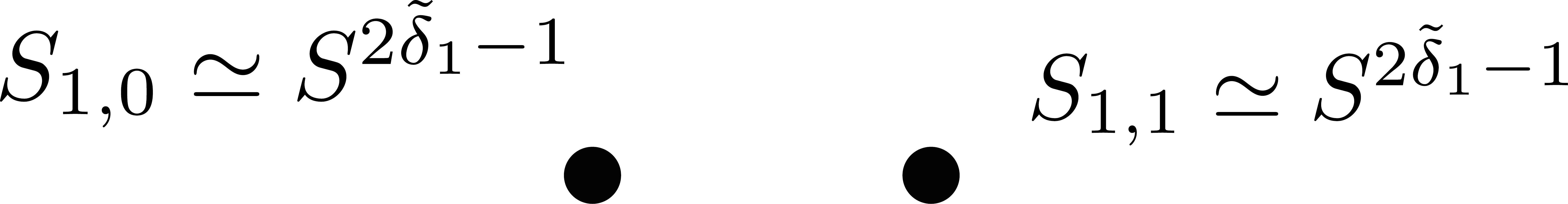}
  \caption{A pictorial representation of $X$, when $n=1$.  }
  \label{fig:theidean1}
\end{figure} 
\begin{figure}
  \scalebox{.1}{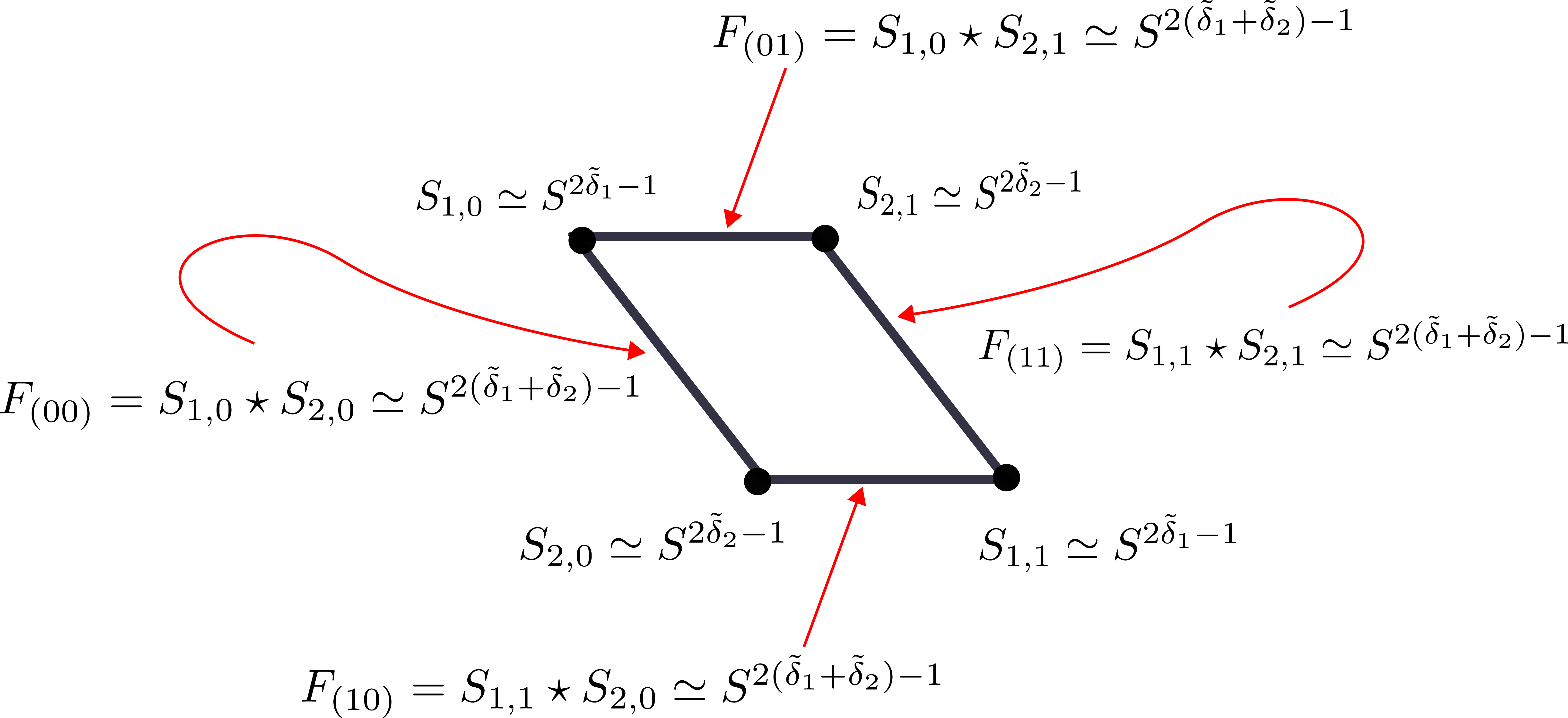}
  \caption{An image of $X$ for $n=2$.}
  \label{fig:theidean2}
\end{figure} 
\begin{figure}
  \scalebox{.1}{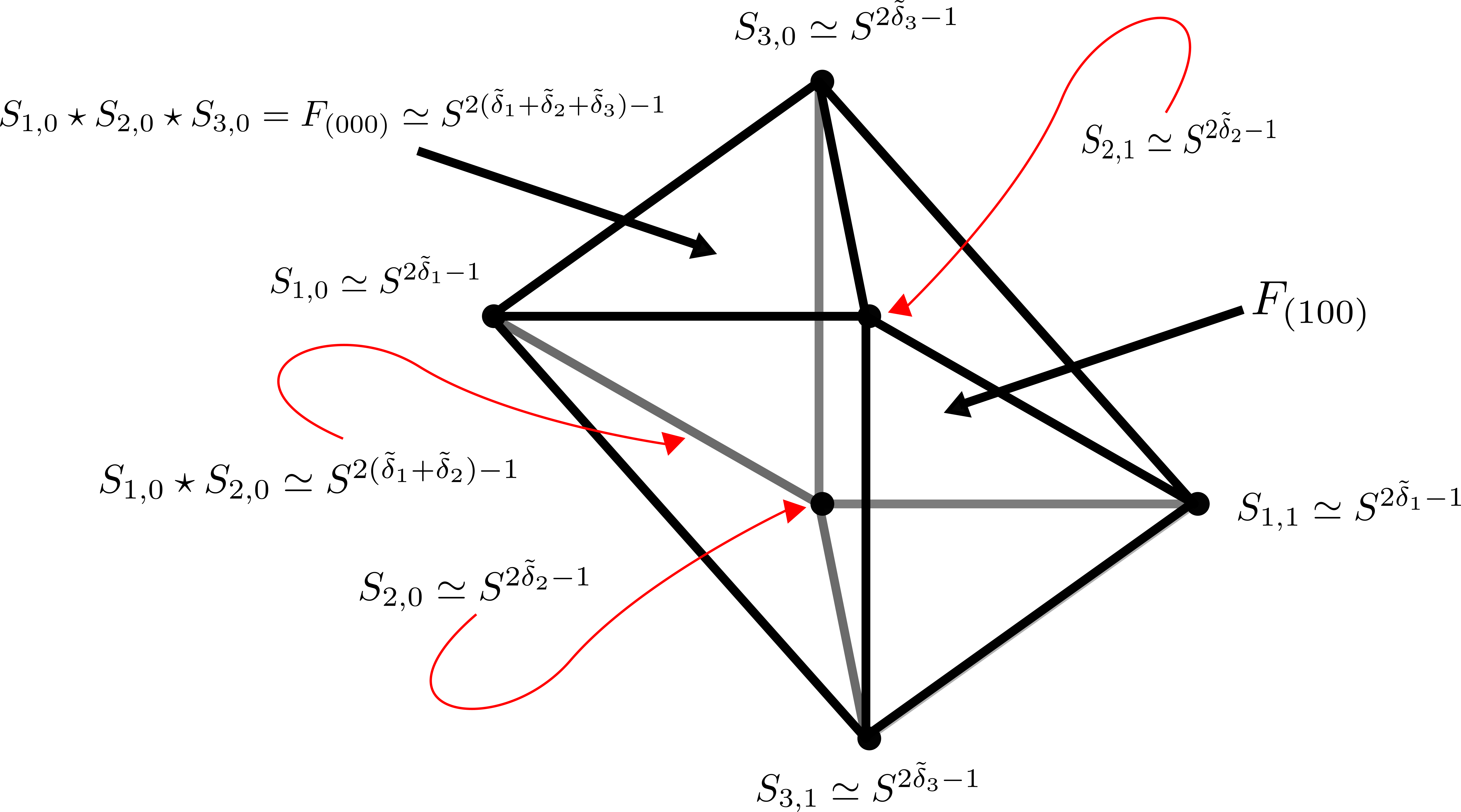}
  \caption{The join $X$ for $n=3$.  Here, we only notate a few of the faces.  }
  \label{fig:theidean3}
\end{figure} 

By Fact \ref{fct:dtows}, $\delta(\tilde{\Sigma}X_i)=\tilde{\delta}_i$.  Proposition \ref{prop:dcalc} then implies (\ref{eq:dcalc}). \medskip

\textbf{Proof of (\ref{eq:acalc}).}  We observe that $S^{2\sum_{i=1}^n \tilde{\delta}_i -1}\simeq \star_{i=1}^n S_{i,0}\subseteq X$, as $S^1$-spaces, where the action on both sides is given by complex multiplication.  We then have a map:
\[S^{2\sum_{i=1}^n \tilde{\delta}_i -1}\amalg S^{2\sum_{i=1}^n \tilde{\delta}_i -1}\rightarrow \star_{i=1}^n S_{i,0}\amalg\star_{i=1}^n S_{i,1}\subseteq X\] of $G$-spaces, where the action of $j$ interchanges the factors of $S^{2\sum_{i=1}^n \tilde{\delta}_i -1}\amalg S^{2\sum_{i=1}^n \tilde{\delta}_i -1}$.  Taking the quotient by the action of $G$ we have a diagram:
\begin{equation}\label{eq:projectivediag}
\begin{tikzcd}
S^{2\sum_{i=1}^n \tilde{\delta}_i -1}\amalg S^{2\sum_{i=1}^n \tilde{\delta}_i -1} \arrow{d} \arrow{r} & X \arrow{d} \\
\mathbb{CP}^{\sum_{i=1}^n\tilde{\delta}_i-1 } \arrow{r}& X/G \arrow{r} & BG
\end{tikzcd}
\end{equation}
with vertical arrows given by $G$-quotient.  The composition $H_*(\mathbb{CP}^{\sum^n_{i=1}\tilde{\delta}_i-1})\rightarrow H_*(BG)$ coming from the second line of (\ref{eq:projectivediag}) is the characteristic class map $\kappa_{*,\mathbb{CP}^{\sum_{i=1}^n\tilde{\delta}_i-1}}$ of $S^{2\sum_{i=1}^n \tilde{\delta}_i -1}\amalg S^{2\sum_{i=1}^n \tilde{\delta}_i -1}$ as a $G$-bundle, so using Fact \ref{fct:bs1bg}, we have:
\[\kappa_{*,\mathbb{CP}^{\sum_{i=1}^n\tilde{\delta}_i-1}}(U^{-
2\lfloor\frac{\sum_{i=1}^n\tilde{\delta}_i-1}{2}\rfloor})=v^{-\lfloor\frac{\sum_{i=1}^n\tilde{\delta}_i-1}{2}\rfloor}.\]  
Here $U^{-i}$, for $i,N\geq 0$, is the unique element of $H_*(\mathbb{CP}^N)$ so that $U^i(U^{-i})=1$, where $1$ is the unique nonzero element of $H_0(\mathbb{CP}^N)$, and similarly $q^{-i},v^{-i}$ are, respectively, the unique elements of $H_*(BG)$ so that $q^i(q^{-i})=1=v^i(v^{-i})$, where $1\in H_0(BG)$ is nonzero.  Then $\mathrm{Im}\, \kappa_*$ must be nonzero in degree $4\lfloor \frac{\sum_{i=1}^n \tilde{\delta}_i-1}{2}\rfloor $, so \[a(\tilde{\Sigma}X) \geq 4 \lfloor\frac{\sum_{i=1}^n\tilde{\delta}_i-1}{2}\rfloor +4.\]
However, $\kappa^d_*$ must be zero in all degrees $d\geq 4 \lfloor\frac{\sum_{i=1}^n\tilde{\delta}_i-1}{2}\rfloor +4$, since $\mathrm{dim}\, X=2\sum_{i=1}^n\tilde{\delta}_i-1$.  Thus, using Definition \ref{def:manolescudefn1}:
\[\alpha(\tilde{\Sigma}X) = E(\sum_{i=1}^n\tilde{\delta}_i),\]
giving (\ref{eq:acalc}).  \medskip

\textbf{Proof of (\ref{eq:bcalc}).}  We have a ($G$-equivariant) map $\phi_{\beta}:S^{2\sum_{i=1}^{n-1}\tilde{\delta}_i-1} \times S^0\rightarrow X$ (where $j$ acts by interchanging the factors $S^{2\sum_{i=1}^{n-1}\tilde{\delta}_i-1}$) given by the inclusion \[\star_{i=1}^{n-1} S_{i,0}\amalg \star_{i=1}^{n-1} S_{i,1}\subseteq \star_{i=1}^n (S_{i,0} \amalg S_{i,1}).\]  We will use the map $\phi_{\beta}$ to find classes in $H_*(BG)$ in the image of $\kappa_*$ in degree congruent to $1 \; \mathrm{mod} \; 4$.  
 
Let \[F^{n-1}=\amalg_{(l_1,\dots,l_{n-1}) \in \mathcal{L}}\star_{i=1}^{n-1}(S_{l_i,0}\amalg S_{l_i,1}), \]
where $\mathcal{L}$ is the set of all $(n-1)$-tuples of distinct elements of $\{1,\dots,n\}$.  In the analogy from the start of the proof, $F^{n-1}$ is the ``$(n-1)$-skeleton" of $X$.  

Note that associated to a linear subspace $\mathbb{C}^{K}\subseteq \mathbb{C}^N$, there is an $S^1$-equivariant submanifold $S^{2K-1}\subseteq S^{2N-1}$.  That is, there is a map from $\mathrm{Gr}(K,N)$, the space of all $K$-planes in $\mathbb{C}^N$, to the space of all submanifolds $S^{2K-1}\subseteq S^{2N-1}$.  We will call an embedded sphere obtained from a linear subspace this way a \emph{linear} sphere.  We also see that the inclusion \begin{equation}\label{eq:firstlinearsphere}\star^{n-1}_{i=1}S_{i,0}\subseteq F_{(0,\dots,0)}\end{equation} corresponds to the inclusion of a linear subspace $\mathbb{C}^{\sum_{i=1}^{n-1}\tilde{\delta}_i}\subset \mathbb{C}^{\sum_{i=1}^n\tilde{\delta}_i}$ (i.e. (\ref{eq:firstlinearsphere}) is linear).   

Since $\mathrm{Gr}(K,N)$ is connected, we see that any two linear spheres $S^{2K-1}\rightarrow F_{(k_1,\dots,k_n)}$, with $K \leq \sum_{i=1}^{n}\tilde{\delta}_i$, are homotopic in $F_{(k_1,\dots,k_n)}$, through linear spheres. 

Further, we note that for any $\star_{i=1}^{n-1}S_{l_i,k_{l_i}}\subset F_{(k_1,\dots,k_n)}$, there exists some linear sphere \begin{equation}\label{eq:spherexi}S\simeq S^{2K-1} \subseteq \star_{i=1}^{n-1}S_{l_i,k_{l_i}},\end{equation} for all $K\leq \sum_{i=1}^{n-1}\tilde{\delta}_i$ (here we have used $\tilde{\delta}_1 \leq \dots \leq \tilde{\delta}_n$).

In particular, fixing $K \leq \sum_{i=1}^{n-1}\tilde{\delta}_i$, we have a linear sphere $S$ as in (\ref{eq:spherexi}).  Then $S$ is $S^1$-equivariantly homotopic (through linear spheres, in $F_{(k_1,\dots,k_n)}$) to a copy of $S^{2K-1}$ in $\star_{i=1}^{n-1}S_{l'_i,k_{l'_i}}$, for any other sequence of integers $1\leq l'_1 < \dots < l'_{n-1}\leq n$.  Inductively then, $S$ is homotopic to a subset of 
\[\star_{i=1}^{n-1}S_{l''_i,k'_{l''_i}},\]
for any sequences $l''\in \mathcal{L},$ and $k'_i\in \{0,1\}$, in $X$.  It follows that there exists a homotopy from $\star_{i=1}^{n-1} S^{2\tilde{\delta}_i-1}_{i,0}$ to $\star_{i=1}^{n-1} S^{2\tilde{\delta}_i-1}_{i,1}=j(\star_{i=1}^{n-1} S^{2\tilde{\delta}_i-1}_{i,0})$ in $X$. See Figures \ref{fig:homotopybeta1} and \ref{fig:homotopybeta2} for illustrations in the $n=2,3$ cases.

\begin{figure}
  \scalebox{.1}{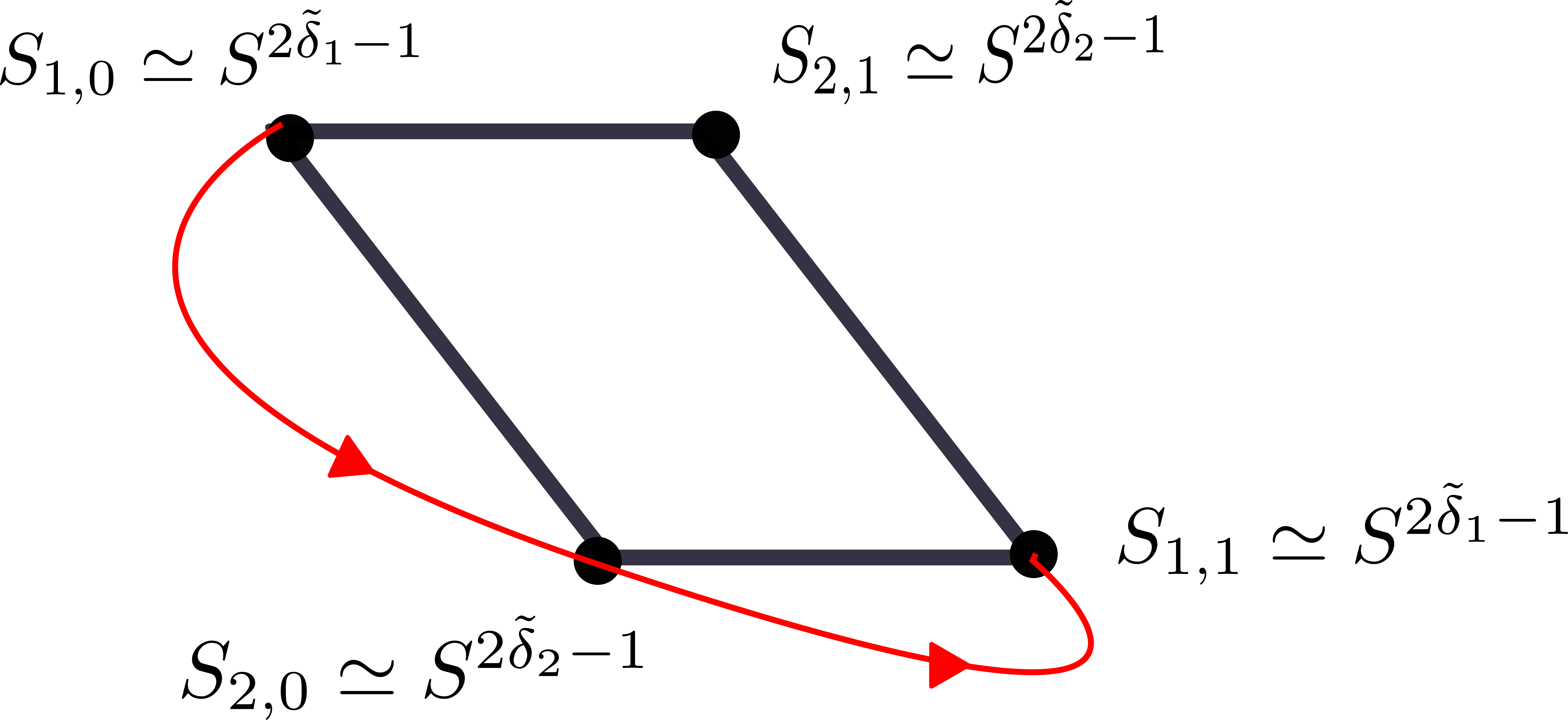}
  \caption{The homotopy from $S_{1,0}$ to $jS_{1,0}$ in the case $n=2$.  The sphere $S_{1,0}$ is homotopic to a copy of $S^{2\tilde{\delta}_1-1}\subseteq S_{2,0}$ in $F_{(00)}\simeq S_{1,0} \star S_{2,0}\simeq S^{2(\tilde{\delta}_1+\tilde{\delta}_2)-1}.$  Furthermore, $S^{2\tilde{\delta}_1 -1}\subseteq S_{2,0}$ is homotopic to $S_{1,1}$ in $F_{(10)}$.  Thus, we have found a homotopy $\star^{n-1}_{i=1} S_{i,0} \rightarrow j(\star^{n-1}_{i=1} S_{i,0} )$ for $n=2$.}
  \label{fig:homotopybeta1}
\end{figure} 
\begin{figure}
  \scalebox{.1}{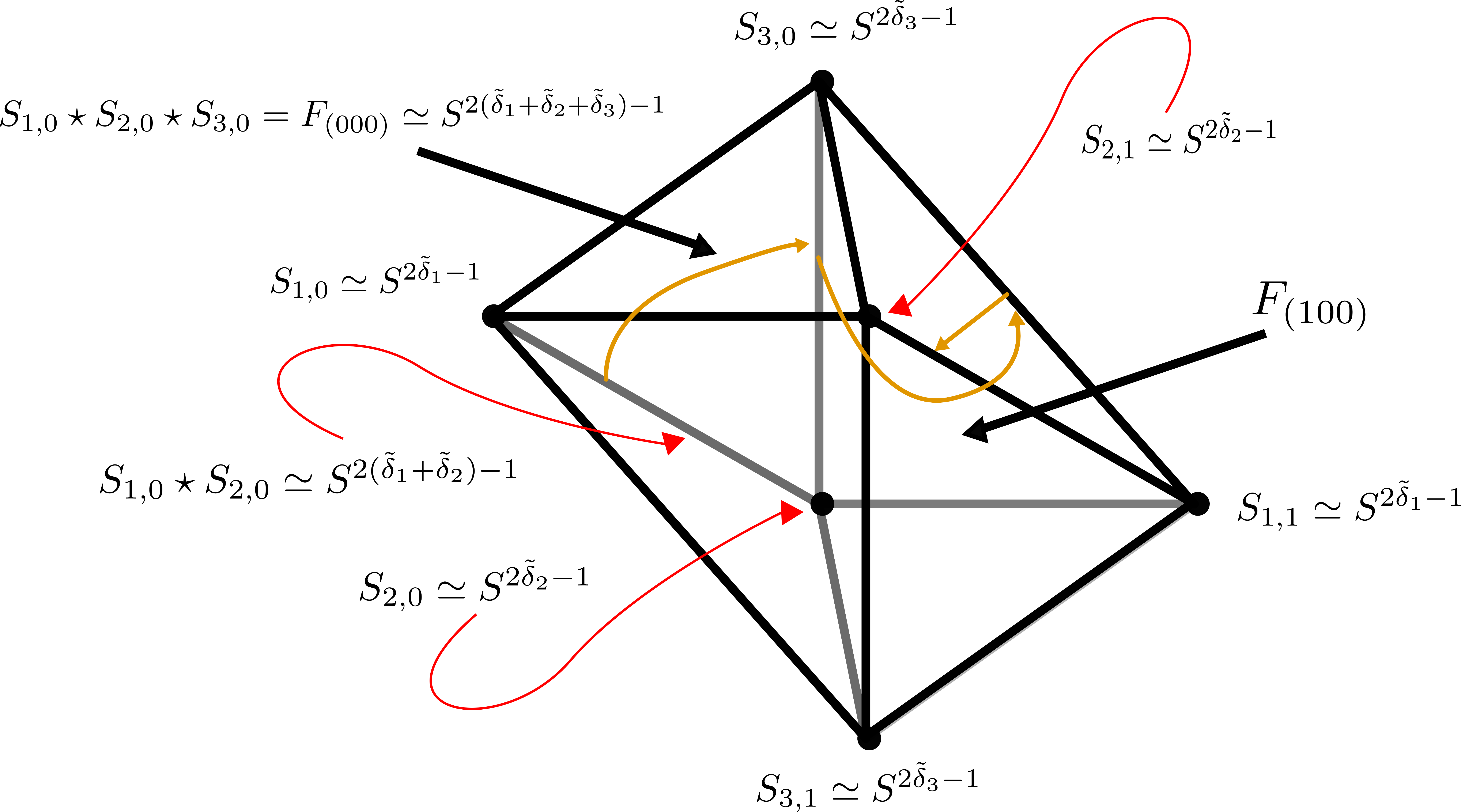}
  \caption{An illustration of the homotopy $ \star^{n-1}_{i=1} S_{i,0} \rightarrow j(\star^{n-1}_{i=1} S_{i,0}), $ for $n=3$. In $F_{(000)}$ we have $S_{1,0}\star S_{2,0} \simeq S^{2(\tilde{\delta}_1+\tilde{\delta}_2)-1}$ homotopic to a copy $S'$ of $S^{2(\tilde{\delta}_1+\tilde{\delta}_2)-1}$ contained in $S_{2,0}\star S_{3,0}$.  The sphere $S'$ is then homotopic in $F_{(100)}$ to $S'' \subseteq S_{1,1} \star S_{3,0}$.  In $F_{(110)}$, $S''$ is homotopic to $S''' \subseteq S_{1,1} \star S_{2,1}$, so we have constructed a homotopy $\star^2_{i=1} S_{i,0} \rightarrow j(\star^2_{i=1}S_{i,0}),$ as needed.  A similar procedure applies for $n \geq 4$.}
  \label{fig:homotopybeta2}
\end{figure} 

Now we take advantage of the Gysin sequence from (\ref{eq:gysin}).  Let $\Phi_\alpha$ denote the fundamental class of the projective space \[\mathbb{CP}^{\sum_{i=1}^{n-1}\tilde{\delta}_i-1}\simeq (S^{2\sum_{i=1}^{n-1}\tilde{\delta}_i-1}\times S^0)/G\simeq (\star_{i=1}^{n-1} S_{i,0})/S^1,\] and let $\iota_*$ denote the map on homology induced by the inclusion:
\[\iota: (\star_{i=1}^{n-1} S_{i,0}\amalg \star_{i=1}^{n-1} S_{i,1})/G \rightarrow X/G.\]
Then, as in the argument proving (\ref{eq:acalc}), we have $\kappa_*(\iota_*(\Phi_\alpha))=v^{-\lfloor \frac{\sum_{i=1}^{n-1}\tilde{\delta}_i-1}{2}\rfloor}$.  

We check that $\iota_*(\Phi_\alpha)$ is in the image of $q$ (for the action of $q$ on $H_*(X/G)$).  Indeed, we have that $(1+j)\cdot \iota_*(\Phi_\alpha)$, viewed as a class of $X/S^1$, is zero by the above homotopy from $\star_{i=1}^{n-1} S_{i,0}$ to $\star_{i=1}^{n-1} S_{i,1}=j(\star_{i=1}^{n-1} S_{i,0})$.  Then, by Fact \ref{fct:gysin1plusj}, $\iota_*(\Phi_\alpha)$ is in the image of $q\cap -$.

Thus, there exists some class $\Phi_\beta^X \in H^G_*(X)$ so that $q\Phi_\beta^X =\iota_*(\Phi_\alpha)$.  It follows that $\kappa_*(\Phi^X_\beta)$ must be nonzero, and we obtain $q^{-1}v^{-\lfloor \frac{\sum_{i=1}^{n-1}\tilde{\delta}_i-1}{2}\rfloor}\in \mathrm{Im}\, \kappa_*$.  Using (\ref{eq:bsusp}), we see:
\begin{equation}\label{eq:blowbound}b(\tilde{\Sigma}X) \geq 2E(\sum_{i=1}^{n-1}\tilde{\delta}_i).\end{equation}
Using the Definition \ref{def:manolescudefn1} of $\beta$, that is:
\begin{equation}\label{eq:betalowbound}
\beta(\tilde{\Sigma}X) \geq E(\sum_{i=1}^{n-1}\tilde{\delta}_i).\end{equation}
By Theorem \ref{thm:easyineq},
\begin{equation}
\begin{array}{lcl}\label{eq:bupbound} \beta(\tilde{\Sigma}X) & \leq & \alpha(\tilde{\Sigma}(\star_{i=1}^{n-1} X_i))+\beta(\tilde{\Sigma}X_n) \\
& \leq & E(\sum_{i=1}^{n-1}\tilde{\delta}_i)+0. 
\end{array}
\end{equation}
Here we have used Fact \ref{fct:dtowmanos} to see $\beta(\tilde{\Sigma}X_n)=0$.  Finally, (\ref{eq:betalowbound}) and (\ref{eq:bupbound}) together imply (\ref{eq:bcalc}).  \medskip

\textbf{Proof of (\ref{eq:ccalc}).}  We again apply the Gysin sequence after constructing a homotopy.  Repeating the argument from (\ref{eq:bcalc}), we construct a homotopy, where $I$ is the unit interval:
\[\psi: I \times S^{2\sum^{n-2}_{i=1} \tilde{\delta}_i-1} \rightarrow X\]
so that $\psi(0,-)$ is a linear sphere:
\[S^{2\sum^{n-2}_{i=1}\tilde{\delta}_i-1}\rightarrow \star_{i=1}^{n-2} S_{i,0}, \]
and so that $\psi(1,-)$ is a linear sphere:
\[S^{2\sum^{n-2}_{i=1}\tilde{\delta}_i-1}\rightarrow \star_{i=1}^{n-2} S_{i,1}=j(\star_{i=1}^{n-2} S_{i,0}). \]
Following the argument of (\ref{eq:bcalc}), we see that we may choose $\psi$ to lie entirely within $F^{n-1}$, the ``$(n-1)$-skeleton" of $X$.
The construction of $\psi$ gives that it is a composition of homotopies in the faces: \[F^{n-1}_{(k_1,\dots,k_n)}=F^{n-1}\cap F_{(k_1,\dots,k_n)},\] so that in each $F_{(k_1,\dots,k_n)}$, $\psi$ is a homotopy through linear spheres. 

We will construct a homotopy from $\psi$ to $j\psi$ (perhaps up to reparameterization in the domain).  Knowing that the homotopy $\psi$ was constructed by combining homotopies in the ``faces" $F_{(k_1,\dots,k_n)}$, we constuct a homotopy from $\psi$ to $j\psi$ by considering homotopies between homotopies in the ``faces".  

Let $S\subseteq \star_{i=1}^{n-2}S_{l_i,k_{l_i}}\subset F_{(k_1,\dots,k_n)}$ where $1\leq l_1 < \dots <l_{n-2} \leq n$, and $S\simeq S^{2K-1}$ for some $K \leq \sum_{i=1}^{n-2} \tilde{\delta}_i$.  Let $\psi'$ be a homotopy, through linear spheres, in $F_{(k_1,\dots,k_n)}$, from $S$ to some $S'\simeq S^{2K-1} \subseteq \star_{i=1}^{n-2}S_{l'_i,k_{l'_i}}$, where $1 \leq l'_1 < \dots < l'_{n-2} \leq n$. 

Let $L_1,\dots, L_m\in\mathcal{L}$ so that
\[(l_1,\dots,l_{n-2})\subset L_1\; \mathrm{and} \; (l'_1,\dots,l'_{n-2})\subset L_m,\] 
and so that $L_i$ and $L_{i+1}$ differ in only one place; see Figure \ref{fig:pyramid}.    
\begin{figure}
\scalebox{.2}{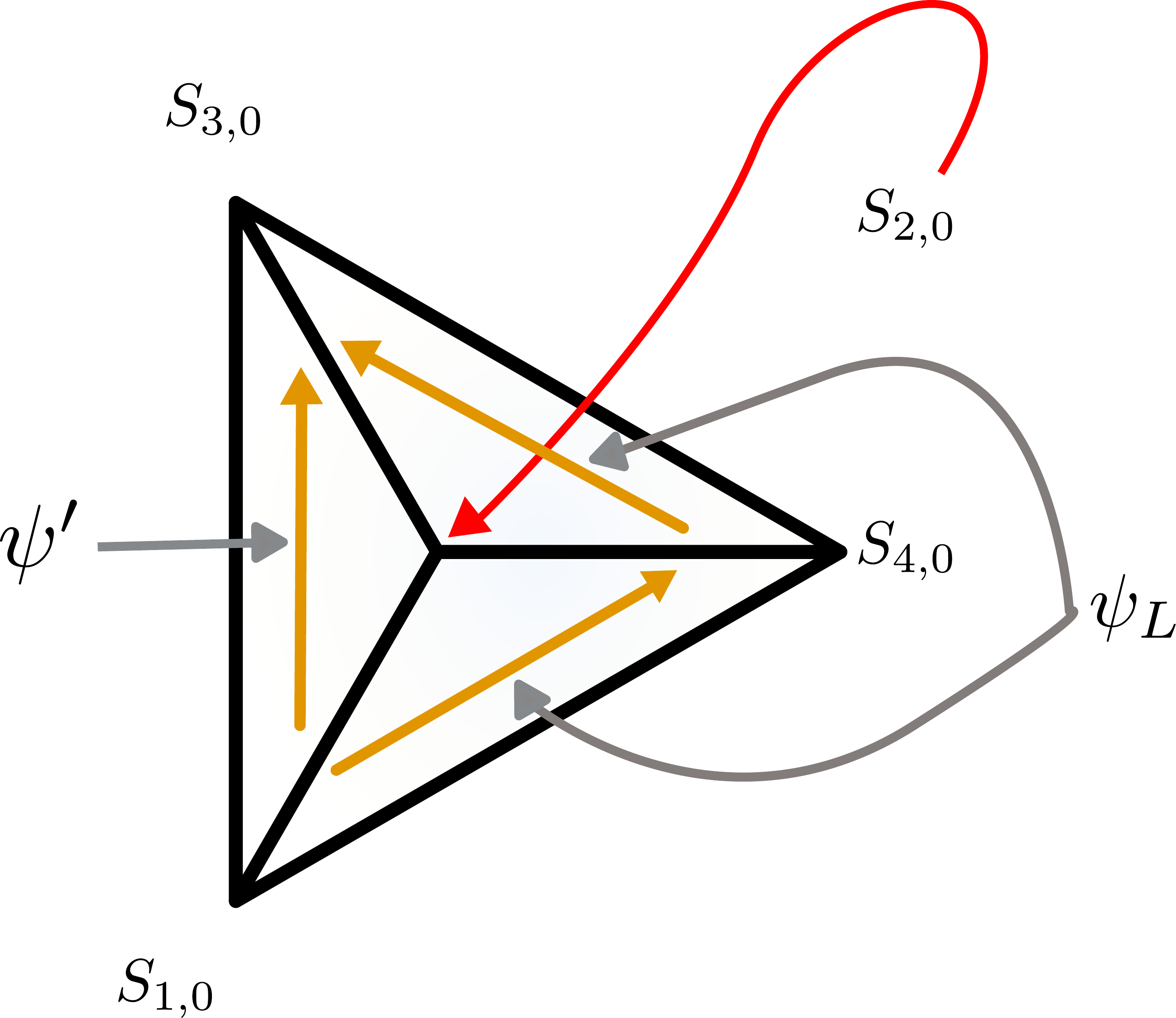}
\caption{The tetrahedron corresponding to the face $F_{(0000)}$, where $n=4$.  In this example, the image of $\psi'$ is contained in $S_{1,0} \star S_{2,0} \star S_{3,0}$, and $\psi'$ takes a sphere in $S_{1,0} \star S_{2,0}$ to a sphere in $S_{2,0} \star S_{3,0}$.  Further, for this example, $L_1=(1,2,4)$, and $L_2=(2,3,4)$.  The path followed by $\psi_L$ is pictured.}
\label{fig:pyramid}
\end{figure}
Then there exists a homotopy: \[\psi_L: I \times S \rightarrow F_{(k_1,\dots,k_n)},\] 
so that $\psi_{L}(0,-)$ is the inclusion of $S$ and $\psi_{L}(1,-)$ is $\psi'(1,-)$, and so that: \[\psi_L([\frac{p-1}{m},\frac{p}{m}],-)\subset \star_{l \in L_p} S_{l,k_l},\] for $1\leq p \leq m$.
The homotopy $\psi_L|_{[\frac{p-1}{m},\frac{p}{m}]\times S}$ is constructed exactly as in the proof of (\ref{eq:bcalc}).

Next, let $\psi^-_L: I \times S\rightarrow F_{(k_1,\dots,k_n)}$ be given by $\psi^-_L(x,y)=\psi_L(1-x,y)$.  

Consider the concatenation \begin{equation}\label{eq:hdef}H=\psi^-_L\ast\psi':I \times S \rightarrow F_{(i_1,\dots,i_n)}\end{equation} obtained by applying $\psi'$ and then running $\psi_L$ backwards.  Since \[\mathrm{Im}\, H(0,-)=\mathrm{Im}\, H(1,-),\] we see that $H$ corresponds to a loop in $\mathrm{Gr}(K-1,\sum_{i=1}^n\tilde{\delta}_i-1)$.  However, $\pi_1(\mathrm{Gr}(K-1,\sum_{i=1}^n\tilde{\delta}_i-1))=1$, from which we see that $H$ is null-homotopic.  That is, $\psi'$ is homotopic to $\psi_L$ (again, perhaps up to reparameterization in the domain), as needed.

As in the proof of (\ref{eq:bcalc}), we compose a sequence of the $\psi'$ to $\psi_L$ homotopies to see that $\psi$ is homotopic to $j\psi$, as in Figure \ref{fig:pyramidpath}.  
\begin{figure}
\includegraphics[scale=.18]{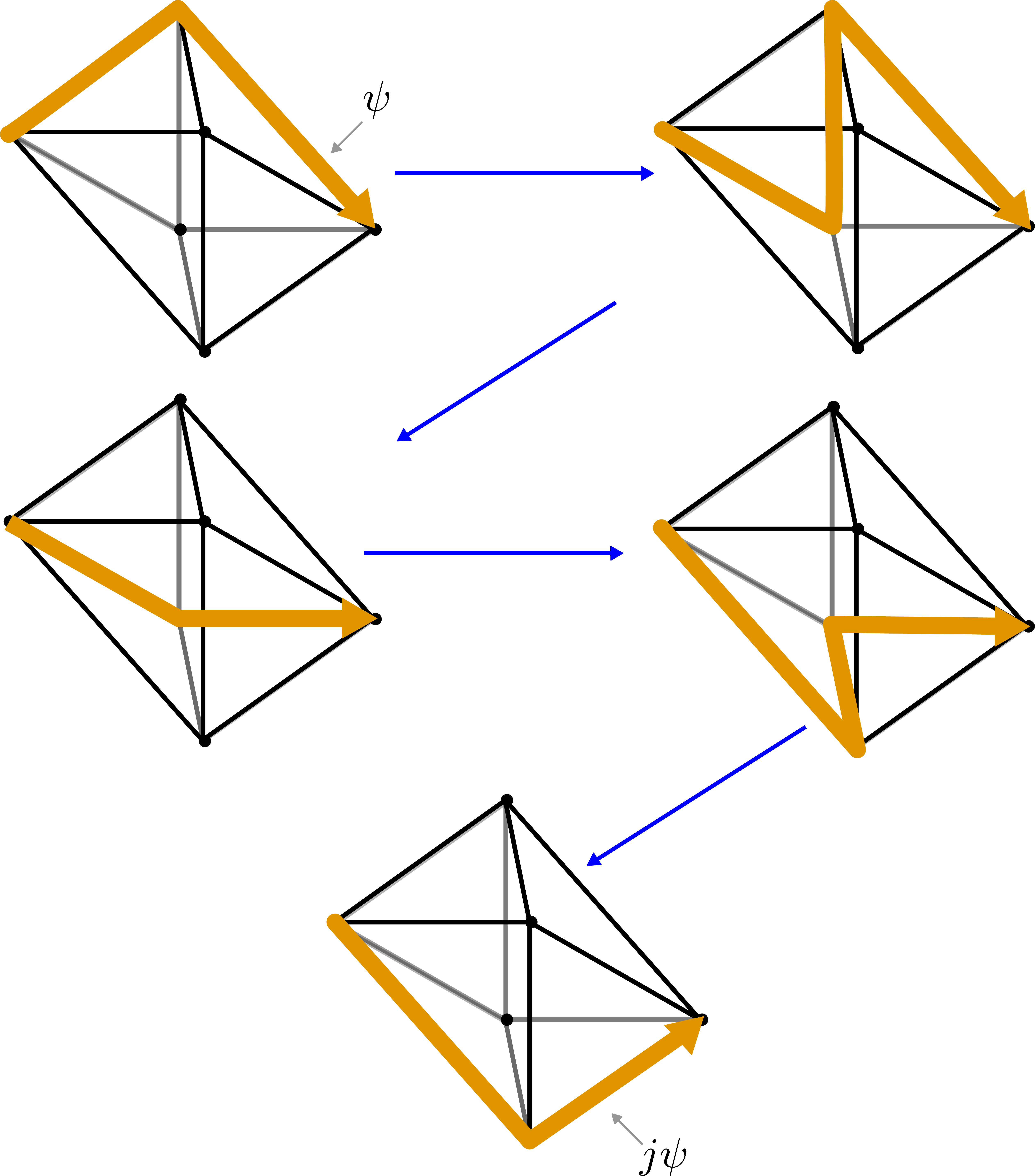}
\caption{A homotopy from $\psi$ to $j\psi$ in the case $n=3$.  Here $\psi$ is a homotopy from $S_{1,0}$ to $jS_{1,0}$, and each stage pictured is one instance of the above construction of $\psi_L$.  Composing these intermediate homotopies in the faces, we have the homotopy between $\psi$ and $j\psi$.} 
\label{fig:pyramidpath}
\end{figure}
Concatenating the reverse $(j\psi)^-$ and $\psi$, we obtain a map:
\[(j\psi)^-\ast \psi: I \times S^{2\sum^{n-2}_{i=1} \tilde{\delta}_i-1}\rightarrow X.\]
Since $\mathrm{Im}\;j\psi(1)=\mathrm{Im}\; \psi(0)$, by reparameterizing the domain $S^{2\sum^{n-2}_{i=1} \tilde{\delta}_i-1}$ we obtain a map:
\[\iota:=(j\psi)^-\ast\psi:S^1 \tilde{\times}S^{2\sum^{n-2}_{i=1} \tilde{\delta}_i-1}\rightarrow X.\]
Here $S^1 \tilde{\times} S^{2\sum^{n-2}_{i=1} \tilde{\delta}_i-1}$ is a space obtained by gluing the ends of $I \tilde{\times} S^{2\sum^{n-2}_{i=1} \tilde{\delta}_i-1}$. 

The map $\iota$ descends to quotients by $S^1$ and $G$ to give maps $\iota_{S^1}$ and $\iota_G$, respectively.  

Now that we have constructed the homotopy between $\psi$ and $j\psi$, we repeat the Gysin sequence argument we have already used in proving (\ref{eq:bcalc}).

Let $\Phi_{\alpha}$ denote the fundamental class of \[(1\times S^{2\sum^{n-2}_{i=1} \tilde{\delta}_i-1})/S^1\simeq \mathbb{CP}^{\sum^{n-2}_{i=1} \tilde{\delta}_i-1}\subseteq (S^1 \tilde{\times} S^{2\sum^{n-2}_{i=1} \tilde{\delta}_i-1})/G. \] 
We have that $(1+j)\cdot \Phi_{\alpha}=0$ as a homology class in $H^{S^1}_*(S^1 \tilde{\times}S^{2\sum^{n-2}_{i=1} \tilde{\delta}_i-1})$, since the homotopy $\psi$ takes $\Phi_{\alpha}$ to $j\Phi_{\alpha}$.  Then $\Phi_{\alpha}$, viewed as a class in $H^{G}_*(S^1 \tilde{\times}S^{2\sum^{n-2}_{i=1} \tilde{\delta}_i-1})$, is in $\mathrm{Im}\; q$.  Let $\Phi_{\beta}$ denote the fundamental class of \[(S^1 \tilde{\times} S^{2\sum^{n-2}_{i=1} \tilde{\delta}_i-1})/G.\]  Then $\Phi_{\beta}\in H^{G}_*(S^1 \tilde{\times}S^{2\sum^{n-2}_{i=1} \tilde{\delta}_i-1})$ is the only class in degree $2\sum^{n-2}_{i=1}\tilde{\delta}_i-1$, so $q\Phi_{\beta}=\Phi_{\alpha}$.  Our next goal will be to show that $(1+j)\cdot\iota_{G,*}\Phi_{\beta}=0$, as a class in $H^{S^1}_*(X)$.  

Note that a chain representative $C$ of $\Phi_{\beta}$ in $(S^1 \tilde{\times}S^{2\sum^{n-2}_{i=1} \tilde{\delta}_i-1})/S^1$ is the relative fundamental class of $([0,\frac{1}{2}]\times S^{2\sum^{n-2}_{i=1} \tilde{\delta}_i-1})/S^1$, as in Figure \ref{fig:fundclass}.  Then we see that $(1+j)\cdot\Phi_{\beta}$ is the fundamental class of $(S^1 \tilde{\times}S^{2\sum^{n-2}_{i=1} \tilde{\delta}_i-1})/S^1$.  It follows that \[0=\iota_{S^1,*}(1+j)\cdot\Phi_{\beta}=(1+j)\cdot\iota_{G,*}\Phi_{\beta},\] since \[\psi([0,1] \times S^{2\sum^{n-2}_{i=1} \tilde{\delta}_i-1}) \; \mathrm{and}\; j\psi([0,1] \times S^{2\sum^{n-2}_{i=1} \tilde{\delta}_i-1})\] are homotopic in $X$.  By Fact \ref{fct:gysin1plusj}, we have $\iota_{G,*}\Phi_{\beta}=q\Phi^X_\gamma$ for some $\Phi^X_\gamma \in H^G_*(X)$.  
\begin{figure}
\scalebox{.2}{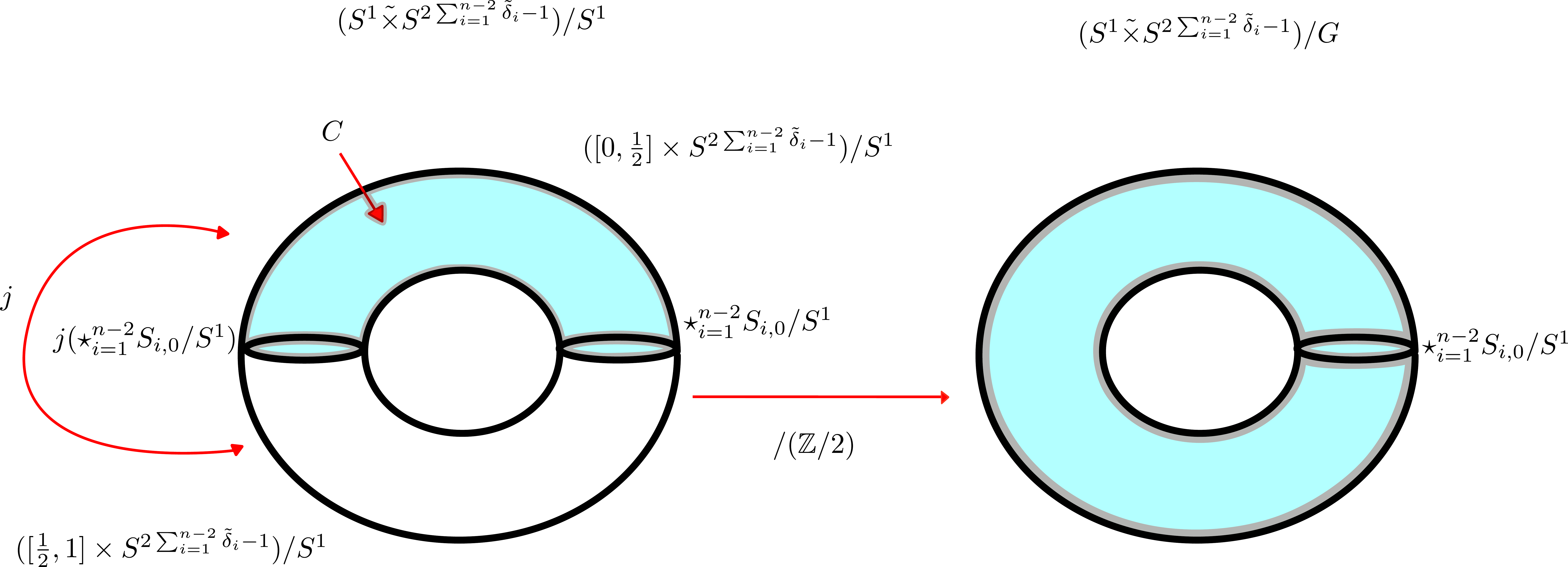}
\caption{The shaded region $C$ denotes the relative fundamental class of $([0,\frac{1}{2}]\times S^{2\sum^{n-2}_{i=1}\tilde{\delta}_i-1})/S^1$, the domain of $\psi$.  We see from the figure that the quotient by the action of $\mathbb{Z}/2=G/S^1$ takes $([0,\frac{1}{2}]\times S^{2\sum^{n-2}_{i=1}\tilde{\delta}_i-1})/S^1)$ surjectively onto $(S^1 \tilde{\times} S^{2\sum^{n-2}_{i=1} \tilde{\delta}_i-1})/G$. Thus $C$ is indeed a chain representative for $\Phi_{\beta}$, as a class in $(S^1 \tilde{\times}S^{2\sum^{n-2}_{i=1} \tilde{\delta}_i-1})/S^1$.}
\label{fig:fundclass}
\end{figure}

As in the argument for (\ref{eq:bcalc}) we note that $\kappa_*\iota_{G,*}(\Phi_\alpha)\neq 0$, since $\kappa_*\iota_{G,*}$ is the characteristic class map for $S^1 \tilde{\times} S^{2\sum^{n-2}_{i=1} \tilde{\delta}_i-1}$ as a $G$-bundle.  Then $\kappa_*\iota_{G,*}(\Phi_\beta)$ is nonzero, because $\kappa_*\iota_{G,*}$ must be $\f[q,v]/(q^3)$-equivariant.  Similarly, we see $\kappa_*(\Phi^X_\gamma)\in H_*(BG)$ must be nonzero, from which we obtain \[q^{-2}v^{-\lfloor \frac{\sum_{i=1}^{n-2}\tilde{\delta}_i-1}{2}\rfloor}\in \mathrm{Im}\, \kappa_*.\]

Thus:
\[c(\tilde{\Sigma}X) \geq 2E(\sum_{i=1}^{n-2}\tilde{\delta}_i),\] so
\begin{equation}\label{eq:clowbound}\gamma(\tilde{\Sigma}X) \geq E(\sum_{i=1}^{n-2}\tilde{\delta}_i).\end{equation}
From Theorem \ref{thm:easyineq}, we have the inequalities (using $0\leq \gamma(\tilde{\Sigma}(X_{n-1} \star X_n))\leq \beta(\tilde{\Sigma}X_{n-1})+\beta(\tilde{\Sigma}X_{n})=0$):
\begin{equation}
\begin{array}{lcl}\label{eq:cupbound} \gamma(\tilde{\Sigma}X) & \leq & \alpha(\tilde{\Sigma}(\star_{i=1}^{n-2} X_i))+\gamma(\tilde{\Sigma}(X_{n-1} \star X_{n})) \\
  & \leq & E(\sum_{i=1}^{n-2}\tilde{\delta}_i)+0. 
  \end{array}
  \end{equation}
  Finally, (\ref{eq:clowbound}) and (\ref{eq:cupbound}) imply (\ref{eq:ccalc}).
\qed

\section{Manolescu Invariants for Connected Sums of Seifert Spaces}\label{sec:gauge} 
First, we recall the output of Manolescu's construction in \cite{ManolescuPin}:

\begin{thm}[Manolescu {\cite{ManolescuPin},\cite{ManolescuK}}]\label{thm:ManolescuSWF}
There is an invariant $\mathit{SWF}(Y,\mathfrak{s})$, the Seiberg-Witten Floer spectrum class, of rational homology three-spheres with spin structure $(Y,\s)$, taking values in $\E$.  A spin cobordism $(W,\mathfrak{t})$, with $b_2(W)=0$, from $Y_1$ to $Y_2$, induces a map $\mathit{SWF}(Y_1, \mathfrak{t}|_{Y_1}) \rightarrow \mathit{SWF}(Y_2,\mathfrak{t}|_{Y_2})$. The induced map is a homotopy-equivalence on $S^1$-fixed-point sets.  
\end{thm}

Manolescu \cite{ManolescuPin} constructs $\mathit{SWF}(Y,\s)$ by using a finite-dimensional approximation to the Seiberg-Witten equations on $Y$.  That is, $\mathit{SWF}(Y,\s)$ is constructed by looking at approximations to the Seiberg-Witten equations on certain finite-dimensional subspaces $V_\lambda$, for $\lambda\in \mathbb{R}$, of the global Coulomb slice.  These approximations define, for each $\lambda$, a flow $\phi_\lambda$ on the vector space $V_\lambda$.  Manolescu \cite{ManolescuS1} shows that there is a certain isolated invariant set associated to the flows $\phi_\lambda$.  

Associated to any isolated invariant set $S$ of a flow $(M,\phi)$ on a manifold $M$, Conley \cite{Conley} constructed a pointed homotopy type $I(S)$, an invariant of the triple $(M,\phi,S)$. 
Manolescu then defines $\mathit{SWF}(Y,\s)$ as (a suitable desuspension of) the Conley index $I(S)$ associated to the isolated invariant set of $(V_\lambda,\phi_\lambda)$.  The approach by finite-dimensional approximation allows one to avoid the need to introduce perturbations or to use the blowup construction of \cite{KM}. 
  
\begin{defn}\label{def:mano3invars} For $(Y,\s)$ a spin rational homology three-sphere, the Manolescu invariants $\alpha(Y,\s),$ $\beta(Y,\s),$ and $\gamma(Y,\s)$ are defined by $\alpha(\mathit{SWF}(Y,\s)),$ $\beta(\mathit{SWF}(Y,\s)),$ and $\gamma(\mathit{SWF}(Y,\s))$, respectively.
\end{defn}

\begin{thm}[\cite{ManolescuPin}]\label{thm:duality}
Let $(Y,\s)$ be a spin rational homology three-sphere, and let $-Y$ denote $Y$ with orientation reversed.  Then
\[\alpha(Y,\s)=-\gamma(-Y,\s), \; \beta(Y,\s)=-\beta(-Y,\s), \; \gamma(Y,\s)=-\alpha(-Y,\s).\]
Furthermore $\delta(Y,\s)=-\delta(-Y,\s)$.  
\end{thm}

We also recall that chain local equivalence classes behave well under connected sum:

\begin{fact}[\cite{betaseifert}]\label{fct:homf}
The map $\theta^H_3 \rightarrow \CEL$ given by $Y \rightarrow [\mathit{SWF}(Y,\s)]_{cl}$ is a homomorphism. 
\end{fact} 

We can now prove Theorems \ref{thm:standineq} and \ref{thm:ineq2} of the Introduction.

\noindent
\emph{Proof of Theorem \ref{thm:standineq}.}
By Definition \ref{def:mano3invars}, $M(Y_1\#Y_2,\s_1\#\s_2)=M(\mathit{SWF}(Y_1\#Y_2,\s_1 \# \s_2))$, where $M$ is any of $\alpha,\beta$ and $\gamma$.  By Fact \ref{fct:homf}, $M(\mathit{SWF}(Y_1\#Y_2,\s_1\#\s_2))=M(\mathit{SWF}(Y_1,\s_1)\wedge \mathit{SWF}(Y_2,\s_2))$.  Theorems \ref{thm:easyineq} and \ref{thm:hardineq} applied to $\mathit{SWF}(Y_1,\s_1)$ and $\mathit{SWF}(Y_2,\s_2)$ yield Theorem \ref{thm:standineq}.  \qed \\

\noindent
\emph{Proof of Theorem \ref{thm:ineq2}.}
It follows from Definition \ref{def:mano3invars} and Proposition \ref{prop:adc} that $\delta(Y,\s) \leq \alpha(Y,\s)$.  The inequality $\gamma(Y,\s) \leq \delta(Y,\s)$ then follows from Theorem \ref{thm:duality}.\qed \\

Next, we specialize to Seifert spaces to acquire Theorem \ref{thm:main} of the Introduction.  First we recall some notation associated with Seifert spaces.

We let the \emph{standard fibered torus of type} $(a,b)$, for integers $a>0,b$, be the mapping torus of the automorphism of the disk $D^2$ given by rotation by $2\pi b/a$.  Let $D^2_{a}$ be the standard disk with orbifold structure where $\mathbb{Z}/a$ acts by rotation by $2\pi/a$; the origin is then an orbifold point, with multiplicity $a$.  The standard fibered torus is naturally a circle bundle over the orbifold $D^2_a$.  

Let $f: Y \rightarrow P$ be a circle bundle over an orbifold $P$, and $x \in P$ an orbifold point of multiplicity $a$.  If a neighborhood of the fiber of $x$ is equivalent, as an orbifold circle bundle, to the standard fibered torus of type $(a,b)$, we say that $Y$ has \emph{local invariant} $b$ at $x$.  

We call a closed three-manifold $Y$ a Seifert fiber space if it may be written as a disjoint union of circles, so that a neighborhood of any circle is homeomorphic to a standard fibered torus.  Such a decomposition gives $Y$ the structure of a circle bundle $Y \rightarrow P$, for $P$ an orbifold surface with singularities coming from the $\mathbb{Z}/a$ action above.  We restrict attention to Seifert spaces with \emph{orientable} base $|P|$ (The desingularization of the $P$)\footnote{There are also Seifert fibered rational homology spheres with base orbifold $\mathbb{RP}^2$, and  some of them do not have a Seifert structure over $S^2$. These are not considered in this paper. None of them are integral homology spheres.  Furthermore, in order for a Seifert fiber space $Y$ to be a rational homology sphere, it must fiber over an orbifold with underlying space either $\mathbb{RP}^2$ or $S^2$.}.  

For $a_i \in \mathbb{Z}_{\geq 1}$, let $S(a_1,\dots,a_k)$ denote the orbifold with underlying space $S^2$ and $k$ orbifold points, with corresponding multiplicities $a_1,\dots,a_k$.  Fix $b_i \in \mathbb{Z}$ with $\gcd(a_i,b_i)=1$ for all $i$.  We let $\Sigma(b,(b_1,a_1),\dots, (b_k,a_k))$ denote the circle bundle over $S(a_1,\dots,a_k)$ with first Chern class $b$ and local invariants $b_i$.  We define the \emph{degree} of the Seifert space $\Sigma(b,(b_1,a_1), \dots, (b_k,a_k))$ by $b+\sum \frac{b_i}{a_i}$.  Finally, we call a space $\Sigma(b,(b_1,a_1),\dots,(b_k,a_k))$ negative (positive) if $b+\sum \frac{b_i}{a_i}$ is negative (positive).  We note that the orientation reversal of a positive space is negative, and vice versa.  The spaces $\Sigma(b,(b_1,a_1), \dots, (b_k,a_k))$ of nonzero degree are rational homology spheres. 

\begin{defn}\label{def:projtype}
We call a negative Seifert rational homology three-sphere with spin structure $(Y,\s)$ \emph{of projective type} if, for some constants $d,n,m,a_i,m_i$,
\begin{equation}\label{eq:hfcond}
\mathit{HF^{+}}(Y,\mathfrak{s})=\bt_{d}\oplus \bt_{-2n+1}(m) \oplus \bigoplus_{i \in I} \bt_{a_i}(m_i)^{\oplus 2}.
\end{equation}
\end{defn}
We focus on Seifert spaces of projective type because their chain local equivalence class is simplest, as given by the following theorem from \cite{betaseifert}.  Its proof is based on the description of the monopole moduli space on Seifert spaces from \cite{MOY}.
\begin{thm}[\cite{betaseifert}]\label{thm:loctyp}
Let $(Y,\s)$ be a negative Seifert rational homology three-sphere of projective type as in (\ref{eq:hfcond}), and $d(Y,\s)$ the Heegaard Floer correction term.  
Then, for some $s \in \mathbb{Q}$,
\begin{equation}\label{eq:chainprojtypeseif}
[\mathit{SWF}(Y,\s)]_{cl}= [(\tilde{\Sigma}(S^{d(Y)+2s-1}\amalg S^{d(Y)+2s-1}),0,s/2)]_{cl}.
\end{equation}
If $Y$ is an integral homology three-sphere, the quantity $s$ is $n=\bar{\mu}(Y)$.  Moreover, for $(Y,\s)$ any negative Seifert integral homology sphere, not necessarily of projective type, \[\alpha(Y)=E(d(Y)/2+\bar{\mu}(Y))-\bar{\mu}(Y),\]
\[\beta(Y)=\gamma(Y)=-\bar{\mu}(Y).\]
\end{thm}
Theorem \ref{thm:loctyp} explains why spaces $Y$ of projective type are thus named; they satisfy: \[(\mathit{SWF}(Y,\s)/(\mathit{SWF}(Y,\s))^{S^1})/G \] is a copy of $\mathbb{CP}^N$ for some $N$.  Applying Theorem \ref{thm:premain}, we obtain Theorem \ref{thm:main} of the Introduction:\\

\noindent
\emph{Proof of Theorem \ref{thm:main}.}
By Theorem \ref{thm:loctyp} and Fact \ref{fct:homf}, we have:
\begin{equation}\label{eq:seifconn}[SWF(Y_1 \# \dots \# Y_n)]_{cl} =[(\wedge_{i=1}^n(\tilde{\Sigma}(S^{2(d(Y_i)/2+\bar{\mu}(Y_i))-1}\amalg S^{2(d(Y_i)/2+\bar{\mu}(Y_i))-1}),0,\bar{\mu}(Y_1 \#\dots\#Y_n)/2)]_{cl}.\end{equation} 
In Theorem \ref{thm:premain}, we computed $\alpha, \beta,$ and $\gamma$ for the right-hand side of (\ref{eq:seifconn}), completing the proof.   
\qed

\section{Applications}\label{sec:app}
We use Theorem \ref{thm:main} to obtain Theorem \ref{thm:nonSeifert} of the Introduction:

\noindent
\emph{Proof of Theorem \ref{thm:nonSeifert}.}  Define $\tilde{\delta}(Y_i)$ by $d(Y_i)/2+\bar{\mu}(Y_i)$.  Assume without loss of generality that $\tilde{\delta}(Y_1) \leq \dots \leq \tilde{\delta}(Y_n)$.  We have, by Theorem \ref{thm:main}: 
\[\beta(Y)-\gamma(Y)=E(\sum^{n-1}_{i=1} \tilde{\delta}(Y_i))-E(\sum^{n-2}_{i=1} \tilde{\delta}(Y_i)).\]
Since we assumed $\tilde{\delta}(Y_i) \geq 2$ for at least two distinct $i$, we have $\tilde{\delta}(Y_{n-1})\geq 2$, so:
\[\beta(Y)-\gamma(Y)\geq 2.\]
Negative Seifert integral homology spheres $Z$ have $\beta(Z)-\gamma(Z)=0$, so $Y$ is not homology cobordant to any negative Seifert integral homology sphere.  

Using Theorem \ref{thm:main} again, we similarly obtain $\alpha(Y) - \beta(Y) \geq 2$.  But positive Seifert spaces have $\alpha(Z)=\beta(Z)$, using Theorems \ref{thm:duality} and \ref{thm:loctyp}.  Thus $Y$ is not homology cobordant to any positive Seifert space, completing the proof.\qed

\begin{defn}\label{def:hsplit}
We call a rational homology three-sphere with spin structure $(Y,\s)$ \emph{$H$-split} if $\alpha(Y,\s)=\beta(Y,\s)=\gamma(Y,\s)$, in analogy to the concept of $K$-split from \cite{ManolescuK}.  We note from Theorem \ref{thm:standineq} that the subset $\theta_{H\text{--}\mathrm{split}}$ of $H$-split homology cobordism classes is a subgroup of $\theta_3^H$.\end{defn}

\begin{lem}\label{lem:deltadeterm}
Let $Y=Y_1 \# \dots \# Y_n$ be a connected sum of negative Seifert integral homology spheres of projective type $Y_i$, with $\tilde{\delta}(Y_1) \leq \dots \leq \tilde{\delta}(Y_n)$.  Then $\tilde{\delta}(Y_n)$ is determined by $[Y]\in \theta^H_3$.  That is, $\tilde{\delta}(Y_n)$ is a homology cobordism invariant of $Y_1 \# \dots \# Y_n$ among connected sums of negative Seifert integral homology spheres of projective type.   
\end{lem}
\begin{proof}  We show how to determine $\tilde{\delta}(Y_n)$ from $Y$.  First, we note that $Y$ is $H$-split if and only if $\tilde{\delta}(Y_n)=0$ using (\ref{eq:aadd})-(\ref{eq:dadd}), so we may assume from now on that $\tilde{\delta}(Y_n) \geq 1$.  Consider $Y \# \Sigma(2,3,11)$ (recalling that $d(\Sigma(2,3,11))=2$, and $\bar{\mu}(\Sigma(2,3,11))=0$).  We have:
\begin{align}\label{eq:ab1}
\alpha(Y)-\beta(Y)&=E(\sum^n_{i=1} \tilde{\delta}(Y_i))-E(\sum^{n-1}_{i=1} \tilde{\delta}(Y_i))
\\ \label{eq:ab2}
\alpha(Y \# \Sigma(2,3,11))-\beta(Y\# \Sigma(2,3,11))&=E(\sum^n_{i=1} \tilde{\delta}(Y_i)+1)-E(\sum^{n-1}_{i=1} \tilde{\delta}(Y_i)+1)
\end{align}  
If $\tilde{\delta}(Y_n)$ is even, then the difference in (\ref{eq:ab1}) is $\tilde{\delta}(Y_n)$, while if $\tilde{\delta}(Y_n)$ is odd, (\ref{eq:ab1}) is $\tilde{\delta}(Y_n)+1$ if $\sum^{n-1}_{i=1} \tilde{\delta}(Y_i)$ is even, or $\tilde{\delta}(Y_n)-1$ otherwise.  If $\tilde{\delta}(Y_n)$ is even, the difference in (\ref{eq:ab2}) is $\tilde{\delta}(Y_n)$, while if $\tilde{\delta}(Y_n)$ is odd, (\ref{eq:ab2}) is $\tilde{\delta}(Y_n)-1$ if $\sum^{n-1}_{i=1} \tilde{\delta}(Y_i)$ is even, or $\tilde{\delta}(Y_n)+1$ otherwise.  

In particular, we observe that $\alpha(Y), \beta(Y),$ $\alpha(Y\# \Sigma(2,3,11))$, and $\beta(Y\# \Sigma(2,3,11))$ determine $\tilde{\delta}(Y_n)$.  \end{proof}
    
We show the existence of a summand of a certain subgroup of the homology cobordism group.  Let $\theta_{SFP}$ denote the subgroup of $\theta^H_3$ generated by negative Seifert spaces of projective type.  

\begin{thm}\label{thm:minisum} Let $\theta_{H\text{--}\mathrm{split},SFP}=\theta_{H\text{--}\mathrm{split}} \cap \theta_{SFP}$.  The group $\theta_{SFP}$ splits into a direct sum
\begin{equation}\label{eq:Seifertsplitting1} \theta_{SFP}=\theta_{H\text{--}\mathrm{split},SFP}\oplus \bigoplus_{\{x>0 \mid \exists Y,\, \tilde{\delta}(Y)=x \} } \mathbb{Z}. \end{equation}
\end{thm}
\begin{proof}
Here the rightmost direct sum runs over all positive $x$ for which there exists a negative Seifert integral homology sphere $Y$ of projective type with $\tilde{\delta}(Y)=x$.  Let $H$ be the free abelian group with generators $e_i$, for each $i\in\mathbb{Z}_{>0}$.  The group $H$ is isomorphic to $\mathbb{Z}^\infty$.  

We define a homomorphism $\psi: \theta_{SFP} \rightarrow H$.  For $Y$ a negative Seifert integral homology sphere of projective type with $\tilde{\delta}(Y)>0$, we define $\psi(Y)=e_{\tilde{\delta}(Y)}$, while if $\tilde{\delta}(Y)=0$, we set $\psi(Y)=0$.  To define $\psi$ on all of $\theta_{SFP}$ we extend linearly.  To establish that $\psi$ is a homomorphism, we need only show that the set (with multiplicity) $\{ \tilde{\delta}(Y_1),\dots,\tilde{\delta}(Y_n) \}$ associated to $Y \sim Y_1 \# \dots \# Y_n$ is indeed a homology cobordism invariant of $Y$, i.e. that it does not depend on how we express $Y$ as a connected sum of Seifert integral homology spheres in $\theta_{SFP}$.  

Say we have an identity in $\theta_{SFP}$ among (not necessarily negative) Seifert spaces of projective type: 
\begin{equation}\label{eq:infg1}
Y_1 \# \dots \# Y_n \sim Z_1 \# \dots \# Z_m.
\end{equation}
We need to show $\sum \psi(Y_i) = \sum \psi(Z_i)$.  To do so, by rearranging (\ref{eq:infg1}) we may assume that all the $Y_i,Z_j$ are negative Seifert spaces.  
We assume without loss of generality that $\tilde{\delta}(Y_1) \leq \dots \leq \tilde{\delta}(Y_n)$ and $\tilde{\delta}(Z_1) \leq \dots \leq \tilde{\delta}(Z_m)$, and that $n \leq m$.  

By Lemma \ref{lem:deltadeterm}, $\tilde{\delta}(Y_n)=\tilde{\delta}(Z_m)$.  By Theorem \ref{thm:loctyp}, 
\[ [\mathit{SWF}(Z_m\# - Y_n)]_{cl}=[(S^0,0,\frac{\bar{\mu}(Z_m)-\bar{\mu}(Y_n)}{2})]_{cl}.\]  Thus, subtracting $Y_n$ from both sides of (\ref{eq:infg1}), we obtain:
\begin{equation}\label{eq:swfbreak} [\mathit{SWF}(Y_1 \# \dots \# Y_{n-1})]_{cl}=[ (\mathit{SWF}(Z_1 \# \dots \# Z_{m-1})) \wedge (S^0,0,\frac{\bar{\mu}(Z_m)-\bar{\mu}(Y_n)}{2})]_{cl}\end{equation}
The right-hand side of (\ref{eq:swfbreak}) is \[ [\mathit{SWF}((\#_{(\bar{\mu}(Y_n)-\bar{\mu}(Z_m))}\Sigma(2,3,5))\# Z_1\# \dots\# Z_{m-1})]_{cl},\]
using $d(\Sigma(2,3,5))=2$ and $\bar{\mu}(\Sigma(2,3,5))=-1$. 

We repeat the use of Lemma \ref{lem:deltadeterm} to find $\tilde{\delta}(Y_{n-i})=\tilde{\delta}(Z_{m-i})$ for all $i \leq n$.  This gives finally that $Z_1 \# \dots \# Z_{m-n}$ must be $H$-split, and so in particular $\tilde{\delta}(Z_i)=0$ for all $i \leq m-n$.  This shows that $\sum_{i=1}^m\psi(Z_i)=\sum_{i=1}^n\psi(Y_i)$, whence $\psi$ is well-defined on $\theta_{SFP}$.  It is clear that $\psi$ is surjective onto the $\bigoplus_{\{x>0 \mid \exists Y,\, \tilde{\delta}(Y)=x \} } \mathbb{Z}$ factor, with kernel $\theta_{H\text{--}\mathrm{split},SFP}$, giving the splitting stated in the Theorem.   \end{proof}
\noindent\emph{Proof of Theorem \ref{thm:furuta}.}
By Theorem \ref{thm:maingrad}, for all $N>0$ there exists some negative Seifert space of projective type $Y$ for which $\tilde{\delta}(Y)=N$.  Theorem \ref{thm:furuta} then follows from Theorem \ref{thm:minisum}.  

However, other generators for \[\bigoplus_{\{x>0 \mid \exists Y,\, \tilde{\delta}(Y)=x \} } \mathbb{Z}\] are easier to find, using results of N{\'e}methi (we use $Y_p$ from Theorem \ref{thm:maingrad} in order to obtain Corollary \ref{cor:knot}).  

We record a different generating set, starting with some notation from \cite{Nem07}.  Let, for relatively prime $p$ and $q$, $\mathcal{S}_{p,q}\subset \mathbb{Z}_{\geq 0}$ denote the semigroup
\[ \mathcal{S}_{p,q}=\{ap+bg \mid (a,b) \in \mathbb{Z}^2_{\geq 0}\},\]
and \[\alpha_i=\# \{ s\not\in \mathcal{S}_{p,q} \mid s>i\}.\]  Also, set \[g=\frac{(p-1)(q-1)}{2}.\]  Then N{\'e}methi \cite{Nem07} shows
\[\mathit{HF}^+(-\Sigma(p,q,pqn+1))=\bt_0\oplus \bt_0(\alpha_{g-1})^{\oplus n}\oplus \bigoplus^{n(g-1)}_{i=1}\bt_{(\lfloor \frac{i}{n}\rfloor+1)(\{\frac{i}{n}\}n+i)}(\alpha_{g-1+\lceil\frac{i}{n}\rceil})^{\oplus 2}.  \]
Reversing orientation, we have:
\[\mathit{HF}^+(\Sigma(p,q,pqn+1))=\bt_0\oplus \bt_{1-2\alpha_{g-1}}(\alpha_{g-1})^{\oplus n}\oplus \bigoplus^{n(g-1)}_{i=1}\bt_{1-(\lfloor \frac{i}{n}\rfloor+1)(\{\frac{i}{n}\}n+i)-2\alpha_{g-1+\lceil \frac{i}{n}\rceil}}(\alpha_{g-1+\lceil\frac{i}{n}\rceil})^{\oplus 2}.\]
Then Definition \ref{def:projtype} implies that $\Sigma(p,q,pqn+1)$ is of projective type, and Theorem \ref{thm:loctyp} gives, for $n$ odd, $\alpha_{g-1}=d(\Sigma(p,q,pqn+1))/2+\bar{\mu}(\Sigma(p,q,pqn+1))$.  

Fixing $p=2$, we note that the complement of $\mathcal{S}_{p,q}$ is precisely $\{s \mid s<q,\; s\; \mathrm{odd}\}$.  We see from the definition of $\alpha_{g-1}$ that $\alpha_{g-1}=\lfloor \frac{q+1}{4}\rfloor$.  We then have that $\{\Sigma(2,q,2q+1)\mid q>1, \; \mathrm{odd} \}$ attains all positive values of $\tilde{\delta}=d/2+\bar{\mu}$.  By Theorem \ref{thm:minisum}, $\Sigma(2,4k+3,8k+7)$ then span a $\mathbb{Z}^\infty$ summand of $\theta_{SFP}$.   
\qed \\

\noindent\emph{Proof of Corollary \ref{cor:hsplitcor}.}
By the calculation in \cite{ManolescuPin}, for all $k\geq 1$, \[d(\Sigma(2,3,12k-1))=2,\;\bar{\mu}(\Sigma(2,3,12k-1)=0,\]\[d(\Sigma(2,3,12k-7)=2,\;\bar{\mu}(\Sigma(2,3,12k-7)=-1.\]   
In particular, by Theorem \ref{thm:loctyp}, $[\Sigma(2,3,12k-7)]_{cl}$ is independent of $k$.  Furthermore,\[ [\Sigma(2,3,12k-7)]\in \theta_{H\text{--}\mathrm{split}}\] for all $k\geq 1$.  However, Furuta \cite{Furuta} shows $\Sigma(2,3,6k-1)$ are linearly independent in $\theta^H_3$.  Then $\{\Sigma(2,3,12k-7)\}_{k\geq 1}$ generates a $\mathbb{Z}^\infty$ subgroup of $\theta_{H\text{--}\mathrm{split}}$, as needed. \qed \\

We establish Theorem \ref{thm:asymp} of the Introduction, using Theorem \ref{thm:premain}.

\noindent
\emph{Proof of Theorem \ref{thm:asymp}.}  
By Fact \ref{fct:dtows}, for $X$ a space of type SWF at level $t$ the complex $C^{CW}_*(X)$ must contain a copy of $T=T_{(d(X)-t)/2}(t)$.  We recall, by Fact \ref{fct:interpdtow}, that $T$ is chain locally equivalent to \[\Sigma^{t\tilde{\mathbb{R}}}\tilde{\Sigma}(S^{d(X)-t-1}\amalg S^{d(X)-t-1}).\]  Theorem \ref{thm:premain} then shows:
\begin{align}\label{eq:asymp1}
a(T^{\otimes_n})&=2E(n(d(X)-t)/2)+nt,\\ \label{eq:asymp2}
b(T^{\otimes_n})&=2E((n-1)(d(X)-t)/2)+nt,\\ \label{eq:asymp3}
c(T^{\otimes_n})&= 2E((n-2)(d(X)-t)/2)+nt.
\end{align}
Let $(X,g,h)=SWF(Y,\s)$, and let $X$ be of type SWF at level $t$.  Then $\delta(Y,\s)=d(X)/2-g/2-2h$.  From \[\bigwedge^n(T_{(d(X)-t)/2}(t),g,h)\leq \bigwedge^n(X,g,h)\] and (\ref{eq:asymp1})-(\ref{eq:asymp3}) we obtain:
\begin{align}\nonumber
\alpha(\bigwedge^n(X,g,h)) &\geq E(n(d(X)-t)/2)+\frac{nt-ng-4nh}{2},\\ \nonumber \beta(\bigwedge^n(X,g,h)) &\geq E((n-1)(d(X)-t)/2)+\frac{nt-ng-4nh}{2},\\ \nonumber \gamma(\bigwedge^n(X,g,h)) &\geq E((n-2)(d(X)-t)/2)+\frac{nt-ng-4nh}{2},\\ \nonumber \delta(\bigwedge^n(X,g,h))&=nd(X)/2-ng/2-2nh.\end{align}
Using $E(x)\geq x$, we see:
\begin{align}\nonumber 
\alpha(\#_n(Y,\s)) &\geq n\delta(Y,\s),\\ \nonumber 
\beta(\#_n(Y,\s)) &\geq (n-1)\delta(Y,\s)+\frac{(t-g-4h)}{2},\\ \label{eq:gammaasymp}
\gamma(\#_n(Y,\s))& \geq (n-2)\delta(Y,\s)+2\frac{(t-g-4h)}{2},\\ \nonumber
\delta(\#_n(Y,\s))&=n\delta(Y,\s).\end{align}
From (\ref{eq:gammaasymp}), we obtain:
\begin{equation}\label{eq:asympineq}
\gamma(\#_n (Y,\s)) \geq n\delta(Y,\s)+C\end{equation}
where $C$ is some constant depending on $Y$ (but not $n$).  
However, by Theorem \ref{thm:ineq2}, $\gamma(\#_n (Y,\s))\leq \delta(\#_n (Y,\s))=n\delta(Y,\s)$, from which we obtain that $\gamma(\#_n (Y,\s)) -n\delta(Y,\s)$ is a bounded function of $n$.  
Using the properties of $\alpha,\beta,$ and $\gamma$ under orientation reversal we find that $\alpha(\#_n (Y,\s)) - n\delta(Y,\s)$ is also a bounded function of $n$.
Since $\gamma(\#_n (Y,\s)) \leq \beta(\#_n (Y,\s)) \leq \alpha(\#_n (Y,\s))$, we also obtain that $\beta(\#_n (Y,\s)) -n\delta(Y,\s)$ is a bounded function of $n$.
\qed

\section{Graded Roots}\label{sec:gr}
In this section we collect the preliminaries needed to show Theorem \ref{thm:submain}.  We use graded roots, which were introduced by N\'{e}methi \cite{Nemethigr} in order to study the Heegaard Floer homology of plumbed manifolds.  The graded roots of Seifert spaces were studied in \cite{CanKarakurt},\cite{KarakurtLidman}.  Our brief introduction to graded roots will follow \cite[\S 4]{HLK} extremely closely.  

\subsection{Definitions}\label{subsec:grdef}
\begin{defn}[\cite{Nemethigr}]\label{def:gr}
A \emph{graded root} consists of a pair $(\Gamma,\chi)$, where $\Gamma$ is an infinite tree, and $\chi:\mathrm{Vert} (\Gamma) \rightarrow \mathbb{Z}$ satisfies the following.  
\begin{itemize}
\item $\chi(u)-\chi(v)=\pm 1$, if $u,v$ are adjacent.  
\item $\chi(u) > \mathrm{min} \{ \chi(v),\chi(w)\}$ if $u$ and $v$ are adjacent and $u$ and $w$ are adjacent.  
\item $\chi$ is bounded below.
\item For all $k\in \mathbb{Z}$, $\chi^{-1}(k)$ is finite.
\item For $k$ sufficiently large, $| \chi^{-1}(k)|=-1$.
\end{itemize}
\end{defn} 
Some examples of graded roots are featured in Figure \ref{fig:examplegr1}.  

\begin{figure} 
\includegraphics[scale=.8]{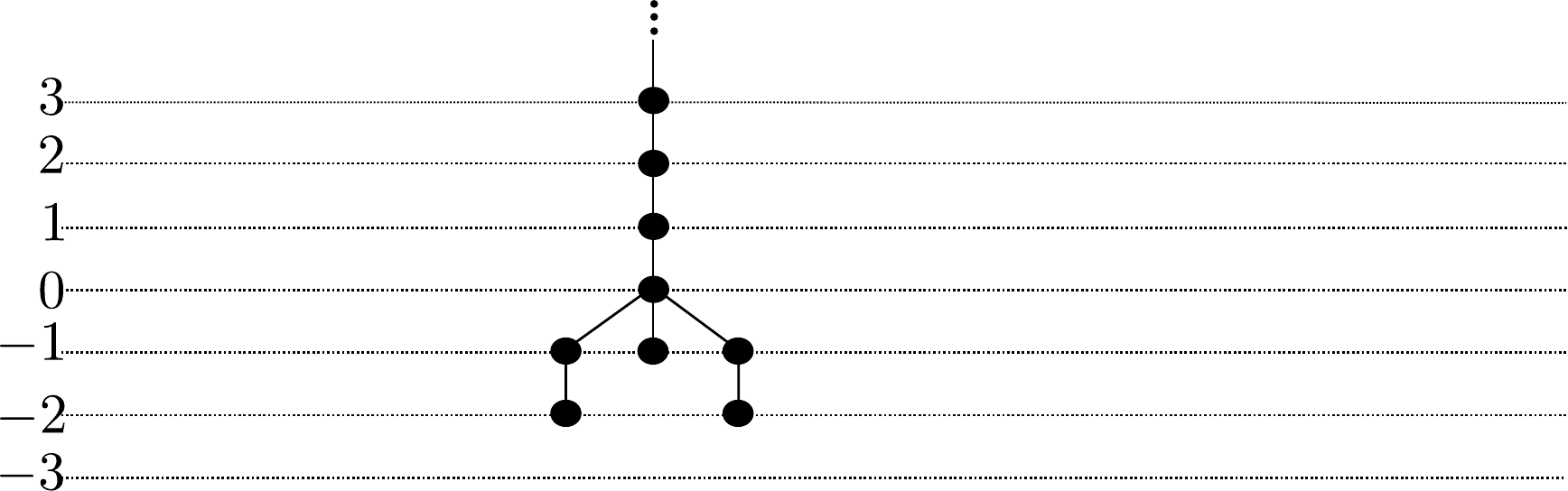}
\caption{Example of a graded root. }
\label{fig:examplegr1}
\end{figure}

Graded roots are specified, up to degree shift, by a finite sequence, as follows.  Let $\Delta: \{ 0, \dots, N\} \rightarrow \mathbb{Z}$, and define $\tau_\Delta : \{ 0, \dots, N\} \rightarrow \mathbb{Z}$ by the recurrence:
\begin{equation}\label{eq:taurec}
\tau_\Delta(n+1)-\tau_{\Delta}(n)=\Delta(n), \; \mathrm{with} \; \tau_\Delta(0)=0.
\end{equation}
For each $n \in \{ 0,\dots, N+1\}$, let $R_n$ be the graph with vertex set $\{\tau(n),\tau(n)+1,\dots\}$, with edges between $k$ and $k+1$ for all $k \geq \tau(n)$.  The graded root associated to $\tau_\Delta$ is the infinite tree obtained by identifying the common edges and vertices of $R_n$ and $R_{n+1}$ for each $n\in \{0,\dots, N+1\}$; call this tree $\Gamma_\Delta$.  We define the grading function $\chi_\Delta$ on $\Gamma_\Delta$ by setting $\chi_\Delta(v)$ to be the integer corresponding to $v$ (this integer is independent of which tree $R_n$ we consider $v$ as a vertex of, by the construction).  Notice that lengthening $\Delta$ by assigning $0$ to $\{ N+1,\dots , M\}$, for some $M>N$ does not change the graded root determined by $\Delta$.  

To a graded root $(\Gamma,\chi)$ is associated a graded $\f[U]$-module $\mathbb{H}(\Gamma,\chi)$.  We define $\hgc$ by the $\f$-vector space with generators the vertices of $\Gamma$.  The element of $\hgc$ corresponding to a vertex $v\in \Gamma$ has grading $2\chi(v)$.  The $\f[U]$-module structure is given by setting $Uv$ to be the sum of all vertices $w$ adjacent to $v$ with $\chi(w)=\chi(v)-1$.  

\subsection{Delta Sequences}\label{subsec:deltas}
Karakurt and Lidman \cite{KarakurtLidman} define an \emph{abstract delta sequence} as a pair $(X,\Delta)$ with $X$ a well-ordered finite set, and $\Delta: X \rightarrow \mathbb{Z}-\{0\}$, with $\Delta$ positive on the minimal element of $X$.  As we saw in \S \ref{subsec:grdef}, an abstract delta sequence specifies a graded root up to a grading shift. 

To connect graded roots back to topology: N\'{e}methi associates a graded root to any manifold belonging to a large family of plumbed manifolds (including Brieskorn spheres).  The corresponding $\f[U]$-module $\hgc$ is isomorphic to $\mathit{HF}^+(-Y)$ up to a grading shift.  Can and Karakurt \cite{CanKarakurt} simplify the method for Seifert homology spheres.  In the proof of Theorem \ref{thm:maingrad} we will use their reformulation.  

In particular, we review the abstract delta sequence $(X_Y,\Delta_Y)$ of an arbitrary Brieskorn sphere $Y=\Sigma(p,q,r)$, following \cite{CanKarakurt}.  Set $N_Y=pqr-pq-pr-qr$.  Let $S_Y$ be the intersection of the semigroup on the generators $pq,pr,qr$ with $[0,N_Y]$.  Set 
\[Q_Y=\{N_Y-s \mid s \in S_Y\},\]
and
\[X_Y=S_Y \cup Q_Y.\]
Can and Karakurt show $S_Y$ and $Q_Y$ are disjoint.  Define $\Delta_Y: X_Y \rightarrow \{-1,1\}$ by $\Delta_Y=1$ on $S_Y$ and $-1$ on $Q_Y$.  It is clear that $(X_Y,\Delta_Y)$ is an abstract delta sequence.  
\begin{thm}[\cite{CanKarakurt} Theorem 1.3,\cite{Nemethigr} Section 11 ,\cite{OzSzplumb} Theorem 1.2]\label{thm:grdroo}
Let $Y=\Sigma(p,q,r)$ for coprime $p,q,r$.  Let $(\Gamma_Y,\chi_Y)$ be the graded root associated to the abstract delta sequence $(X_Y,\Delta_Y)$ described above.  Then $\mathbb{H}(\Gamma_Y,\chi_Y)\cong \mathit{HF}^+(-Y)$ as relatively graded $\f[U]$-modules.  
\end{thm}

Note furthermore that $\Delta_Y(x)=-\Delta_Y(N_Y-x)$ for $x \in X_Y$.  
\subsection{Operations on Delta Sequences}\label{subsec:refmerg}
Different abstract delta sequences may correspond to the same graded root.  For instance, Let $(X,\Delta)$ be an abstract delta sequence.  Fix $t\geq 2$ and $z \in X$ with $|\Delta(z)| \geq t$.  Choose $n_1,\dots, n_t \in \mathbb{Z}$, so that the sign of all $n_i$ is the same as that of $\Delta(z)$ and so that $n_1+\dots+n_t=\Delta(z)$.  From this data we construct an abstract delta sequence wtih the same graded root as $(X,\Delta)$.  Let $X'=X/z \cup\{ z_1,\dots,z_t\}$ for some new elements $z_1\leq \dots \leq z_t$ taking the place of $z$ in $X$.  Define $\Delta': X' \rightarrow \mathbb{Z}$ by $\Delta'(x)=\Delta(x)$ for $x \in X / \{z\}$ and by $\Delta'(z_i)=n_i$ for all $i$.  We call $(X',\Delta')$ a \emph{refinement} of $(X,\Delta)$, and $(X,\Delta)$ a \emph{merge} of $(X',\Delta')$.  
\begin{defn}\label{def:reduced}
We call an abstract delta sequence $(X,\Delta)$ \emph{reduced} if it has no consecutive positive or negative values of $\Delta$ (this is the same as $(X,\Delta)$ not admitting any merges).  Every abstract delta sequence admits a unique reduced form.  We call an abstract delta sequence \emph{expanded} if it does not admit any refinement (this is equivalent to all values of $\Delta$ being $\pm 1$).  
\end{defn}
It is more convenient to work with reduced delta sequences, but we saw in Section \ref{subsec:deltas} that the abstract delta sequence associated to Brieskorn spheres is expanded, so we will need a way to explicitly write the reduced form of $(X_Y,\Delta_Y)$.  This will be handled in Section \ref{sec:semicre} using several lemmas from \cite{HLK}.  

\subsection{Successors and Predecessors}\label{subsec:sucpre}
Let $(X,\Delta)$ be an abstract delta sequence.  Let $S \subset X$ be the set on which $\Delta$ is positive, and $Q \subset X$ the set on which $\Delta$ is negative.  For $x \in X$, we define the positive successor 
\[ \sucp (x) = \mathrm{min}\, \{ x' \in S \mid x < x' \}\] and negative successor $\sucm(x)=\mathrm{min}\, \{ x' \in Q \mid x< x' \}$.  

The sequence $(X,\Delta)$ is reduced if and only if for all $x \in S$:
\[x< \sucm(x) \leq \sucp(x),\]
and, for all $x \in Q$:
\[ x < \sucp(x) \leq \sucm(x).\] 
We also define $\mathrm{pre}_\pm(x)$, the positive and negative \emph{predecessors}, analogously.  

We will need a specific model for the reduced form of $(X,\Delta)$.  First, we need a few further pieces of notation.
For $x\in S$, let \[\pi_+(x)=\mathrm{max} \{ z\in S \mid z < \sucm(x)\} \; \; \mathrm{and} \; \; \pi_-(x) = \mathrm{min}\{ z\in S \mid z> \mathrm{pre}_-(x)\}.\]  For $y \in Q$, let \[\eta_+(y) = \mathrm{max} \{ z \in Q \mid z < \sucp(y) \} \;\; \mathrm{and}\;\;\eta_-(y)=\mathrm{min} \{z \in Q \mid z > \mathrm{pre}_+(y)\}.\] Now define $\tilde{S}=\{ \pi_+(x) \mid x \in S\}$ (noting that $S$ contains one element for each maximal interval of elements of $X$ on which $\Delta$ is positive).  Similarly, define $\tilde{Q} = \{ \eta_-(y) \mid y \in Q\}$.  Then set $\tilde{X} =\tilde{S} \cup \tilde{Q}$.  We define $\tilde{\Delta}$ on $\tilde{S}$ by \[\tilde{\Delta}(\pi_+(x))=\sum_{z \mid \pi_-(x) \leq z \leq \pi_+(x)} \Delta(z),\] and on $\tilde{Q}$ by \[\tilde{\Delta}(\eta_-(y))=\sum_{z \mid \eta_-(y) \leq z \leq \eta_+(y)} \Delta(z).\]  The pair $(\tilde{X},\tilde{\Delta})$ is the reduced form of $(X,\Delta)$.  

Note, in particular, that we may consider $\tilde{X}$ as a subset of $X$.  
\subsection{Tau Functions and Sinking Delta Sequences}\label{subsec:sink}
Let $\suc(x)$ be $\mathrm{min}\{ x' \in X \mid x<x'\}$, and let $x_\mathrm{min}$, $x_\mathrm{max}$ be the minimal and maximal elemetns of $X$.  For an abstract delta sequence $(X,\Delta)$, we define $\tau_\Delta$ as in (\ref{eq:taurec}) by:
\[\tau_\Delta(\suc(x))-\tau_\Delta(x)=\Delta(x),\; \mathrm{with} \; \tau_\Delta(x_\mathrm{min})=0.\]
Let $X^+=X \cup \{x^+\}$ where $x^+=\suc(x_\mathrm{max})$.  The function $\tau_\Delta$ is then defined on $X^+$.  

We call $\tau_\Delta$ the \emph{tau function} associated to the abstract delta sequence $(X,\Delta)$.  

\begin{defn}[\cite{HLK}]\label{def:sink}
Let $(X,\Delta)$ be an abstract delta sequence and $(\tilde{X},\tilde{\Delta})$ its reduced form.  We call $(X,\Delta)$ sinking if the following hold.  
\begin{itemize}
\item The maximal element $x_\mathrm{max}$ of $X$ belongs to $Q$ (i.e. $\Delta(x_\mathrm{max})<0$).
\item For all $x \in \tilde{S}$, $\tilde{\Delta}(x)\leq |\tilde{\Delta}(\sucm(x))|$.
\item $\tilde{\Delta}(\mathrm{pre}_+(x_\mathrm{max}))<|\tilde{\Delta}(x_\mathrm{max})|$.  
\end{itemize}
\end{defn}
Sinking delta sequences will be significant to us because of the following Proposition, which follows immediately from Definition \ref{def:sink}.
\begin{prop}[Proposition 4.7 \cite{HLK}]\label{prop:47}
A sinking delta sequence attains its minimum at and only at its last element.
\end{prop}
\subsection{Symmetric Delta Sequences}\label{subsec:sym}
There is a symmetry in Figure \ref{fig:examplegr1} obtained by reflecting the graded roots across the vertical axis.  This symmetry holds for graded roots of all Seifert integral homology spheres.  For simplicity, write $\Delta=\langle k_1,k_2, \dots, k_n \rangle$ for the function $\Delta:X \rightarrow \mathbb{Z} / \{ 0 \}$, where $X$ is a finite well-ordered set, and $k_1$ is the value of $\Delta$ on the minimal element of $X$, $k_2$ is the value of $\Delta$ on the successor of the minimal element of $X$, and so on.  
\begin{defn}\label{def:symzn}
Let $(X,\Delta)$ be an abstract delta sequence with $\Delta=\langle k_1, \dots, k_n \rangle$.  Define the \emph{symmetrization} of $(X,\Delta)$ by the abstract delta sequence $\Delta^{\mathrm{sym}}=\langle k_1, \dots, k_n, -k_n,\dots,-k_1)$.  We call a delta sequence $\Delta$ symmetric if $\Delta=(\Delta')^{\mathrm{sym}}$ for some delta sequence $\Delta'$.  
\end{defn}
\begin{defn}\label{def:join}
For delta sequences $\Delta_1=\langle k_1, \dots, k_n \rangle$ and $\Delta_2=\langle \ell_1, \dots, \ell_m \rangle$, we define the \emph{join} delta sequence $\Delta_1 \ast \Delta_2$ by 
\[\Delta_1 \ast \Delta_2 = \langle k_1, \dots, k_n,\ell_1,\dots,\ell_m\rangle.\]
\end{defn}

For $\Delta$ a symmetric delta sequence, the $\f[U]$-module $\mathbb{H}(\Gamma_\Delta)$ admits an involution $\iota_\Delta$, given as follows.  The delta sequence $\Delta$ gives a map:   
\[ \Delta: \{ 0,\dots,2n+1\} \rightarrow \mathbb{Z}.\]
Let $\iota: \{0,\dots,2n+2\} \rightarrow \{0,\dots,2n+2\}$ be $\iota(k)=2n+2-k$.  Then $\tau_\Delta$ is $\iota$-equivariant:
\begin{align} \Delta(\iota(k))=& \tau_\Delta(\iota(k+1))-\tau_\Delta(\iota(k)) \\ \nonumber =& \tau_\Delta(2n+2-(k+1))-\tau_\Delta(2n+2-k) \\ \nonumber =& -(\tau_\Delta(2n+2-k)-\tau_\Delta(2n+1-k) ) \\ \nonumber =&-\Delta(2n+1-k) \\ \nonumber =& \Delta(k).\end{align}
where in the last equality we have used that $\Delta$ is symmetric.  We may then define $\iota_\Delta$ on each of the $R_{\tau_\Delta(k)}$ by acting as the identity map:
\[\iota_\Delta:R_{\tau_\Delta(k)}\rightarrow R_{\tau_\Delta(\iota(k))}.\]
Then $\iota_\Delta$ induces an involution of $\Gamma_{\Delta}$, and so also of $\mathbb{H}(\Gamma_{\Delta})$, as an $\f[U]$-module.  

We use the definition of symmetrization for delta sequences to further specify the form of the abstract delta sequence (and its reduction) associated to Brieskorn spheres.  

Since $x \in S_Y$ if and only if $N_Y-x \in Q$ (so, in particular, $\Delta_Y(x)=-\Delta_Y(N_Y-x)$), we have $N_Y/2 \not\in X_Y$, and 
\begin{equation}\label{eq:briessym}
\Delta_Y=(\Delta_Y |_{[0,N_Y/2]})^{\mathrm{sym}}.
\end{equation}
We also need a version of (\ref{eq:briessym}) for the reduction.  By $\Delta_Y(x)=-\Delta_Y(N_y-x)$, if the maximal element of $X_Y \cap [0,N_Y/2]$ is in $S_Y$ (respectively $Q_Y$), then the minimal element of $X_Y \cap [N_Y/2,N_Y]$ is in $Q_Y$ ($S_Y$).  Then
\begin{equation}\label{eq:redbriess}
\tilde{\Delta}_Y=(\tilde{\Delta}_Y|_{[0,N_Y/2]})^\sym.
\end{equation}
\section{Semigroups and Creatures}\label{sec:semicre}
In this section we will prove Theorem \ref{thm:maingrad}, and so Theorem \ref{thm:submain}.  First, we will introduce the \emph{creatures} from \cite{HLK} and write their delta sequences.  Then we will prove a technical decomposition result (Lemma \ref{lem:main}) for the graded roots of the Brieskorn spheres $\Sigma(p,2p-1,2p+1)$, for $p$ odd, using several auxiliary lemmas collected from \cite{HLK}.  The proof of Lemma \ref{lem:main} is adapted from the proof for $p$ even from \cite{HLK}, with only minor changes.  We will then verify that $\Sigma(p,2p-1,2p+1)$ is of projective type, and calculate its $\beta$ and $d$.  As in Section \ref{sec:gr}, we will be following \cite{HLK} extremely closely.  

\subsection{Creatures}\label{subsec:crea} 
Hom, Karakurt, and Lidman \cite{HLK} observe via examples that there are certain sub-graded roots occuring in $\Sigma(p,2p-1,2p+1)$, as shown in Figure \ref{fig:creature}.  The two graded roots $\gcp$ in Figure \ref{fig:creature} are both called \emph{creatures}.
\begin{figure}
\includegraphics[scale=.8]{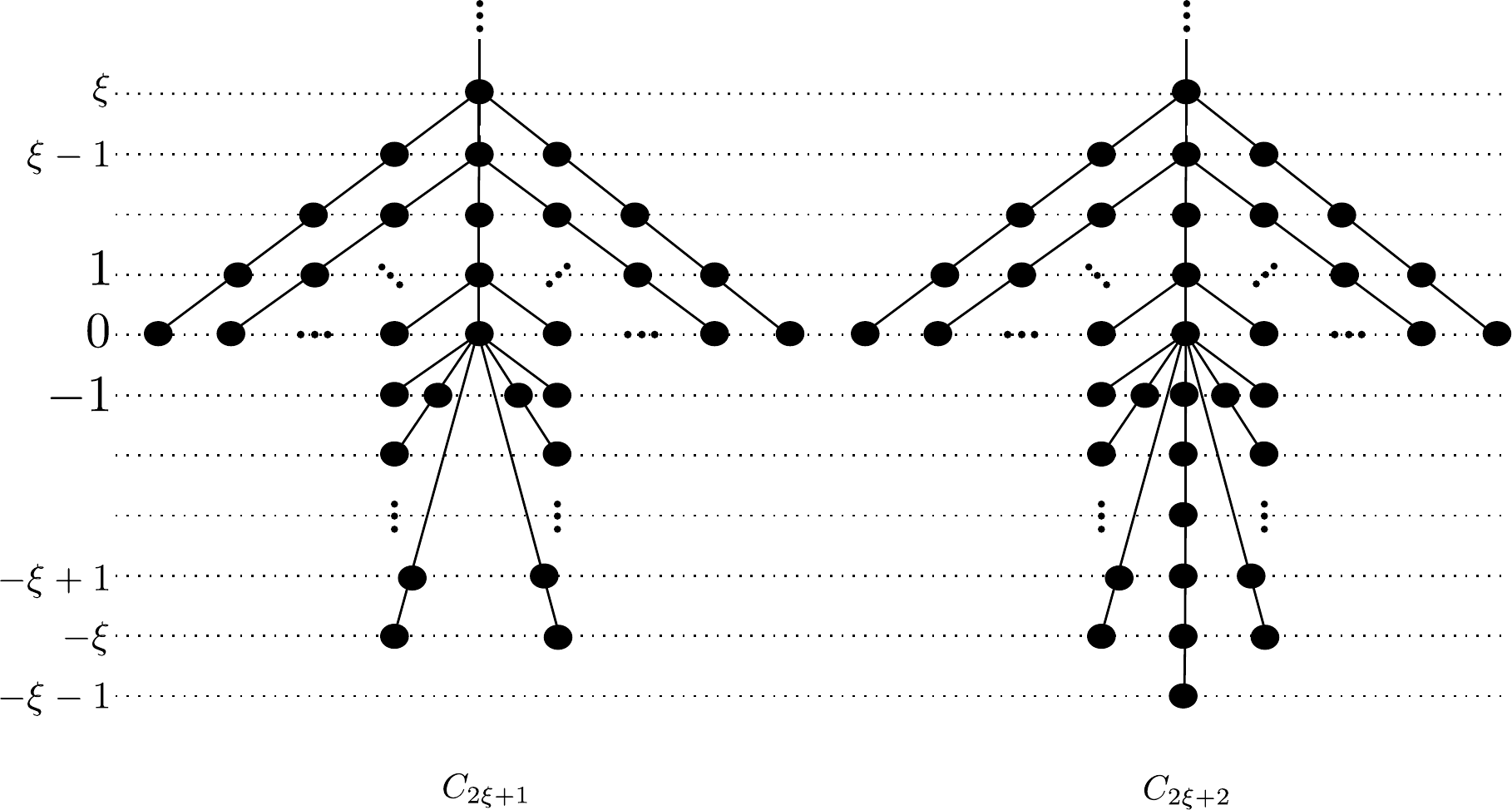}
\caption{Creatures $\Gamma_{C_p}$. Based on Figure 5 of \cite{HLK}.}
\label{fig:creature}
\end{figure}

The abstract delta sequence for the creature $\gcp$ for $p=2\xi+2$, $\xi\in \mathbb{Z}_{\geq 1}$ is the symmetrization of
\[\dcp = \langle \xi ,-\xi,(\xi-1),-(\xi-1),\dots,2,-2,1,-2,1,-2,2,\dots,-(\xi-1),\xi-1,-\xi,\xi,-(\xi+1)\rangle,\]
as observed in \cite{HLK}.  
\begin{defn}\label{def:crea}
For every $p=2\xi+1$, with $\xi\in \mathbb{Z}_{\geq 1}$, the creature $\gcp$ is the graded root defined by the symmetrization of the abstract delta sequence:
\begin{equation}\label{eq:creadef}
\dcp = \langle \xi ,-\xi,(\xi-1),-(\xi-1),\dots,2,-2,1,-2,1,-2,2,\dots,-(\xi-1),\xi-1,-\xi,\xi\rangle.
\end{equation}
\end{defn}
Set $Y_p=\Sigma(p,2p-1,2p+1)$, and $\dyp$ the abstract delta sequence corresponding to $Y_p$, with reduced form $\tilde{\Delta}_{Y_p}$.  We have the following technical lemma, the analogue of \cite{HLK}[Lemma 5.3].  
\begin{lem}\label{lem:main}
For every odd integer $p \geq 3$, we have the decomposition:
\begin{equation}\label{eq:maingrdecomp}
\tdyp=(\Delta_{Z_p} \ast \dcp)^\sym,
\end{equation}
where $\Delta_{Z_p}$ is a sinking delta sequence.
\end{lem}
The proof of Lemma \ref{lem:main} is built from the proof of \cite{HLK}[Lemma 5.3].  In particular, we list several auxiliary Lemmas from \cite{HLK} without proof.  

Set $r_{\pm}=p(2p\pm1)$ and $w=(2p+1)(2p-1)$.  We work with the semigroup $S(r_-,r_+,w)$ on the generators $r_-,r_+,$ and $w$ in studying the graded root associated to $Y_p$.  

\begin{lem}[\cite{HLK} Lemma 5.4]\label{lem:54}
Let $S(r_-,r_+)$ be the semigroup generated by $r_-$ and $r_+$.  The intersection $S(r_-,r_+)\cap[0,(p-1)r_+]$, as an ordered set, is given by:
\begin{multline}\label{eq:rminplussemi}
\{ 0, \\ r_-,r_+,\\2r_-,r_-+r_+,2r_+,\\3r_-,2r_-+r_+,r_-+2r_+,3r_+,\\ \vdots  \\ (p-1)r_-,(p-2)r_-+r_+,\dots,(p-1)r_+\}.
\end{multline} 
\end{lem}

\begin{lem}[\cite{HLK} Lemma 5.6]\label{lem:56} 
Say that $x \in S_{Y_p}$ is of the form $x=ar_-+br_+$, with $a,b \geq 0$, and $x \leq 2r_-+(p-3)r_+$.  Then, 
\begin{enumerate}
\item $x< N_{Y_p}-(p-a-1)r_--(p-b-3)r_+ < \sucp(x)$.
\item $[\pi_-(x),\pi_+(x)]\cap S_{Y_p} = \{ x-\mathrm{min}\{a,b\},\dots,x\}$, unless $x=(p-2)r_+$ or $(p-1)r_-$.  In either of these exceptional cases, $[\pi_-(x),\pi_+(x)]\cap S_{Y_p}= \{(p-2)r_+,(p-1)r_-\}$.  
\end{enumerate}
\end{lem}

\begin{lem}[\cite{HLK} Proposition 5.7 ]\label{lem:57}
The reduced form $\tilde{\Delta}_{Y_p}$ of $\dyp$ satisfies:
\begin{enumerate}
\item As ordered subsets of $\mathbb{N}$, $\tilde{S}_{Y_p}\cap[0,2r_-+(p-3)r_+]=S(r_-,r_+) \cap [0,2r_-+(p-3)r_+]\backslash \{(p-2)r_+\}$.
\item Let $x \in S(r_-,r_+)\cap [0,2r_-+(p-3)r_+]\backslash \{ (p-2)r_+,(p-1)r_-\}$ be written $x=ar_-+br_+$.  Then $\tdyp (x)=\mathrm{min}\{a,b\}+1$.  Further, $\tdyp ((p-1)r_-)=2$.
\item Let $x \in \tilde{S}_{Y_p}$ and say $x< N_{Y_p}-cr_--dr_+< \sucp(x)$, where $c,d\geq 0$.  Then $\tdyp(\sucm(x))\leq -\mathrm{min}\{c,d\}-1$.  
\end{enumerate}
\end{lem}

Fix $p=2\xi+1$ for a positive integer $\xi$.  Define\begin{equation}\label{eq:kdef}
K=(\xi-1)r_-+(\xi-1)r_+.
\end{equation}
We note two inequalities:
\begin{equation}\label{eq:53}
(p-1)r_-+(p-3)r_+<N_{Y_p},
\end{equation}
\begin{equation}\label{eq:54}
(p-2)r_-+(p-2)r_+>N_{Y_p}.
\end{equation}
Note
\begin{equation}\label{eq:nypk}
K<(p-3)r_+<N_{Y_p}/2,
\end{equation} 
by (\ref{eq:54}).  By (\ref{eq:redbriess}),
\begin{equation}\label{eq:58}
\tdyp=(\tdyp |_{\tilde{X}_{Y_p}\cap [0,K)} \ast \tdyp |_{\tilde{X}_{Y_p}\cap[K,N_{Y_p}/2]})^{\mathrm{sym}}.
\end{equation}
Let $S(r_-,r_+)$ be the semigroup generated by $r_-,r_+$.  Observe that $K\in S(r_-,r_+)\cap[0,2r_-+(p-3)r_+]$ and $K \neq (p-2)r_+$, so $K\in \tilde{S}_{Y_p}$ by Lemma \ref{lem:57}.  Set:
\begin{align}\label{eq:59}
\Delta_{Z_p}&=\tdyp|_{\tilde{X}_{Y_p} \cap[0,K)}\\ \label{eq:510}
\Delta_{W_p}&=\tdyp|_{\tilde{X}_{Y_p}\cap[K,\ny/2]}. 
\end{align}
\begin{lem}[cf. Lemma 5.8 of \cite{HLK}]\label{ref:lem:58}
The abstract delta sequence $\dz$ is sinking.  
\end{lem}
\begin{proof}
We check (\lowerromannumeral{1})-(\lowerromannumeral{3}) of Definition \ref{def:sink}.  For (\lowerromannumeral{1}), we recall that $\dz$ is in reduced form.  We saw above that $K \in \tilde{S}_{Y_p}$, so if the last element of the delta sequence $\dz$ were positive, $\tdyp$ would have two consecutive positive values, contradicting that $\tdyp$ is reduced.  This establishes (\lowerromannumeral{1}) in Definition \ref{def:sink}.  

As in \cite{HLK}, we denote predecessors and successors taken with respect to $\tx$ with a tilde, and those with respect to $X_{Y_p}$ without a tilde.  By the construction of the reduced delta sequence as in Section \ref{subsec:sucpre}, 
\begin{equation}\label{eq:511}
\sucp(x) \leq \widetilde{\mathrm{suc}}_+(x) \; \mathrm{for \; every} \; x \in \tx.
\end{equation}
We will next show:
\begin{equation}\label{eq:512}
\tdyp(x) \leq -\tdyp(\widetilde{\mathrm{suc}}_-(x)) \; \mathrm{for \; all} \; x \in \tilde{S}_{Y_p}\cap[0,K),
\end{equation}
to establish (\lowerromannumeral{2}) of Definition \ref{def:sink}.  
Let $x \in \tilde{S}_{Y_p}\cap[0,K)$.  Then $x \in S(r_-,r_+)\cap[0,(p-3)r_+]$ by (\ref{eq:nypk}) and Lemma \ref{lem:57}(\lowerromannumeral{1}).  Writing $x=ar_-+br_+$, Lemma \ref{lem:57}(\lowerromannumeral{2}) gives $\tdyp(x)=\mathrm{min}\{a,b\}+1$.  Set 
\[y=(p-a-1)r_-+(p-b-3)r_+.\]
Lemma \ref{lem:56} and (\ref{eq:511}) give:
\[x<\ny-y<\sucp(x)\leq \widetilde{\mathrm{suc}}_+(x).\]
By $x\in S(r_-,r_+)\cap [0,(p-3)r_+]$, we see that $a+b\leq p-3$.  Thus $p-a-1\geq 0$ and $p-b-3\geq 0$.  Then, by the definition of $Q_{Y_p}$, $\ny -y \in Q_{Y_p}$.  Lemma \ref{lem:57}(\lowerromannumeral{3}) gives 
\[\tdyp(\widetilde{\mathrm{suc}}_-(x))\leq -\mathrm{min}\{p-a-1,p-b-3\}-1. \]
Then, to prove (\ref{eq:512}) it is sufficient to show 
\begin{equation}\label{eq:513}
\mathrm{min}\{a,b\} \leq \mathrm{min}\{p-a-1,p-b-3\}.
\end{equation}
But $a+b\leq p-3$, so $a\leq p-b-3$ and $b\leq p-a-3$, showing (\ref{eq:513}).  

We must check that Definition \ref{def:sink}(\lowerromannumeral{3}) holds for $\dz$.  The last positive value of $\dz$ occurs at $\widetilde{\mathrm{pre}}_+(K)=\xi r_-+(\xi-2)r_+$ by Lemma \ref{lem:54} and Lemma \ref{lem:57}(\lowerromannumeral{1}).  Thus $\widetilde{\mathrm{suc}}_-(\xi r_-+(\xi-2)r_+)$ is the largest element of $Z_p$.  Then to show Definition \ref{def:sink}(\lowerromannumeral{3}) holds for $\dz$, we need to show:
\begin{equation}\label{eq:514}
\tdyp(\xi r_-+(\xi-2)r_+)< -\tdyp (\widetilde{\mathrm{suc}}_-(\xi r_-+(\xi-2)r_+)).
\end{equation}
By Lemma \ref{lem:57}(\lowerromannumeral{2}), $\tdyp(\xi r_-+(\xi-2)r_+)=\xi -1$.  However, Lemma \ref{lem:56}(\lowerromannumeral{1}) gives:
\[\xi r_-+(\xi-2)r_+ < \ny -(p-\xi-1)r_--(p-\xi-1)r_+<\sucp(\xi r_- + (\xi -2)r_+)\leq K.\]
Then from Lemma \ref{lem:57}(\lowerromannumeral{3}):
\[-\tdyp(\widetilde{\mathrm{suc}}_-(\xi r_-+(\xi-2)r_+))\geq \mathrm{min} \{ p-\xi-1,p-\xi-1\}+1=p-\xi.\]
Then to show (\ref{eq:514}), we need only show $\xi -1 < p -\xi$, which is clear since $p=2\xi+1$.
\end{proof}   
\begin{lem}[cf. Lemma 5.9 of \cite{HLK}]\label{lem:59}
As abstract delta sequences $\dw \cong \Delta_{C_p}$ where $\Delta_{C_p}$ is as in Definition \ref{def:crea}.  
\end{lem}
\begin{proof}
We must explicitly compute $\dw$.  We begin by describing $\tilde{S}_{Y_p}\cap [K,\ny /2]$.  By (\ref{eq:nypk}), $K<\ny$, and by (\ref{eq:54}), $\ny/2 < (p-2)r_+$.  By Lemma \ref{lem:57}, we see $\tilde{S}_{Y_{p}} \cap[K,\ny/2]=S(r_-,r_+)\cap[K,\ny/2]$.  Then Lemma \ref{lem:54} gives:
\begin{multline}\label{eq:515}
\tilde{S}_{Y_p}\cap [K,\ny/2]=\{ (\xi-1)r_-+(\xi-1)r_+, (\xi-2)r_-+\xi r_+, \dots, r_-+(2\xi-3)r_+,\\(2\xi-2)r_+, (2\xi-1)r_-, (2\xi -2 )r_-+r_+,\dots, \xi r_-+(\xi-1)r_+ \}.
\end{multline}
To check that the last term of the sequence (\ref{eq:515}) is as written, we need to show 
\begin{equation}\label{eq:516}
\xi r_-+(\xi -1 )r_+ < \ny/2,
\end{equation}
and
\begin{equation}\label{eq:517}
(\xi-1)r_-+\xi r_+> \ny/2.
\end{equation}
To see (\ref{eq:516}), note that (\ref{eq:53}) gives $2 \xi r_-+(2\xi -2)r_+< \ny$, so $\xi r_-+(\xi -1)r_+ < \ny/2$.  

To see (\ref{eq:517}), note (\ref{eq:54}) gives $(2\xi-1)r_-+(2\xi -1 )r_+ > \ny$, so $(\xi -\frac{1}{2})r_-+(\xi -\frac{1}{2})r_+ > \ny/2$, and observe $(\xi-1)r_-+\xi r_+ > (\xi -\frac{1}{2})r_-+(\xi -\frac{1}{2}) r_+$.  Thus, (\ref{eq:515}) holds.  

We also find $\tilde{Q}_{Y_p} \cap [K,\ny/2]$, which is the same as to find $\tilde{S}_{Y_{p}} \cap [\ny/2, \ny -K]$.  

By (\ref{eq:54}) and (\ref{eq:nypk}),
\begin{equation}\label{eq:nyp2k}
\ny/2 < \ny -K < 2r_-+(p-3)r_+.
\end{equation}
By Lemma \ref{lem:57}(\lowerromannumeral{1}), $\tilde{S}_{Y_p}\cap [\ny/2,\ny-K]=S(r_-,r_+)\cap [\ny/2,\ny-K] \backslash \{ (2\xi-1)r_+\}$.  

Then, by Lemma \ref{lem:54}, 
\begin{multline}\label{eq:518}
\tilde{S}_{Y_p}\cap [\ny/2,\ny-K]=\{ (\xi-1)r_-+\xi r_+, (\xi-2)r_-+(\xi+1)r_+,\dots,r_-+(2\xi-2)r_+, \\ 2\xi r_-, (2\xi-1)r_-+r_+,\dots, (\xi +1)r_-+(\xi-1)r_+\}.  
\end{multline}
Note that $(2\xi-1)r_+$ is not present in (\ref{eq:518}).  To verify that $(\xi+1)r_-+(\xi -1)r_+$ is the last element in $\syp \cap [\ny/2,\ny-K]$, we must show 
\begin{align}\label{eq:519}
(\xi+1)r_-+(\xi-1)r_+ &< \ny-K, \; \mathrm{and}\; \\ \label{eq:520}
\xi r_-+\xi r_+ &> \ny-K.
\end{align}
Inequality (\ref{eq:519}) follows from (\ref{eq:53}) and the definition of $K$, while (\ref{eq:520}) follows from (\ref{eq:54}).  Thus (\ref{eq:518}) holds.  

We find the positions of elements of $\tilde{Q}_{Y_p}\cap [K,\ny/2]$ relative to the elements of $\syp \cap [K,\ny/2]$.  To do so, we use the following inequalities all obtained from (\ref{eq:53}) and (\ref{eq:54}).  
\begin{align}\label{eq:522}
(\xi -1-j)r_-+(\xi -1+j)r_+ &< \ny -(\xi+1+j)r_--(\xi-1-j)r_+, \; j=0,\dots, \xi-1,\\ \label{eq:523}
\ny-(\xi +1+j)r_--(\xi-1-j)r_+ &< (\xi -2-j)r_-+(\xi+j)r_+, \; j=0,\dots,\xi-2,\\ \label{eq:524}
jr_++(2\xi-1-j)r_- &< \ny-(j+1)r_--(2\xi-2-j)r_+, \; j=0,\dots,\xi-2, \\ \label{eq:525}
\ny-(j+1)r_--(2\xi-2-j)r_+ &< (j+1)r_++(2\xi-2-j)r_-, \; j=0,\dots,\xi-2.
\end{align}
We observe 
\begin{equation}\label{eq:526p1}
\ny-2\xi r_- < (2\xi-1)r_-
\end{equation}
directly from the definitions, and 
\begin{equation}\label{eq:526p2}
\ny-(\xi-1)r_--\xi r_+< \xi r_-+(\xi-1)r_+
\end{equation}
from (\ref{eq:54}).  

It follows from (\ref{eq:515}), (\ref{eq:518}), (\ref{eq:522})-(\ref{eq:525}), (\ref{eq:526p1}), and (\ref{eq:526p2}) that $\tx \cap [K,\ny/2]$ is:
\begin{multline}\label{eq:527}
\tx \cap [K,\ny/2]= \{ (\xi-1)r_-+(\xi-1)r_+, \ny -(\xi+1)r_--(\xi-1)r_+,(\xi-2)r_-+\xi r_+,\\ \ny - (\xi+2)r_--(\xi-2)r_+, \dots, r_-+(2\xi-3)r_+,\\ \ny -(2\xi-1)r_--r_+,(2\xi-2)r_+,\ny -2\xi r_-,(2\xi-1)r_-,\\ \ny-r_--(2\xi-2)r_+,(2\xi-2)r_-+r_+, \ny -2r_--(2\xi-3)r_+,\dots,(\xi+2)r_-+(\xi-3)r_+,\\ \ny-(\xi-2)r_--(\xi+1)r_+,(\xi+1)r_-+(\xi-2)r_+, \ny -(\xi-1)r_--\xi r_+,\xi r_-+(\xi-1)r_+\}.
\end{multline}
Now we need to calculate $\tdyp$ on $\tx \cap [K,\ny/2]$, and verify that it agrees with $\Delta_{C_p}$.  By Lemma \ref{lem:57}(\lowerromannumeral{2}) and $\ny/2<(p-2)r_+$,
\begin{equation}\label{eq:528}
\tdyp (cr_-+dr_+)=\mathrm{min} \{c,d\}+1 \; \mathrm{for} \; cr_-+dr_+ \in \syp \cap [K,\ny/2].
\end{equation}
Similarly, for $\ny-cr_--dr_+ \in \tilde{Q}_{Y_p}\cap [K,\ny/2]$ such that $cr_-+dr_+ \neq 2\xi r_-$:
\begin{equation}\label{eq:529}\tdyp(\ny-cr_--dr_+)=-\tdyp(cr_-+dr_+)=-\mathrm{min}\{c,d\}-1 \end{equation}
by Lemma \ref{lem:57}(\lowerromannumeral{2}), using (\ref{eq:nyp2k}) to obtain $cr_-+dr_+ < \ny -K < 2r_-+(p-3)r_+$.  Also, Lemma \ref{lem:57} gives 
\begin{equation}\label{eq:530} -2= -\tdyp (2 \xi r_-)=\tdyp(\ny-2\xi  r_-). \end{equation}
Computing $\tdyp$ using (\ref{eq:528}),(\ref{eq:529}), and (\ref{eq:530}), we see that $\Delta_{W_p}$ agrees with $\Delta_{C_p}$ from Definition \ref{def:crea}.  This completes the proof of Lemma \ref{lem:main}.
\end{proof} 

\begin{fact}[\cite{betaseifert}]\label{fct:proje} 
Let $Y=\Sigma(b,(b_1,a_1),\dots,(b_k,a_k))$ be a negative Seifert rational homology sphere with spin structure $\s$.  Let $(\Gamma_{Y},\chi)$ be the (symmetric) graded root associated to $(-Y,\mathfrak{s})$, and let $\iota$ be the associated involution of $\Gamma_{Y}$.  Let $v \in \Gamma_Y$ be the vertex of minimal grading which is invariant under $\iota$.  The space $(Y,\s)$ is of projective type if and only if there exists a vertex $w$, and a path from $v$ to $w$ in $\Gamma_Y$ which is grading-decreasing at each step, with $\chi(w)=\mathrm{min}_{x\in \Gamma_Y} \chi(x)$.  Moreover, $\delta(Y,\s)-\beta(Y,\s)=\chi(v)-\chi(w)$.  
\end{fact}

\begin{thm}\label{thm:maingrad} 
The Seifert spaces $Y_p=\Sigma(p,2p-1,2p+1)$, for $p$ odd, are of projective type, with $d(Y_p)=p-1$ and $\beta(Y_p)=0$.
\end{thm}
\begin{proof}
By Remark 3.3 of \cite{HLK}, $d(Y_p)=p-1$, so we need only show that $Y_p$ is of projective type, and that $\beta(Y_p)=0$.

Let $\Gamma_{Y_p}$ have its grading shifted so that it agrees with the grading of $\mathit{HF}^+(-Y_p)$ (using Theorem \ref{thm:grdroo}).  The decomposition in Lemma \ref{lem:main} implies that $\gcp$ embeds into $\Gamma_{Y_p}$ as a subgraph.  Since $d(-Y_p)=1-p$, we see that the embedding of $\gcp$ is degree-preserving.  Since $\Delta_{Z_p}$ is sinking, by Proposition \ref{prop:47} the minimal value of $\tau_{Z_p}$ is $0$.  Thus \[\mathbb{H}_{\leq 0}(\Gamma_{C_p})=\mathbb{H}_{\leq 0}(\Gamma_{Y_p}).\]  By Fact \ref{fct:proje} applied to the graded root $\Gamma_{\Delta_{Y_p}}$ (see Figure \ref{fig:creature}), we have that $Y_p$ is of projective type.  It is clear from Figure \ref{fig:creature} that the vertex of minimal grading which is invariant under $\iota$ is in degree $0$, from which we obtain $\beta(Y_p)=0$.  
\end{proof} 

\bibliography{15325bocs.bib}
\bibliographystyle{plain}

\end{document}